\documentclass[11pt]{article}
\usepackage{array}
\usepackage{amsfonts}
\usepackage{amscd}
\usepackage{amssymb}
\usepackage{amsthm}
\usepackage{amsmath}
\usepackage{stmaryrd}
\usepackage{graphicx}
\usepackage{color}
\usepackage{verbatim}
\usepackage{mathrsfs}
\usepackage{tikz}
\usetikzlibrary{calc}
\usepackage{calrsfs}
\usepackage{caption}
\usepackage[neveradjust]{paralist}

\parskip\medskipamount

\usepackage{hyperref}
\usepackage[capitalize]{cleveref}

 \theoremstyle{plain}
 \newtheorem{thm1}{Theorem}
 
\newtheorem{thm}{Theorem}[section]
\newtheorem{theo}{Theorem}[section]
\newtheorem{lemma}[thm]{Lemma}
\newtheorem{prop}[thm]{Proposition}
\newtheorem{cor}[thm]{Corollary}

\newtheorem{fact}[thm]{Fact}
\theoremstyle{definition}
\newtheorem{defn}[thm]{Definition}
\newtheorem{remark}[thm]{Remark}
\newtheorem{example}[thm]{Example}

\numberwithin{equation}{section}

\def\sA{\mathsf{A}}
\def\sB{\mathsf{B}}
\def\sC{\mathsf{C}}
\def\sD{\mathsf{D}}
\def\sE{\mathsf{E}}
\def\sF{\mathsf{F}}
\def\sG{\mathsf{G}}

\def\sX{\mathsf{X}}

\def\su{\mathsf{u}}

\def\cA{\mathcal{A}}

\def\cE{\mathcal{E}}
\def\cF{\mathcal{F}}
\def\cO{\mathcal{O}}
\def\cH{\mathcal{H}}
\def\cI{\mathcal{I}}

\def\cL{\mathcal{L}}

\def\cO{\mathcal{O}}

\def\cV{\mathcal{V}}

\def\wE{\widetilde{E}}

\def\AA{\mathbb{A}}

\def\FF{\mathbb{F}}

\def\HH{\mathbb{H}}

\def\KK{\mathbb{K}}
\def\LL{\mathbb{L}}

\def\OO{\mathbb{O}}

\def\ZZ{\mathbb{Z}}

\def\K{\mathbb{K}}
\def\PG{\mathsf{PG}}
\def\PSL{\mathsf{PSL}}
\def\uniclass{uniclass\ }
\DeclareMathOperator\Aut{\mathsf{Aut}}
\DeclareMathOperator\kar{\mathsf{char}}

\DeclareMathOperator\diam{\mathsf{diam}}
\DeclareMathOperator\Res{\mathsf{Res}}

\DeclareMathOperator\OOpp{\mathsf{Opp}}
\DeclareMathOperator\Diag{\mathsf{Diag}}
\DeclareMathOperator\disp{\mathsf{Disp}}

\newcommand{\pperp}{\perp\hspace{-0.15cm}\perp}

\def\<{\langle}
\def\>{\rangle}

\makeatletter
\renewcommand{\@makefnmark}{\mbox{\textsuperscript{}}}
\makeatother

\title{Automorphisms and opposition in spherical buildings of exceptional type, IV: The $\sE_7$ case}
\author{Yannick Neyt \and James Parkinson 
\and
Hendrik Van Maldeghem \and Magali Victoor}
\date{\today}

\begin{document}

\maketitle

\begin{abstract} An automorphism of a spherical building is called \textit{domestic} if it maps no chamber onto an opposite chamber. This paper forms a significant part of a large project classifying domestic automorphisms of spherical buildings of exceptional type. In previous work the classifications for $\mathsf{G}_2$, $\mathsf{F}_4$ and $\mathsf{E}_6$ have been completed, and the present work provides the classification for buildings of type~$\mathsf{E}_7$. In many respects this case is the richest amongst all exceptional types.
%
%
%
\end{abstract}

\tableofcontents

\section{Introduction}

The study of the geometry of fixed elements of automorphisms of spherical buildings is a well-established and beautiful topic (see~\cite{PMW:15}). Over the past decade a complementary theory concerning the ``opposite geometry'', consisting of those elements mapped to opposite elements by an automorphism of a spherical building, has been developed (see~\cite{DPVM,Lam-Mal:24,PVM,PVMsmall,PVMexc,PVMclass,TTVM3,TTVM2,Mal:12}). A starting point for this theory is the fundamental result of Abramenko and Brown~\cite[Proposition 4.2]{AB:09}, stating that if $\theta$ is a nontrivial automorphism of a thick spherical building then the opposite geometry $\OOpp(\theta)$ is necessarily nonempty. Indeed the generic situation is that $\OOpp(\theta)$ is rather large, and typically contains many chambers of the building (\textit{chambers} are the simplices of maximal dimension). The more special situation is when $\OOpp(\theta)$ contains no chamber, in which case $\theta$ is called \textit{domestic}. 

In \cite{PVMexc} the second and third authors initiated an intensive investigation of domestic automorphisms of buildings of exceptional type, with an overarching objective being to obtain a complete classification of such automorphisms. In \cite{PVMexc} we were able to classify all domestic automorphisms of thick buildings of type $\mathsf{E_6}$, and split buildings of types $\sF_4$ and $\sG_2$, and in \cite{Lam-Mal:24} Lambrecht and the third author obtained the classification for all spherical buildings of type~$\sF_4$. Moreover, partial results were obtained in \cite{PVMexc} for buildings of types $\sE_7$ and $\sE_8$, including the classification of domestic homologies of these buildings (a homology is a collineation pointwise fixing an apartment and a panel; it fixes a full weak subbuilding), and we also exhibited examples of unipotent domestic automorphisms.

In the present paper we provide the classification of domestic automorphisms of~$\sE_7$ buildings. Both the statement of the classification, and its proof, are considerably more intricate and involved than the $\sE_6$ and (split) $\sF_4$ and $\sG_2$ cases, and indeed we discover new and beautiful behaviour not present in these  lower rank cases. The classification for $\sE_8$ will be dealt with in future work.

Before stating our classification theorems we first recall some preliminary notation and definitions. In the sequel, we will use the Bourbaki labelling (see \cite{Bou:02}) of the nodes in the $\mathsf{E_7}$ diagram:
$$
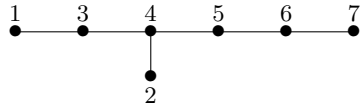
\begin{tikzpicture}[scale=0.3]
\node at (0,0.3) {};
\node [inner sep=0.8pt,outer sep=0.8pt] at (-6,0) (1) {$\bullet$};
\node [inner sep=0.8pt,outer sep=0.8pt] at (-6,0.8) (1) {\footnotesize 1};

\node [inner sep=0.8pt,outer sep=0.8pt] at (-3,0) (3) {$\bullet$};
\node [inner sep=0.8pt,outer sep=0.8pt] at (-3,0.8) (3) {\footnotesize 3};

\node [inner sep=0.8pt,outer sep=0.8pt] at (0,0) (4) {$\bullet$};
\node [inner sep=0.8pt,outer sep=0.8pt] at (0,0.8) (4) {\footnotesize 4};

\node [inner sep=0.8pt,outer sep=0.8pt] at (3,0) (5) {$\bullet$};
\node [inner sep=0.8pt,outer sep=0.8pt] at (3,0.8) (5) {\footnotesize 5};

\node [inner sep=0.8pt,outer sep=0.8pt] at (6,0) (6) {$\bullet$};
\node [inner sep=0.8pt,outer sep=0.8pt] at (6,0.8) (6) {\footnotesize 6};

\node [inner sep=0.8pt,outer sep=0.8pt] at (9,0) (7) {$\bullet$};
\node [inner sep=0.8pt,outer sep=0.8pt] at (9,0.8) (7) {\footnotesize 7};

\node [inner sep=0.8pt,outer sep=0.8pt] at (0,-2) (2) {$\bullet$};
\node [inner sep=0.8pt,outer sep=0.8pt] at (0,-2.8) (2) {\footnotesize 2};

\phantom{\draw [line width=0.5pt,line cap=round,rounded corners] (1.north west)  rectangle (1.south east);}
\phantom{\draw [line width=0.5pt,line cap=round,rounded corners] (6.north west)  rectangle (6.south east);}
\draw (-6,0)--(9,0);
\draw (0,0)--(0,-2);
\end{tikzpicture}$$

An irreducible thick spherical building of rank at least $3$ is called \textit{large} if it contains no Fano plane residues. By Tits' classification of spherical buildings~\cite{Tits:74}, a large building of type $\sE_7$ is any building $\sE_7(\KK)$ over a field $\KK$ with at least $3$ elements. By the main result of \cite{PVM}, each automorphism $\theta$ of a large spherical building is \textit{capped}, meaning that it satisfies the following property: If $\theta$ maps a flag of type $J_1$ to an opposite flag, and another flag of type $J_2$ to an opposite flag, then $\theta$ maps a flag of type $J_1\cup J_2$ to an opposite flag. Hence in a large spherical building, in order to know the types of all flags mapped to an opposite it suffices to know the types of the minimal ones. Since these minimal types are orbits of the induced action of the automorphism on the Dynkin (or Coxeter) diagram, and since that action is always trivial in the case of $\mathsf{E_7}$, it suffices to know the types of the vertices mapped to an opposite. Encircling those types on the Coxeter diagram gives the \emph{opposition diagram} of the automorphism. In \cite{PVM}, all possible opposition diagrams are classified, and the list for $\mathsf{E_7}$ is given in Figure~\ref{fig:Dynkin}. 

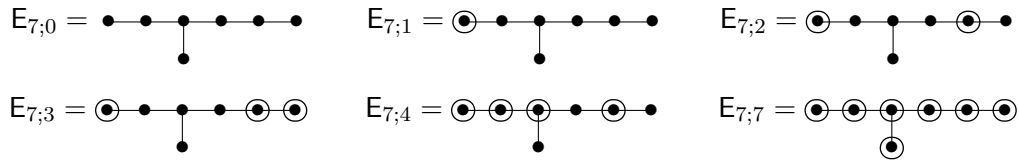
\begin{figure}[h!]
\begin{center}
$\sE_{7;0}=\begin{tikzpicture}[scale=0.5,baseline=-0.5ex]
\node at (0,0.8) {};
\node at (0,-0.8) {};
\node [inner sep=0.8pt,outer sep=0.8pt] at (-2,0) (1) {$\bullet$};
\node [inner sep=0.8pt,outer sep=0.8pt] at (-1,0) (3) {$\bullet$};
\node [inner sep=0.8pt,outer sep=0.8pt] at (0,0) (4) {$\bullet$};
\node [inner sep=0.8pt,outer sep=0.8pt] at (1,0) (5) {$\bullet$};
\node [inner sep=0.8pt,outer sep=0.8pt] at (2,0) (6) {$\bullet$};
\node [inner sep=0.8pt,outer sep=0.8pt] at (3,0) (7) {$\bullet$};
\node [inner sep=0.8pt,outer sep=0.8pt] at (0,-1) (2) {$\bullet$};
\draw (-2,0)--(3,0);
\draw (0,0)--(0,-1);
\phantom{\draw [line width=0.5pt,line cap=round,rounded corners] (1.north west)  rectangle (1.south east);}
\phantom{\draw [line width=0.5pt,line cap=round,rounded corners] (7.north west)  rectangle (7.south east);}
\end{tikzpicture}$\qquad 
$\sE_{7;1}=\begin{tikzpicture}[scale=0.5,baseline=-0.5ex]
\node at (0,0.8) {};
\node at (0,-0.8) {};
\node [inner sep=0.8pt,outer sep=0.8pt] at (-2,0) (1) {$\bullet$};
\node [inner sep=0.8pt,outer sep=0.8pt] at (-1,0) (3) {$\bullet$};
\node [inner sep=0.8pt,outer sep=0.8pt] at (0,0) (4) {$\bullet$};
\node [inner sep=0.8pt,outer sep=0.8pt] at (1,0) (5) {$\bullet$};
\node [inner sep=0.8pt,outer sep=0.8pt] at (2,0) (6) {$\bullet$};
\node [inner sep=0.8pt,outer sep=0.8pt] at (3,0) (7) {$\bullet$};
\node [inner sep=0.8pt,outer sep=0.8pt] at (0,-1) (2) {$\bullet$};
\draw (-2,0)--(3,0);
\draw (0,0)--(0,-1);
\draw [line width=0.5pt,line cap=round,rounded corners] (1.north west)  rectangle (1.south east);
\end{tikzpicture}$\qquad
$\sE_{7;2}=\begin{tikzpicture}[scale=0.5,baseline=-0.5ex]
\node at (0,0.8) {};
\node at (0,-0.8) {};
\node [inner sep=0.8pt,outer sep=0.8pt] at (-2,0) (1) {$\bullet$};
\node [inner sep=0.8pt,outer sep=0.8pt] at (-1,0) (3) {$\bullet$};
\node [inner sep=0.8pt,outer sep=0.8pt] at (0,0) (4) {$\bullet$};
\node [inner sep=0.8pt,outer sep=0.8pt] at (1,0) (5) {$\bullet$};
\node [inner sep=0.8pt,outer sep=0.8pt] at (2,0) (6) {$\bullet$};
\node [inner sep=0.8pt,outer sep=0.8pt] at (3,0) (7) {$\bullet$};
\node [inner sep=0.8pt,outer sep=0.8pt] at (0,-1) (2) {$\bullet$};
\draw (-2,0)--(3,0);
\draw (0,0)--(0,-1);
\draw [line width=0.5pt,line cap=round,rounded corners] (1.north west)  rectangle (1.south east);
\draw [line width=0.5pt,line cap=round,rounded corners] (6.north west)  rectangle (6.south east);
\end{tikzpicture}$\\
$\sE_{7;3}=\begin{tikzpicture}[scale=0.5,baseline=-0.5ex]
\node at (0,0.8) {};
\node at (0,-0.8) {};
\node [inner sep=0.8pt,outer sep=0.8pt] at (-2,0) (1) {$\bullet$};
\node [inner sep=0.8pt,outer sep=0.8pt] at (-1,0) (3) {$\bullet$};
\node [inner sep=0.8pt,outer sep=0.8pt] at (0,0) (4) {$\bullet$};
\node [inner sep=0.8pt,outer sep=0.8pt] at (1,0) (5) {$\bullet$};
\node [inner sep=0.8pt,outer sep=0.8pt] at (2,0) (6) {$\bullet$};
\node [inner sep=0.8pt,outer sep=0.8pt] at (3,0) (7) {$\bullet$};
\node [inner sep=0.8pt,outer sep=0.8pt] at (0,-1) (2) {$\bullet$};
\draw (-2,0)--(3,0);
\draw (0,0)--(0,-1);
\draw [line width=0.5pt,line cap=round,rounded corners] (1.north west)  rectangle (1.south east);
\draw [line width=0.5pt,line cap=round,rounded corners] (6.north west)  rectangle (6.south east);
\draw [line width=0.5pt,line cap=round,rounded corners] (7.north west)  rectangle (7.south east);
\end{tikzpicture}$\qquad 
$\sE_{7;4}=\begin{tikzpicture}[scale=0.5,baseline=-0.5ex]
\node at (0,0.8) {};
\node at (0,-0.8) {};
\node [inner sep=0.8pt,outer sep=0.8pt] at (-2,0) (1) {$\bullet$};
\node [inner sep=0.8pt,outer sep=0.8pt] at (-1,0) (3) {$\bullet$};
\node [inner sep=0.8pt,outer sep=0.8pt] at (0,0) (4) {$\bullet$};
\node [inner sep=0.8pt,outer sep=0.8pt] at (1,0) (5) {$\bullet$};
\node [inner sep=0.8pt,outer sep=0.8pt] at (2,0) (6) {$\bullet$};
\node [inner sep=0.8pt,outer sep=0.8pt] at (3,0) (7) {$\bullet$};
\node [inner sep=0.8pt,outer sep=0.8pt] at (0,-1) (2) {$\bullet$};
\draw (-2,0)--(3,0);
\draw (0,0)--(0,-1);
\draw [line width=0.5pt,line cap=round,rounded corners] (1.north west)  rectangle (1.south east);
\draw [line width=0.5pt,line cap=round,rounded corners] (3.north west)  rectangle (3.south east);
\draw [line width=0.5pt,line cap=round,rounded corners] (4.north west)  rectangle (4.south east);
\draw [line width=0.5pt,line cap=round,rounded corners] (6.north west)  rectangle (6.south east);
\end{tikzpicture}$\qquad
$\sE_{7;7}=\begin{tikzpicture}[scale=0.5,baseline=-0.5ex]
\node at (0,0.8) {};
\node at (0,-0.8) {};
\node [inner sep=0.8pt,outer sep=0.8pt] at (-2,0) (1) {$\bullet$};
\node [inner sep=0.8pt,outer sep=0.8pt] at (-1,0) (3) {$\bullet$};
\node [inner sep=0.8pt,outer sep=0.8pt] at (0,0) (4) {$\bullet$};
\node [inner sep=0.8pt,outer sep=0.8pt] at (1,0) (5) {$\bullet$};
\node [inner sep=0.8pt,outer sep=0.8pt] at (2,0) (6) {$\bullet$};
\node [inner sep=0.8pt,outer sep=0.8pt] at (3,0) (7) {$\bullet$};
\node [inner sep=0.8pt,outer sep=0.8pt] at (0,-1) (2) {$\bullet$};
\draw (-2,0)--(3,0);
\draw (0,0)--(0,-1);
\draw [line width=0.5pt,line cap=round,rounded corners] (1.north west)  rectangle (1.south east);
\draw [line width=0.5pt,line cap=round,rounded corners] (3.north west)  rectangle (3.south east);
\draw [line width=0.5pt,line cap=round,rounded corners] (4.north west)  rectangle (4.south east);
\draw [line width=0.5pt,line cap=round,rounded corners] (6.north west)  rectangle (6.south east);
\draw [line width=0.5pt,line cap=round,rounded corners] (2.north west)  rectangle (2.south east);
\draw [line width=0.5pt,line cap=round,rounded corners] (5.north west)  rectangle (5.south east);
\draw [line width=0.5pt,line cap=round,rounded corners] (7.north west)  rectangle (7.south east);
\end{tikzpicture}$
\end{center}
\caption{The opposition diagrams of type $\mathsf{E_7}$}\label{fig:Dynkin}
\end{figure}

On the other hand, the \textit{fixed element diagram} of an automorphism of an $\sE_7$ building is given by encircling the types of the vertices of the building fixed by~$\theta$. We will use the same symbols $\sE_{7;j}$ for fixed diagrams.

The opposition diagram $\sE_{7;7}$ is the opposition diagram of any non-domestic automorphism, and hence we shall not be concerned with it. Also, since the opposite geometry of a nontrivial automorphism is never empty (by \cite{AB:09}), the only automorphism with opposition diagram $\sE_{7;0}$ is the identity automorphism. Moreover, in \cite[Theorems~1 and 4]{PVMexc} we proved that each automorphism with opposition diagram $\sE_{7;1}$ is a nontrivial central collineation (and vice-versa), and that each automorphism with opposition diagram $\sE_{7;2}$ is the product of two nontrivial perpendicular root elations (and vice-versa). Thus the focus of the present paper is to classify the automorphisms with opposition diagrams $\sE_{7;3}$ and $\sE_{7;4}$. 

In \cite[Theorems~5 and~6]{PVMexc} we provided examples of automorphisms with opposition diagrams $\sE_{7;3}$ and $\sE_{7;4}$. These examples were certain products of $3$ or $4$ perpendicular nontrivial central collineations, respectively, or certain homologies. In particular, all of these examples fix a chamber of the building (equivalently, they are conjugate to members of the Borel subgroup $B$). Moreover, in the case of $\mathsf{E_6}$, it is shown in \cite[Theorem 8]{PVMexc} that all domestic automorphisms of a thick $\sE_6$ building fix a chamber of the building.

In contrast, in the present paper we shall see that in $\sE_7$ buildings there exist automorphisms with opposition diagrams $\sE_{7;3}$ and $\sE_{7;4}$ fixing no chamber, provided the underlying field admits certain extensions. The complete classification of such automorphisms is given in the following theorem (see Remark~\ref{E7im} for the definition of equator geometries for part ($iii$)). 
%
%

\begin{thm1}\label{thm:nonchamberfixing} Let $\Delta=\sE_7(\K)$ with $|\K|>2$. 
\begin{compactenum}[$(i)$] \item For each quadratic extension $\LL$ of $\K$, there exists, up to conjugacy, a unique subgroup $H$ of the automorphism group of $\Delta$ each nontrivial element of which is a domestic automorphism with opposition diagram  $\sE_{7;3}$. 
Moreover, as an abstract group, $H$ is isomorphic to $\LL^\times/\K^\times$, the fixed diagram of each nontrivial element of $H$ is $\sE_{7;4}$, and the fixed structure is a building of type $\sF_4$.
\item For each quaternion division algebra $\HH$ over $\K$, there exists, up to conjugacy, a unique subgroup $H$ of the automorphism group of $\Delta$ each nontrivial element of which is a domestic automorphism with opposition diagram $\sE_{7;4}$.  As an abstract group, $H$ is isomorphic to $\HH^\times/\K^\times$, the fixed diagram of each nontrivial element of $H$ is $\sE_{7;3}$, and the fixed structure is a building of type $\sC_3$. 
\item For each quadratic extension $\LL$ of $\K$, every member of the pointwise stabiliser of a subbuilding of type $\mathsf{D}_{6}$ obtained as equator geometry $\Gamma=E(p_1,p_2)$ of the parapolar space $\sE_{7,1}(\K)$ acting without fixed points on the imaginary line  defined by $p_1,p_2$, is a domestic automorphism with opposition diagram $\sE_{7;4}$. 
\end{compactenum}
Conversely, every domestic automorphism of $\Delta$ fixing no chamber is conjugate to some collineation as in $(i)$, $(ii)$ or $(iii)$ above.  
\end{thm1}

In Theorem~\ref{thm:nonchamberfixing}, quadratic extensions of $\K$ are not assumed to be separable; likewise, with quaternion algebra over $\K$ we mean a $4$-dimensional associative quadratic division algebra over $\K$; hence it can also be an inseparable field extension of degree 4 of a field in characteristic 2, in which case we refer to it as an \emph{inseparable quaternion division algebra over $\K$}. 

The automorphisms listed in Theorem~\ref{thm:nonchamberfixing} parts~($i$) and~($ii$) exhibit a particularly attractive duality in the sense that the automorphisms with opposition diagram $\sE_{7;3}$ have fixed diagram $\sE_{7;4}$, and vice-versa. This phenomenon is not random, but fits in the context of the \emph{Freudenthal-Tits Magic Square}, where this behaviour is systematic for cells lying symmetric with respect to the main diagonal. This, and more background, is explained in \cite{Mal:23}, see in particular Section 9.4 therein. 

To complete the classification of domestic automorphisms of large $\sE_7$ buildings we must classify the domestic automorphisms that fix a chamber (equivalently, that are conjugate to a member of the Borel subgroup~$B$). Let $\Phi$ denote the root system of type $\sE_7$ with simple roots $\alpha_1,\ldots,\alpha_7$ (in Bourbaki labelling). For $\alpha\in\Phi$ and $a\in\K$ let $x_{\alpha}(a)$ denote the standard Chevalley generator and for $\lambda$ in the coweight lattice of $\Phi$ and $c\in\K\backslash\{0\}$ let $h_{\lambda}(c)$ denote the standard torus elements (see Section~\ref{sec:chamberfixing} for definitions). 

We first classify the chamber fixing automorphisms with opposition diagram~$\sE_{7;3}$. 

\begin{thm1}\label{thm:E73Classification1}
An automorphism $\theta$ of the building $\Delta=\sE_7(\K)$ with $|\K|>2$ has opposition diagram $\sE_{7;3}$ and fixes a chamber if and only if it is conjugate to
\begin{compactenum}[$(i)$]
\item the unipotent element $x_{\alpha_2}(1)x_{\alpha_5}(1)x_{\alpha_7}(1)$;
\item an homology whose fixed structure is a weak building with thick frame of type $\sE_6$; such elements are conjugate to $h_{\omega_7}(c)$ with $c\in\K\backslash\{0,1\}$.
\end{compactenum}
\end{thm1}

Finally, we classify the chamber fixing automorphisms with opposition diagram $\sE_{7;4}$. Let $\Phi_{\sD_4}$ be the sub-root system of $\Phi$ of type $\sD_4$ (generated by $\alpha_2,\alpha_3,\alpha_4,\alpha_5$). If $\alpha=i\alpha_2+j\alpha_4+k\alpha_3+l\alpha_5\in\Phi_{\sD_4}$ we write $x_{\alpha}(a)=x_{ijkl}(a)$ for the associated Chevalley generators (with $a\in\K$). 

\goodbreak

\begin{thm1}\label{thm:E74chamberfixingmain1}
An automorphism $\theta$ of the building $\Delta=\sE_7(\K)$ with $|\K|>2$ has opposition diagram $\sE_{7;4}$ and fixes a chamber if and only if it is conjugate to
\begin{compactenum}[$(i)$]
\item $x_{0100}(a)x_{1110}(1)x_{1101}(1)x_{0111}(1)$ with $a\in\K\backslash\{0\}$; 
\item $x_{1111}(1)x_{0100}(a)x_{1110}(1)x_{1101}(1)x_{0111}(1)$ with $a\in\K\backslash\{0\}$; 
\item an homology whose fixed structure is a weak building with thick frame of type $\sD_6$ or $\sD_6\times \sA_1$; such elements are conjugate to $h_{\alpha^{\vee}}(c)$ with $\alpha\in\Phi$ and $c\in\K\backslash\{0,1\}$;
\item $x_{\alpha}(1)h_{\alpha^{\vee}}(-1)$ with $\kar(\K)\neq 2$, for any root $\alpha\in\Phi$.
\end{compactenum}
\end{thm1}

The automorphisms appearing in Theorem~\ref{thm:E73Classification1} were already known in~\cite{PVMexc} (however, of course, the work here is proving that the classification is complete). The automorphisms appearing in Theorem~\ref{thm:E74chamberfixingmain1} parts ($i$) and ($iii$) were known from~\cite{PVMexc}, however the automorphisms appearing in ($ii$) and ($iv$) are new, and arose in the course of the classification. In particular, the automorphism in case~($iv$) is interesting in that it is conjugate to neither a unipotent element nor a homology (such elements were never domestic in the lower rank $\sE_6$, $\sF_4$, and $\sG_2$ cases; see Proposition~\ref{prop:extraauto}).

Theorems~\ref{thm:nonchamberfixing}, \ref{thm:E73Classification1}, and~\ref{thm:E74chamberfixingmain1}, combined with \cite[Theorems~1 and 4]{PVMexc}, gives the complete classification of domestic automorphisms of large buildings of type~$\sE_7$. The unique small building of type $\sE_7$ (that is, the building $\Delta=\sE_7(\FF_2)$) behaves differently due to the existence of uncapped automorphisms (see \cite{PVMsmall}), and the full classification of domestic automorphisms of this building is currently unknown, and will be addressed in future work. Thus, in the present paper, we will assume $|\KK|>2$ throughout.

Before outlining the structure of this paper, we outline some additional motivation for studying domestic automorphisms by illustrating that domesticity is intimately connected to beautiful geometric and algebraic phenomena. Firstly there is a remarkable connection with the Freudenthal-Tits Magic Square, already mentioned above (see \cite[Section~9.4]{Mal:23}). The case of $\sE_7$ is particularly interesting here since it is the only case appearing twice in the Magic Square, both times neither split nor quasisplit, resulting in two diagrams leaving nodes uncircled (the $\sE_{7;3}$ and $\sE_{7;4}$ diagrams). This allows for a rich domestic behaviour of buildings of type $\sE_7$, and the nontrivial duality between opposition and fixed diagrams seen in Theorem~\ref{thm:nonchamberfixing}. 

Secondly there is a connection with \emph{linear descent groups}, that is, groups pointwise fixing a subbuilding. \textit{Galois descent} (see~\cite{PMW:15}) provides one source of such groups; but when the companion field automorphisms are trivial we speak instead of \emph{linear descent}. Contrary to the Galois descent case (where the groups are small), not all linear descent groups (which are much larger) are classified, even when the fixed building is known.   Linear descent provides a way to see natural and rather large inclusions of buildings, explaining some of the structure of both the ambient and embedded building. For example in the present paper, in order to understand domesticity in building of type $\sE_7$ we are led to determine the descent group in buildings of type $\sE_6$ of the (fixed) quaternion projective planes (the octonion counterpart coming from Galois descent). The fix group  will turn out to be abstractly isomorphic to the multiplicative group of norm 1 elements of the corresponding quaternion algebra, see Subsection~\ref{Qp}. Along the way, we will show a certain rigidity result for such quaternion projective planes embedded in buildings of type~$\sE_6$:  they are determined up to a unique twin by any anti-flag, that is, a non-incident point-line pair (see Subsection~\ref{fixQP}).

We conclude this introduction with a brief outline of the structure and strategy of the paper. The analysis naturally divides into two cases, depending on whether the domestic automorphism fixes a chamber of not. In the case that the automorphism does not fix a chamber (Theorem~\ref{thm:nonchamberfixing}), our arguments are of a geometric flavour, and the necessary background material on polar and parapolar spaces is given in Section~\ref{sec:parapolar}. In particular, Section~\ref{paras} gives precise and specific information required on various parapolar spaces of exceptional type that are required for this paper. To prove Theorem~\ref{thm:nonchamberfixing} we must first introduce and classify collineations in buildings of types $\sD_n$ and $\sE_6$ whose displacement spectra skips certain values (we call these \textit{kangaroo collineations}). This analysis is given in Sections~\ref{sec:kang}, \ref{sec:kang2}, and~\ref{existenceE74E6}. The classification of kangaroo collineations is then applied in Sections~\ref{sec:E74nochamber} and~\ref{sec:E73nochamber} to prove Theorem~\ref{thm:nonchamberfixing}. 

It then remains to classify domestic automorphisms fixing a chamber. Here our arguments are more algebraic, relying on commutator relations and ultimately reducing to specific calculations in the~$\sD_4$ Chevalley group. This analysis is given in Section~\ref{sec:chamberfixing}. We setup the preliminary arguments of Section~\ref{sec:chamberfixing} in sufficient generality so that they can be applied in future work on~$\sE_8$.

\section{Preliminaries on polar and parapolar spaces}\label{sec:parapolar}
Our approach to buildings of type $\mathsf{E_7}$ is via their related geometries known as parapolar spaces. Everything we say below can be found  in the standard books \cite{BuekCoh} and \cite{Shult}. Although we assume a certain familiarity with this theory, we recall some basic definitions in order to settle notation. We shall use the words ``automorphism'' and ``collineation'' interchangeably (with the former more connected to the building language, and the latter more connected to the incidence geometry language).

\subsection{Abstract definitions}
Let $\Gamma=(X,\cL)$ be a point-line geometry (if the incidence relation is not mentioned, we assume it is induced by containment). 
Points $x,y\in X$ contained in a common line are called \emph{collinear}, denoted as $x\perp y$;  the set of all points collinear to $x$ is denoted by $x^\perp$, the \emph{perp} of $x$. We will always deal with situations where every point is contained in at least one line, so $x\in x^\perp$. Also, for $S\subseteq X$, we denote $S^\perp:=\{x\in X\mid x\perp s\mbox{ for all }s\in S\}$. Likewise we write $S\perp T$ for subsets $S,T\subseteq X$ if $s\perp t$ for each $s\in S$ and each $ t\in T$. If each line has at least three points, we call the geometry \emph{thick}. 

The \emph{point graph} of $\Gamma$ is the graph on $X$ with collinearity as adjacency relation. The \emph{distance $\delta$} between two points $p,q \in X$ (denoted $\delta_\Gamma(p,q)$, or $\delta(p,q)$ if no confusion is possible) is the distance between $p$ and $q$ in the collinearity graph, where $\delta(p,q)=\infty$ if $p$ and $q$ are contained in distinct connected components of the point graph; If $\delta:=\delta(p,q)$ is finite, then a \emph{geodesic path} or a \emph{shortest path} between $p$ and $q$ is a path between them in the point graph of length $\delta$.  The \emph{diameter} of  $\Gamma$ (denoted $\diam \Gamma$) is the diameter of the point graph. We say that $\Gamma$ is \emph{connected} if every pair of vertices is at finite distance from one another. The point-line geometry $\Gamma$ is called a \emph{partial linear space} if each pair of distinct points is contained in at most one line. In this case we usually denote the unique line containing two distinct collinear points $x$ and $y$ by $xy$. 

A \emph{subspace} of $\Gamma$ is a subset $A$ of $X$ such that, if $x,y\in A$ are collinear and distinct, then all lines containing both $x$ and $y$ are contained in $A$. A subspace $A$ is called \emph{convex} if, for any pair of points $\{p,q\}\subseteq A$, every point occurring in a geodesic between $p$ and $q$ is contained in $A$; it is \emph{singular} if $\delta(p,q)\leq 1$ for all $p,q\in A$. The intersection of all convex subspaces of $\Gamma$ containing a given subset $B\subseteq X$ is called the \emph{convex subspace closure} of $B$.  A proper subspace $H$ is called a \emph{geometric hyperplane} if each line of $\Gamma$ has either one or all its points contained in~$H$.  

A \emph{full subgeometry} $\Gamma'=(X',\cL')$ of $\Gamma$ is a geometry with $X'\subseteq X$ and $\cL'\subseteq\cL$. This implies that all points of $\Gamma$ on a line of $\Gamma'$ are points of $\Gamma'$ and explains the adjective `full'. Full subgeometries need not be subspaces.

Now a \emph{polar space} is a thick point-line geometry in which the perp of every point is a geometric hyperplane. This forces all singular subspaces to be projective spaces. In our case the polar spaces will have finite rank, that is, there is a natural number $r\geq 2$ such all singular subspaces (which are projective spaces) have dimension $\leq r-1$, and there exist singular subspaces of dimension $r-1$. A prominent notion in polar geometry is \emph{opposition}. Two singular subspaces $U,W$ are \emph{opposite} if no point of $U\cup W$ is collinear to all points of $U\cup W$. Opposite subspaces automatically have the same dimension. Opposite points are just non-collinear ones. The singular subspaces of dimension $r-2$ are called \emph{submaximal}. 

Now a parapolar space is a point-line geometry satisfying the following four axioms:
\begin{compactenum}[(PPS1)]
\item There is line $L$ and a point $p$ such that no point of $L$ is collinear to $p$. 
\item The geometry is connected.
\item Let $x,y$ be two points at distance 2. Then either there is a unique point collinear to both---and then the pair $\{x,y\}$ is called \emph{special}---or the convex subspace closure of $\{x,y\}$ is a polar space $\xi(x,y)$---and then the pair $\{x,y\}$ is called a \emph{symplectic pair}. Such polar spaces are called \emph{symplecta}, or \emph{symps} for short.
\item Each line is contained in a symplecton.
\end{compactenum}

The parapolar spaces we will encounter all have the rather peculiar property that all symps have the same rank, which is then called the (uniform) \emph{symplectic rank} of the parapolar space. In contrast, the maximal singular subspaces (which will be projective spaces) will not all have the same dimension. The \emph{singular ranks} of a parapolar space with only projective spaces as singular subspaces (which is automatic if the symplectic rank is at least 3) are the dimensions of the maximal singular subspaces. 

A parapolar space without special pairs is called \emph{strong}. 

Now let $\Gamma=(X,\cL)$ be a parapolar space all of whose symps have rank at least 3.  Let $x\in X$. Then we define the geometry $\Res_\Gamma(x)=(\cL_x,\Pi_x)$ as the geometry with point set the set of lines $\cL_x$ through $x$ and set of lines the set $\Pi_p$ of planar line pencils with vertex $x$ and call it the \emph{residue at $x$}, or the \emph{residual geometry at $x$}. 

In the present paper we will mainly deal with buildings of type $\mathsf{B}_n$, $\mathsf{D}_n$, $\mathsf{E_6}$,  $\mathsf{E_7}$ and $\mathsf{F_4}$. The parapolar spaces we will be concerned with are dual polar spaces, half spin geometries, metasymplectic spaces and the exceptional geometries of type $\mathsf{E_{6,1}}$, $\mathsf{E_{7,1}}$ and $\mathsf{E_{7,7}}$, which we now briefly introduce.
 
\subsection{Lie incidence geometries}
Let $\Delta$ be an irreducible thick spherical building. Let $n$ be its rank, let $S$ be its type set and let $s\in S$. Then we define a point-line geometry $\Gamma=(X,\cL,*)$ as follows. The point set $X$ is just the set of vertices of $\Delta$ of type $s$; the set $\cL$ of lines are the flags of type $s^\sim$, where $s^\sim$ is the set of types adjacent to $s$ in the Coxeter diagram of $\Delta$. If $x$ is a vertex of type $s$ and $F$ a flag of type $s^\sim$, then $x*F$ if $F\cup \{x\}$ is a flag. The geometry $\Gamma$ is called a \emph{Lie incidence geometry}. For instance, if $\Delta$ has type $\mathsf{A}_n$, and $s=1$ (remember we use Bourbaki labelling), then $\Gamma$ is the point-line geometry of a projective space. If $\mathsf{X}_n$ is the {Coxeter} type of $\Delta$ and $\Gamma$ is defined using $s\in S$ as above, then we say that $\Gamma$ has \emph{type} $\mathsf{X}_{n,s}$. Another example: Geometries of type $\mathsf{B}_{n,1}$ and $\mathsf{D}_{n,1}$ are polar spaces. 

For the classical diagrams $\mathsf{X}\in\{\mathsf{A,B,D}\}$, the geometries of type $\mathsf{X}_{n,k}$ are the projective and polar Grassmannians. Geometries of type $\mathsf{B}_{n,n}$ are more specifically called \emph{dual polar spaces} and there is a huge literature about them. Geometries of type $\mathsf{D}_{n,n}$ are more specifically called \emph{half spin geometries}. They are in fact Grassmannians of the corresponding oriflamme geometries (see the beginning of Section~\ref{sec:kang}). Many properties of dual polar spaces and half spin geometries can be deduced from the underlying polar spaces.

Buildings of type $\mathsf{A,D,E}$ are uniquely defined by their underlying field $\K$ (or skew field in the case of $\mathsf{A}$), provided the rank is at least 3. We denote the corresponding building of type $\mathsf{X}_n$ by $\mathsf{X}_n(\K)$, and the corresponding Lie incidence geometries of type $\mathsf{X}_{n,k}$ by $\mathsf{X}_{n,k}(\K)$.

We now describe the Lie incidence geometries $\mathsf{E_{6,1}}(\K)$, $\mathsf{E_{7,1}}(\K)$ and $\mathsf{E_{7,7}}(\K)$. This is best done by displaying the possible mutual positions of the symps, the points and, in case of $\mathsf{E_{7,1}}(\K)$, the class of convex subgeometries isomorphic to $\mathsf{E_{6,1}}(\K)$, called \emph{paras}. 

\subsection{Some parapolar spaces of exceptional type}\label{paras}


\subsubsection{$\sE_{6,1}(\K)$}

First let $\Delta=(X,\cL)$ be the parapolar space $\mathsf{E_{6,1}}(\K)$, for some field $\K$. The elements of the corresponding building of types $1,2,3,4,5,6$,  are the \emph{points, $5$-spaces, lines, planes, $4$-spaces} and \emph{symps}, respectively.  The symps are isomorphic to polar spaces $\mathsf{D_{5,1}}(\K)$. The hyperplanes of the $5$-spaces are called \emph{$4'$-spaces} to distinguish them from the $4$-spaces. The $4'$-spaces correspond in the building to flags of type $\{2,6\}$. Hence every $4'$-space is contained in a unique $5$-space, and also in a unique symp. All singular $3$-spaces form one orbit (of $\Aut\Delta$) and correspond to flags of type $\{2,5,6\}$; they all arise as the intersection of a $4$-space and a $4'$-space, or of two $4'$-spaces. 

The mutual positions of the elements can be deduced from a model of an apartment of the corresponding building, more precisely the graph on the vertices of type $1$ of a Coxeter complex of type $\mathsf{E_6}$, adjacent when contained in adjacent chambers. This is exactly the complement of the point graph of the generalised quadrangle of order $(2,4)$ (three points per line, five lines through each point). A concrete model is the following: Let $V$ be the set of all pairs of the set $\{1,2,3,4,5,6\}$, plus two copies of the latter, denoted $\{1,2,3,4,5,6\}$ and $\{1',2',3',4',5',6'\}$. Adjacency is given by $\{i,j\}\sim\{k,\ell\}$ if $|\{i,j,k,\ell\}|=3$; $i\sim\{j,k\}\sim i'$ if $i\notin\{j,k\}$; $i\sim i'$; $i\sim j$ and $i'\sim j'$ if $i\neq j$. As an example how to use such a model, we deduce the mutual positions of points and $5$-space. We may fix a $5$-space $\{1,2,3,4,5,6\}$. Then a point is either of the form $i$, or of the form $(i,j)$ or of the form $i'$. In the first case the point belongs to the $5$-space; in the second case it is collinear to a $3$-space (given by the four vertices of $\{1,2,3,4,5,6\}$ distinct from $i,j$); in the third case it is collinear to a unique point of the $5$-space. 

We now mention all such relevant possible mutual positions, introducing some more terminology.

\begin{fact}\label{factE6}
Let $\Delta=(X,\cL)$ be the parapolar space $\mathsf{E_{6,1}}(\K)$. Then the following holds.
\begin{compactenum}[$(i)$]
\item $\Delta$ is strong and has diameter $2$, that is, two distinct points either lie in a unique symplecton, or are collinear (and then lie in many symplecta).
\item The geometry with point set the set of symps of $\Delta$, and where the lines are the sets of symps containing a fixed $4$-space, is isomorphic to $\Delta$ itself (and we denote it by $\mathsf{E_{6,6}}(\K)$); this is the duality principle, which implies that two distinct symps in $\Delta$ either intersect in a unique point, or in a unique $4$-space.  
\item A point not contained in a symp is collinear to either no points of that symp---then we say that the point and the symp are \underline{far}---or to all points of a unique $4'$-space of that symp---the point and the symp are \underline{close}.
\item A point not contained in a $5$-space is collinear to either a unique point of that $5$-space, or to a $3$-space contained in that $5$-space. 
\item A symp and a $5$-space intersect in either a $4'$-space, a line, or the empty set.
\item Two $5$-spaces are either disjoint or intersect in a point or plane.
\end{compactenum}
\end{fact}

For each point $x\in X$, the residual geometry $\Res_\Delta(x)$ is isomorphic to  the half spin geometry $\mathsf{D_{5,5}}(\K)$.

We will need the following specific property.

\begin{lemma}\label{closefar}
Let $\xi$ and $\xi'$ be two distinct symps of $\mathsf{E_{6,1}}(\K)$. Then there exists a point $x$ far from $\xi$ and close to $\xi'$ (and hence not contained in either).  
\end{lemma}
\begin{proof}
The set of points close to $\xi'$ is precisely the union of all lines having at least one point in common with $\xi'$. It follows from Section~3.2 of \cite{Sch-Sch-Mal:24} that each point $x$ of that set not contained in $\xi'$ is contained in a line containing a point far from $\xi'$, and all such lines are not contained in one single symp.  The assertion now follows from applying this to a point $x$ of $\xi'$ close to (so not contained in) $\xi$.
\end{proof}


\subsubsection{$\sE_{7,1}(\K)$}

Let $\Delta$ be the Lie incidence geometry $\mathsf{E_{7,1}}(\K)$, for some field $\K$. This is sometimes also referred to as the \emph{long root (subgroup) geometry} related to the building $\mathsf{E_7}(\K)$; the node 1 is the so-called \emph{polar node}, see \cite{PVMexc}.  Then $\Delta$ is a parapolar space, which has diameter 3 and is non-strong. 

The elements of the corresponding building of types $1,2,3,4,5,6,7$,  are the \emph{points, $6$-spaces, lines, planes, $4$-spaces, symps} and \emph{paras}, respectively.  The symps are isomorphic to polar spaces $\mathsf{D_{5,1}}(\K)$, and the paras are strong parapolar spaces isomorphic to $\mathsf{E_{6,1}}(\K)$. The other types are singular (projective) subspaces of $\Delta$. Besides those, $\Delta$ also contains singular subspaces of dimension $5$, which do not correspond to a type but each of them is the intersection of a unique para and a unique  $6$-space, that is, it corresponds in the building to a flag of type $\{2,7\}$. The $4$-dimensional subspaces contained in those $5$-spaces are also singular subspaces of $\Delta$ not corresponding to a single type of $\Delta$; those are referred to as $4'$-spaces and correspond in the building to flags of type $\{2,6,7\}$. 

Again, one can deduce the possible mutual positions of points, symps and paras, etc., by considering an appropriate model of an apartment of a building of type  $\mathsf{E}_{7}$. Such models are given in \cite{HVM-MV}. We limit ourselves here to mentioning that the root system of type $\mathsf{E_7}$ provides a good such model: the points of the apartment are the roots; two points are collinear if the corresponding roots form an angle of 60 degrees; two points are special if the corresponding roots form an angle of 120 degrees; two points are symplectic if the corresponding roots are perpendicular; two points are opposite if the corresponding roots are opposite.  The symps are the subsystems of type $\mathsf{D_5}$. The paras those of type $\mathsf{E_6}$.


\begin{fact}[Point-symp relations] \label{factE71} If $p$ is a point and $\xi$ a symp of $\Delta$ with $p \notin \xi$, then precisely one of the following occurs.

\begin{compactenum}[$(i)$]
\item $p$ is symplectic to a unique point $q\in\xi$. In this case, $p$ and $x$ are special for all $x\in \xi \cap (q^\perp\setminus\{q\})$, and $p$ and $x$ are opposite for all $x\in \xi\setminus q^\perp$. 
\item$p$ is collinear to a $4$-space $U$ of $\xi$; also $p$ and $\xi$ are contained in a unique para $\Pi$. 
In this case, $p$ and $x$ are symplectic if $x\in \xi\setminus U$. 
\item $p$ is symplectic to each point of a $4$-space $U$ of $\xi$;  in this case, $p$ and $x$ are special if $x\in \xi\setminus U$. 
\item   there is a unique line $L\subseteq\xi$ with $p\perp L$. In this case, $p$ and $x$ are symplectic if $x\in \xi \cap (L^\perp\setminus L)$ and $p$ and $x$ are special if $x\in \xi\setminus L^\perp$.
\item $p$ is symplectic to all points of $\xi$. In this case, $p$ and $\xi$ are contained in a unique para $\Pi$, in which they are $\Pi$-opposite. 
\end{compactenum}
In Cases~$(i)$, $(iii)$ and $(iv)$, the point $p$ and the symp $\xi$ are not contained in a common para.
\end{fact}

\begin{fact}[Point-para relations] \label{factE73} If $p$ is a point and $\Pi$ a para of $\Delta$ with $p\notin \Pi$, then precisely one of the following occurs.
\begin{compactenum}[$(i)$]
\item  $p$ is collinear to a unique $5$-space $W$ in $\Pi$. In this case, $p$ is said to be \underline{close} to $\Pi$. The point $p$ is symplectic or special to the points of $\Pi\setminus W$; it is special to $x\in\Pi$ precisely when $x$ is collinear to a unique point of $W$.
\item  $p$ is not collinear to any point of $\Pi$, but it is contained in a unique para $\Pi'$ that intersects $\Pi$ in a symp. In this case, $p$ and $\Pi$ are said to be \underline{far} from each other. The point $p$ and the symp $\Pi\cap\Pi'$ are opposite in $\Pi'$. 
\end{compactenum}
\end{fact}

\begin{fact}[Para-para relations]\label{factE72} If $\Pi$ and $\Pi'$ are distinct paras, then precisely one of the following occurs.

\begin{compactenum}[$(i)$]
\item $\Pi\cap\Pi'$ is a {symp}; 
\item $\Pi\cap\Pi'$ is a point;
\item $\Pi\cap\Pi'=\emptyset$ and each point $x\in \Pi$ is far from $\Pi'$. Let $\xi_x$ be the unique symp of $\Pi'$ contained in a para with $x$, unique by \emph{Fact~\ref{factE73}$(ii)$}. Then each point of $\Pi'\setminus\xi_x$ collinear to a point of $\xi_x$ is special to $x$, and each point in $\Pi'$ which is $\Pi'$-opposite $\xi_x$ is at distance $3$ from $x$. The correspondence $\Pi\longrightarrow\Pi':x\mapsto\xi_x$ induces an isomorphism of $\pi$ onto the dual of $\Pi'$. 
\end{compactenum} 
\end{fact}

\begin{fact}\label{6paras}
Let $\Pi$ be a para of $\Delta$ and let $W,W'$ be two $\Pi$-opposite singular $5$-spaces of $\Pi$. Let $U$ and $U'$ be the unique singular $6$-spaces containing $W$ and $W'$, respectively. Then every point $u\in W$ is collinear to a unique point $\theta(u)$ of $W'$.  Let $u\in U$ and $u'\in U'$. Then \begin{compactenum}[$(i)$]\item $u\perp u'$ if, and only if,  $u\in W$ and $u'=\theta(u)$,\item $u\pperp u'$ if, and only if, $u\in W$, $u'\in W'$ and $u'\neq\theta(u)$,\item $u$ and $u'$ are special if, and only if, either $u\in U\setminus W$ and $u'\in W'$, or $u\in W$ and $u'\in U'\setminus W'$, \item $u$ is opposite $u'$ if, and only if, $u\in U\setminus W$ and $u'\in U'\setminus W'$.\end{compactenum} 
Conversely, let $W$ and $W'$ be two singular $5$-spaces such that each point of $W$ is collinear to a unique point of $W'$. Then $W$ and $W'$ are contained in a unique para $\Pi$, where they are $\Pi$-opposite. 
\end{fact}

We also record the following property of $\Delta$ (which in fact holds for all long root geometries related to spherical buildings):

\begin{fact}\label{dist3fact}
Let $p\perp x\perp y\perp q$ be a path in $\Delta$ with $(p,y)$ and $(q,x)$ special. Then $p$ and $q$ are opposite, i.e., $\delta(p,q)=3$. Conversely, if for some points $p,q,r$ holds $p\pperp r\perp q$, then $p$ is never opposite $q$.
\end{fact}

For each point $x\in X$, the residual geometry $\Res_\Delta(x)$ is isomorphic to  the half spin geometry $\mathsf{D_{6,6}}(\K)$.


\subsubsection{$\sE_{7,7}(\K)$}

Let $\Delta$ be the long root geometry $\mathsf{E_{7,1}}(\K)$. 
We shall use the notation $\Delta^*$ for the point-line geometry $\mathsf{E_{7,7}}(\K)$ obtained from $\Delta$ by taking as points the paras of $\Delta$ and as lines the symps of $\Delta$, with obvious incidence relation. We refer to $\Delta^*$ as the \emph{dual} of $\Delta$. 

Then $\Delta^*$ is a strong parapolar space of diameter 3; points at distance 3 are called \emph{opposite}. A maximal singular subspace has either dimension $5$ (in this case occurring as an intersection of two symps and corresponding to a type $3$ element in the Dynkin diagram) or dimension $6$ (type $2$ in the Dynkin diagram). The $5$-dimensional subspaces of a $6$-space will be called \emph{$5'$-spaces}. They do not correspond to a single node of the Dynkin diagram, but rather to a flag of type $\{1,2\}$.  Each symp of $\Delta^*$ is isomorphic to the polar space $\mathsf{D_{6,1}}(\K)$ (corresponding to the residue of an element of type $1$ in the underlying spherical building). Furthermore, the lines, planes, $3$-dimensional singular subspaces and $4$-dimensional subspaces correspond to types 6, 5, 4 and $\{2,3\}$ in the Dynkin diagram. 

We now review the point-symp and symp-symp relations. As in the previous cases, they  can be deduced  by considering an appropriate model of an apartment (the ``thin version'') of a building of type  $\mathsf{E}_{7}$, as given in \cite{HVM-MV}.  Here, such a model can be given by the \emph{Gosset graph}, which has in turn many descriptions and constructions. One of them is as the $1$-skeleton of the $3_{21}$ polytope (see \cite{BCN}, pages 103 and 104). A traditional construction runs as follows. The {56} vertices are the pairs
from the respective $8$-sets $\{1,2,\ldots,8\}$ and $\{1',2',\ldots,8'\}$.
Two pairs from the same set are adjacent if they intersect in precisely
one element; two pairs $\{a,b\}$ and $\{c',d'\}$ from different sets
are adjacent if $\{a,b\}$ and $\{c,d\}$ are disjoint. The symps correspond to cross-polytopes of size 12 (so-called \emph{hexacrosses} or \emph{$6$-orthoplexes}) contained in the Gosset graph. There are 126 such, and 56 of these are determined by an ordered pair $(i,j)$ with $i, j\in\{1,2,3,4,5,6,7,8\}$, $i\neq j$, and induced on the vertices $\{i,k\}$ and $\{j',k'\}$, $k\notin\{i,j\}$, whereas the other 70 are determined by a $4$-set $\{i,j,k,\ell\}\subseteq\{1,2,3,4,5,6,7,8\}$ and are induced on the vertices $\{s,t\}\subseteq\{i,j,k,\ell\}$, $s\neq t$, and $\{u',v'\}\subseteq\{1',2',3',4',5',6',7',8'\}\setminus\{i',j',k',\ell'\}$, $u\neq v$.

Armed with this description of the thin version, one can verify the following facts.

\begin{fact}[Point-symp relations] \label{factE77} If $p$ is a point and $\xi$ a symp of $\Delta^*$ with $p\notin \xi$, then precisely one of the following occurs.
\begin{compactenum}[$(i)$]
\item  $p$ is collinear to a unique point $q\in\xi$. In this case, $p$ and $x$ are symplectic if $x\in \xi\cap (q^\perp\setminus\{q\})$ and $\delta(p,x)=3$ for $x\in \xi\setminus q^\perp$. We say that $p$ is \underline{far} from $\xi$.
\item  $p$ is collinear to a $5'$-space $U$ of $\xi$. In this case, $x$ and $p$ are symplectic if $x\in \xi\setminus U$. We say that $p$ is \underline{close} to $\xi$.
\end{compactenum}
\end{fact}

This fact implies that on each line $L$, there is at least one point symplectic to a given point $p$ (unique when $L$ contains at least one point opposite $p$). We will use this without reference. 

\begin{fact}[Symp-symp relations] \label{factE77symp} If $\xi$ and $\xi'$ are two symps of $\Delta^*$, then precisely one of the following occurs.
\begin{compactenum}[$(i)$]
\item $\xi=\xi'$;
\item $\xi\cap\xi'$ is a $5$-space. We call the symps \underline{adjacent}.
\item $\xi\cap\xi'$ is a line $L$. Then points $x\in\xi\setminus L$ and $x'\in\xi'\setminus L$ are never collinear and $\delta(x,x')=3$ if, and only if, $x^\perp\cap L$ is disjoint from ${x'}^\perp\cap L$. We call $\{\xi,\xi'\}$ \underline{symplectic}. 
\item $\xi\cap\xi'=\emptyset$ and there is a unique symp $\xi''$ intersecting $\xi$ in a $5$-space $U$ and intersecting $\xi'$ in a $5$-space $U'$, with $U$ and $U'$ opposite in $\xi''$. All points of $\xi\setminus U$ are far from $\xi'$ and each point of $U$ is close to $\xi'$. Also, each line connecting a point of $\xi$ with a point of $\xi'$ contains a point of $U\cup U'$. We call $\{\xi,\xi'\}$ \underline{special}.
\item $\xi\cap\xi'=\emptyset$ and every point of $\xi$ is collinear to a unique point of $\xi'$. In this situation, $\xi$ and $\xi'$ are \underline{opposite}. Each point of $\xi$ is far from $\xi'$.
\end{compactenum}
\end{fact}

At each point $x\in X$, the residual geometry $\Res_\Delta(x)$ is isomorphic to  the geometry $\mathsf{E_{6,1}}(\K)$.

The following well-known property can be deduced from Propositions 4.4 and 4.7 of \cite{Sch-Mal:23}.

\begin{fact}\label{imE7}
Let $\xi_1$ and $\xi_2$ be two opposite symps of $\Delta^*$. Let $\cL$ be the set of lines $L$ that contain a point of $\xi_1$ and one of $\xi_2$. Then for each point $p$ on each member of $\cL$ there exists a symp $\xi_p$ intersecting each member of $\cL$. The set of all such symps $\xi_p$ forms a partition of the union of $\cL$. 
\end{fact}

\begin{remark}\label{E7im}The set of symps in \cref{imE7} is called a \emph{full imaginary set of symps}. Since the type of the symps is $1$ in the $\mathsf{E_7}$ diagram, and this is precisely the polar type, a root elation pointwise fixes a symp $\xi$ and stabilises all symps nondisjoint from $\xi$ (also all points close to $\xi$ are fixed). The corresponding root group acts sharply transitively on each full imaginary set of symps containing $\xi$, except for $\xi$ itself, which is obviously fixed. The root groups thus corresponding to two opposite symps $\xi_1$, $\xi_2$ generate a collineation group acting as $\PSL_2(\K)$ on the corresponding full imaginary set of symps. Each collineation of that group pointwise fixes the set $E(\xi_1,\xi_2)$ of points close to both $\xi_1$ and $\xi_2$, hence close to each member of  the corresponding full imaginary set of symps. The set $E(\xi_1,\xi_2)$ endowed with all lines contained in it, is called the \emph{equator geometry (with poles $\xi_1$ and $\xi_2$)} in \cite{DSV} and \cite{Sch-Mal:23}. 
\end{remark}


\subsubsection{Types $\sF_{4,1}$ and $\sF_{4,4}$}

The Lie incidence geometries of types $\mathsf{F_{4,1}}$ and $\mathsf{F_{4,4}}$ are the main examples of the so-called \emph{metasymplectc spaces}. Buildings of type $\mathsf{F_4}$ are not determined by a field $\K$ alone, but they also need a quadratic alternative division algebra $\AA$ over $\K$. The planes of type $\{1,2\}$ are then projective planes over $\K$, and those of type $\{3,4\}$ are projective planes over $\AA$. We denote such a building by $\mathsf{F_4}(\K,\AA)$, and the related metasymplectic spaces by  $\mathsf{F_{4,1}}(\K,\AA)$ and $\mathsf{F_{4,4}}(\K,\AA)$. The latter two are dual to each other in the sense that the geometry deduced from one of them by declaring the symps as new points, and the planes as new lines, with natural incidence, provides the other geometry. 

Let $\Gamma=(X,\cL)$ either be $\mathsf{F_{4,1}}(\K,\AA)$ or $\mathsf{F_{4,4}}(\K,\AA)$, for a field $\K$ and $\AA$ as above. Then $\Gamma$ is a non-strong parapolar space of diameter 3.  Moreover, $\Gamma$ has the following properties.

\begin{fact}[Point-symp relations]\label{M7}
Let $p$ be a point and $\xi$ a symp of $\Gamma$ with $p\notin \xi$. Then one
of the following occurs: 
\begin{compactenum}[$(i)$]
\item $p^\perp \cap \xi$ is line $L$. In this case, $p$ and $x$ are symplectic for all $x\in \xi \cap (L^\perp\setminus L)$, and $p$ and $x$ are special for all $x\in \xi \setminus L^\perp$. We say that $p$ and $\xi$ are \underline{close};
\item $p^\perp \cap \xi$ is empty, but there is a unique point $u$ of $\xi$ symplectic to $p$. Then $x$ and $p$ are special for all $x \in \xi \cap (u^\perp\setminus\{u\})$, and $x$ and $p$ are opposite if $x\in \xi\setminus u^\perp$. We say that $p$ and $\xi$ are \underline{far}.
\end{compactenum}
\end{fact}

\begin{fact}[Symp-symp relations]\label{M6}
The intersection of two symps is either empty, or a
point, or a plane.  If two symps are opposite, then they are disjoint and every point of one symp is far from the other symp. If two disjoint symps are not opposite, then there exists a unique third symp intersecting both symps in planes. 
\end{fact}

At each point $x\in X$ the residue $\Res_\Gamma(x)$ is a dual polar space of rank 3.

We can be more precise about the isomorphism classes of the symps as follows. The symps of $\mathsf{F_{4,1}}(\K,\AA)$ are all isomorphic to the polar space denoted by $\mathsf{B_{3,1}}(\K,\AA)$ and arising from the quadric in $\mathsf{PG}(n,\K)=\mathsf{PG(V)}$, with $n=2r-1+\mathsf{dim_{\K}(\AA)}$ and $V=\K^{2r}\oplus\AA$, with equation 
    \[x_{-r}x_r + \cdots+x_{-2}x_2 + x_{-1}x_1 = \mathsf{n}(x_0),\]
    where $x_{-r},x_r,\ldots,x_{-2},x_2,x_{-1},x_1 \in \K$, $x_0\in \AA$ and $\mathsf{n}$ the natural norm form of $\AA$.

In contrast, the symps of $\mathsf{F_{4,4}}(\K,\AA)$ are all isomorphic to the polar spaces denoted by $\mathsf{C_{3,1}}(\AA,\K)$ and arising from a symplectic polarity if $\K=\AA$, or isomorphic to the unique non-embeddable polar space over $\AA$ if $\AA$ is octonion (see chapter 9 of \cite{Tits:74}), or else, arising from the pseudo-quadratic form  in $\mathsf{PG(5,\AA)}$ given by    \[{\overline x_{-3}}x_3 + {\overline x_{-2}}x_2 + {\overline x_{-1}}x_1 \in \K,\]
    where $x_{-3},x_3,x_{-2},x_2,x_{-1},x_1 \in \AA$ and $ x \mapsto \overline{x}$ is the standard involution of $\AA$ (this includes the case of inseparable field extensions, where the standard involution is trivial).

\begin{remark}\label{oppositionU}
For each of the Lie incidence geometries $\mathsf{E_{7,1}}(\K)$, $\mathsf{E_{7,7}}(\K)$, $\mathsf{F_{4,1}}(\K,\AA)$ and $\mathsf{F_{4,4}}(\K,\AA)$ (and in fact, for each Lie incidence geometry arising from a spherical building in such a way that ``points'' are a self-opposite type), the fact that in an apartment each point has a unique opposite implies readily that a singular subspace $U$ is opposite another singular subspace $U'$ of opposite type if, and only if, each point of $U$ is opposite some point of $U'$. This also holds for symps and paras. 
\end{remark}

\subsection{Quaternion and octonion Veronese varieties}\label{secQOVV}

We will also have to deal with Veronese varieties. Let us briefly introduce these. Let $\K$ be a field and $\AA\supseteq\K$ a quadratic alternative division algebra over $\K$. Recall that such an algebra admits a unique involution $\AA\rightarrow\AA: x\mapsto\overline{x}$ such that $x\overline{x}\in\K$ and $x+\overline{x}\in\K$. This involution is called the \emph{standard involution}. It is trivial if and only if $\AA=\K$ or $\AA$ is an inseparable field extension of $\K$ in characteristic 2. In any case, the dimension of $\AA$ over $\K$ is either infinite or a power of $2$. If the standard involution is not the identity, then $\AA$ is either a quadratic separable field extension of $\K$, a quaternion algebra, or an octonion algebra. The dimensions are then 2,4 and 8, respectively. 
Set $d=\dim_{\K}\AA$.

The \emph{Veronese variety $\cV(\K,\AA)$} is the following set of points of $\PG(2+3d,\K)$:
$$\{(x\overline{x},y\overline{y},1,y,\overline{x},x\overline{y})\mid x,y\in\AA\}\cup\{(x\overline{x},1,0,0,0,x)\mid x\in\AA\}\cup\{(1,0,0,0,0,0)\}.$$ 
 
If $\AA$ is associative, this coincides with the set
$\{(x\overline{x},y\overline{y},z\overline{z},y\overline{z},z\overline{x},x\overline{y})\mid x,y,z\in\AA\}$.

The set of points of $\cV(\K,\AA)$ can be identified with the points of the projective plane $\PG(2,\AA)$ in such a way that lines bijectively correspond to the quadrics of Witt index 1 contained in subspaces of dimension $d+1$ of $\PG(2+3d,\K)$. Such subspaces are called \emph{host spaces} of $\cV(\K,\AA)$. 

If $\AA$ is a quaternion division algebra, then $\cV(\K,\AA)$ lives in $\PG(14,\K)$. We call this Veronese variety a \emph{(separable) quaternion Veronesean}, or \emph{(separable) quaternion Veronese variety}. We add the adjective ``separable'' in order to distinguish it from the other case in which the variety spans a $14$-dimensional projective space, namely, the case where $\AA$ is an inseparable field extension of degree 4 of $\K$. In this case, the corresponding Veronese variety will be---slightly abusively---called an \emph{inseparable quaternion Veronesean} or \emph{inseparable quaternion Veronese variety}. Likewise, we will call the corresponding algebra an \emph{inseparable quaternion algebra}. A ``true'' quaternion division algebra will be called a \emph{separable quaternion algebra}. This (nonstandard) terminology will add to the clarity and brevity of statements. 

Similar considerations hold for the octonion case, but we will never need the inseparable case here.  However, for clarity, we will sometimes add the adjective ``separable'' when we deal with \emph{octonion Veroneseans}, or \emph{octonion Veronese varieties}, meaning that $\AA$ is a non-associative alternative division algebra over $\K$.

\subsection{Opposition}\label{oppsec}
For a given automorphism $\theta$ of a spherical building $\Omega$, we say that a vertex $v$ is \emph{non-domestic} (\emph{domestic}) if $\theta$ maps (does not map) $v$ to an opposite. If all vertices of type $t$ are domestic, we say that $\theta$ is $t$-domestic. We sometimes make $t$ explicit by calling it by its geometric name, like \emph{point-domestic} if we are dealing with a certain point-line geometry related to $\Omega$ (like in \cref{paras}).  

If $v$ is non-domestic, then the automorphism $\theta_v$ of $\Res_\Omega(v)$ maps by definition each chamber $C$ through $v$ to the projection of $C^\theta$ onto $v$; this has the property that two elements $w,w^{\theta_v}$ of $\Res_\Omega(v)$ are opposite in the residue if and only if $w$ and $w^\theta$ are opposite in $\Omega$, see Proposition~3.29 of \cite{Tits:74}. 

A \emph{panel} is a simplex of size $r-1$, where $r$ is the rank of the building.

\section{Kangaroos in oriflamme geometries} \label{sec:kang}
Oriflamme geometries are related to the polar spaces with the property that each submaximal singular subspace (next-to-maximal singular subspace) is contained in exactly two maximal subspaces; it follows that there are two natural classes of maximal singular spaces so that adjacent ones are in a different class. The oriflamme geometry treats these two classes as elements of different type, and ignores the submaximal subspaces, just like the corresponding building of type $\mathsf{D}_n$.   The polar spaces themselves are sometimes called \emph{hyperbolic}, or \emph{of hyperbolic type}. Disjoint maximal singular subspaces are \emph{opposite} in the building theoretical sense. Opposite maximal singular subspaces have the same type if and only if $n$ is even. Maximal subspaces are also sometimes called \emph{generators}.

An automorphism of an oriflamme geometry is called an \emph{oppomorphism} if its natural action on the Dynkin diagram agrees with the opposition. Otherwise, it is called an \emph{anti-oppomorphism}.


A collineation of an oriflamme geometry, or hyperbolic polar space, is called \emph{type-preserving} if it stabilises the two natural systems of generators; hence if it induces a type-preserving automorphism of the corresponding spherical building. A collineation which is not type-preserving is called \emph{type-interchanging}.

As (anti-)oppomorphisms of polar spaces act faithfully on the point set, they are collineations of the polar space. For hyperbolic polar spaces we have:
\begin{compactenum}[$\bullet$]
	\item oppomorphisms are type-preserving if $n$ is even and type-interchanging if $n$ is odd;
	\item anti-oppomorphisms are type-preserving if $n$ is odd and type-interchanging if $n$ is even.
\end{compactenum}

We now introduce the notion of a kangaroo. 
\subsection{Kangaroos and general properties}

\begin{defn}
Let $\theta$ be a non-trivial collineation of a polar space $\Gamma=(X,\cL)$ of type $\mathsf{D}_n$, $n\geq 2$.  If $\theta$ maps no point of $X$ to another collinear point, and $\theta$ has at least one fixed point, then we call $\theta$ a \emph{kangaroo (collineation)} of $\Gamma$.
If $\theta$ has at least one pointwise fixed line, we call it \emph{lazy}. If $\theta$ has no fixed lines, we call it \emph{diligent}.
\end{defn}

Although we are mainly interested in diligent kangaroos, we prove some general properties of also lazy ones, as this helps us recognising diligent kangaroos easier. 

Our first aim is to show a characterisation of kangaroos, providing alternative definitions. Therefore, we prove the following lemmas. 

We assume throughout that $\Gamma=(X,\cL)$ is a polar space of type $\mathsf{D}_{n}$, $n\geq 2$. 

\begin{lemma}\label{capfix}
Let $\theta$ be a kangaroo of $\Gamma$ and let $U$ be any singular subspace of $\Gamma$. Then $U\cap U^\theta$ is fixed pointwise.
\end{lemma}
\begin{proof}
Suppose $p\in U\cap U^\theta$ is mapped to $p^\theta\neq p$. As $p\in U$, we know that $p^\theta\in U^\theta\ni p$, but then $p\perp p^\theta$,  a contradiction. Hence $p=p^\theta$ and the lemma follows.
\end{proof}

\begin{cor}
The fixed point set of a kangaroo in a generator of $\Gamma$ is a singular subspace.
\end{cor}

\begin{lemma}\label{par}
The parity of the dimension of the fixed point set of a kangaroo in a generator is the same as
\begin{compactenum}
	\item[$\bullet$] $n$ for type-interchanging kangaroos;
	\item[$\bullet$] $n-1$ for type-preserving kangaroos.
\end{compactenum}
\end{lemma}
\begin{proof}
The dimension of the intersection of two generators has the same parity as $n$ if and only if they are of different type. The assertion now follows from \cref{capfix}.
\end{proof}

\begin{lemma}\label{fixed_subspace}
The dimension of the fixed point set of a kangaroo in a generator is the same for every generator.
\end{lemma}
\begin{proof}
Suppose $U$ is a generator of $\Gamma$ with a fixed $k$-space $K$, where $k$ is chosen maximal, $0\leq k<n$. Each generator $S$ adjacent to $U$ also contains a fixed $k$-space, as $S\cap S^\theta$ contains the fixed $(k-1)$-space $S\cap K$, and due to parity and \cref{par}, it has to be a $k$-space. A connectivity argument completes the proof.
\end{proof}

Since kangaroos have, by definition, at least one fixed point, it follows that kangaroos do not map generators to disjoint ones, hence they are domestic. The following characterization shows which domestic collineations are kangaroos. First a definition. Recall that line-domestic collineations of generalised quadrangles are characterised by the shape of their fixed point set: either the perp of a point is pointwise fixed, or an ovoid is pointwise fixed, or a large full subquadrangle is pointwise fixed. The latter is a full subquadrangle that has a non-empty intersection with each line.  There is some similarity with the following definition.

\begin{defn}
A subspace $\Gamma'=(X',\cL')$ of $\Gamma$ is called $i$-large, if every $i$-dimensional singular subspace intersects $X'$ nontrivially, and there exists an $(i-1)$-dimensional singular subspace disjoint from $X'$. 
\end{defn}

We are interested in a special class of $i$-large subspaces, namely, the nondegenerate ones of minimal rank.  It will turn out that those can also be characterised by the counterpart for polar subspaces of the so-called \emph{ideal subpolygons}. 

\begin{defn}
A subspace $\Gamma'=(X',\cL')$ of $\Gamma$ is called \emph{ideal} if the residue of each singular subspace $S$ of $\Gamma'$ in $\Gamma'$, where $$\dim S=\max\{\dim U\mid U\subseteq X'\mbox{ is a singular subspace of }\Gamma'\}-1,$$ is an ovoid in the corresponding residue in $\Gamma$.
\end{defn}

We now have the following equivalences. 

\begin{prop}\label{kangcharact}
Let $\Gamma$ have rank at least $3$. The following are equivalent.
\begin{compactenum}[$(i)$]
\item $\theta$ is a kangaroo or the identity\label{kang}
\item There exists $k$, $0\leq k\leq n-1$ such that for every generator $M$, the subspace $M\cap M^\theta$ is $k$-dimensional and globally fixed by $\theta$.\label{glob}
\item There exists $k$, $0\leq k\leq n-1$ such that for every generator $M$, the subspace $M\cap M^\theta$ is $k$-dimensional and pointwise fixed by $\theta$. \label{point}
\item $\theta$ has at least one fixed point, maps no line to a different but coplanar line, and no line to a disjoint but non-opposite and non-collinear one.  \label{longroot}
\item There exists $k$, $0\leq k\leq n-1$ such that the fixed point set of $\theta$ is an $(n-k-1)$-large nondegenerate polar subspace of rank $k+1$. \label{large}
\item The fixed point set of $\theta$ is a nonempty ideal subspace. \label{ideal}
\end{compactenum}
\end{prop}

\begin{proof}
$(\ref{glob})\Rightarrow(\ref{point})$. If $k=0$, the intersections are points themselves and are thus fixed pointwise.

Now let $k>0$, suppose $M$ is a generator and $U=M\cap M^\theta$ is a $k$-dimensional subspace, fixed by $\theta$. Let $V$ be a $(k-1)$-dimensional subspace of $U$.  Suppose for a contradiction that $V\neq V^\theta$. Thus $V\cap V^\theta$ is $(k-2)$-dimensional.

Pick a complement $W$ of $U$ in $M$; then $W$ is $(n-k-2)$-dimensional. Take a generator $N\neq M$ through $\< W,V\>$. We know that $N\cap N^\theta$ is $k$-dimensional and $V\cap V^\theta\subset N\cap N^\theta$. So there exists a line $L$ in $N\cap N^\theta$ disjoint from $V\cap V^\theta$. But $L$ is collinear with $V$ and $V^\theta$, so $L$ is collinear with $U$. Since $L$ is also collinear with $W\subseteq M$ and $W^\theta\subseteq M^\theta$, we infer that $L$ is collinear with $M$ and $M^\theta$.  So it is contained in $M$ and $M^\theta$, and hence in $M\cap M^\theta=U$.  But then $N\cap N^\theta=\< V\cap V^\theta,L\>\subseteq U$, a contradiction. Hence every $(k-1)$-dimensional subspace of $U$ is stabilised, implying that $U$ is pointwise fixed. 

$(\ref{kang})\Rightarrow(\ref{point})$. This follows from Lemma~\ref{fixed_subspace}.

$(\ref{point})\Rightarrow(\ref{kang})$. Obvious.

$(\ref{point})\Rightarrow(\ref{glob})$. Trivial. 

$(\ref{kang})\Rightarrow(\ref{longroot})$. If a line is mapped to a different but coplanar line, then clearly some points are mapped onto collinear points, a contradiction. Similar argument for a line $L$ mapped onto a disjoint but non-opposite line $L^\theta$.   Indeed, a point $x$ on $L$ collinear to all points of $L^\theta$ (which exists since $L$ and $L^\theta$ are not opposite) is mapped onto a collinear point.

$(\ref{longroot})\Rightarrow(\ref{kang})$. Let $x_0$ be a fixed point of $\theta$. We first show that no point collinear to $x_0$ is mapped onto a collinear point. Suppose for a contradiction that some point $x\perp x_0$ is mapped onto a collinear point $x'\neq x$. If $x'$ were not on the line $xx_0$, then the line $xx_0$ would be mapped onto a coplanar but distinct line. Hence $x'\in xx_0$.  Since the rank of $\Gamma$ is at least 3, there exists a point $y\perp xx_0$. If $y$ is mapped onto a collinear (or fixed) point $y'$, then again $y'\in yx_0$. But then the line $xy$ is distinct and coplanar with $x'y'$, a contradiction. Hence $y'=y^\theta$ is not collinear to $y$. But then the line $xy$ is mapped onto the disjoint but non-opposite and non-collinear line $x'y'$ ($x'\perp xy$), the final contradiction. Hence either $x=x^\theta$ or $x$ is mapped onto an opposite point. If there exists some point $x_1\perp x_0$, $x_1\ne x_0$, fixed under $\theta$, then the whole line $x_0x_1$ is fixed pointwise, every point is collinear to a fixed point and the previous argument yields ($\ref{kang}$). 

Hence we may assume that $\theta$ maps every point collinear to $x_0$ to an opposite. It follows that each generator $M$ through $x_0$ is mapped onto a generator $M^\theta$ with $|M\cap M^\theta|=1$. By parity, the intersection $S$ of any generator $N$ with its image is never empty. By the previous paragraph, and since each point of $S$ is mapped onto a collinear point, $S$ is a singleton $\{s\}$. Now suppose for a contradiction that $s\neq s^\theta$. Let $p\in N\cap x_0^\perp$ be arbitrary. Then $s$ is collinear to $ps\cup p^\theta s^\theta$, so $ps$ is not opposite $(ps)^\theta$. Since $p$ is opposite $p^\theta$ by our assumption, $ps$ is not collinear to $(ps)^\theta$ either, a contradiction. Hence $s=s^\theta$ and the fixed point set is an ovoid. In any case, each point is collinear to a fixed point and so by the first paragraph, $\theta$ is a kangaroo.  

$(\ref{point})\Rightarrow(\ref{large})$. Let $\Gamma'$ be the fixed point set. Every $(n-k-1)$-dimensional singular subspace lies in some generator $M$ and hence has at least one point $x$ in common with $M\cap M^\theta$, which has dimension $k$ by assumption.  But $x$ is fixed by assumption, hence $\Gamma'$ is at most $n-k-1$-large. It is exactly $n-k-1$-large since a dimension argument yields in each generator an $(n-k-2)$-dimensional subspace disjoint from $M\cap M^\theta$, and hence disjoint from $\Gamma'$. Since $\Gamma'$ obviously has rank $k+1$, it remains to show that $\Gamma'$ is nondegenerate. Suppose, for a contradiction, that $x\in\Gamma'$ is collinear to all points of $\Gamma'$. An arbitrary generator not through $x$ contains a $k$-dimensional subspace $S$ pointwise fixed under $\theta$. Then $\<x,S\>$ has dimension $k+1$ and is fixed, a contradiction. 

$(\ref{large})\Rightarrow(\ref{ideal})$. Let $S$ be a submaximal singular subspace of the fixed point set $\Gamma'$, hence of dimension $k-1$. Let $M$ be a maximal singular subspace of $\Gamma$ containing $S$. If $M$ did not contain a singular subspace $U$ of $\Gamma'$ of dimension $k$, then it would contain a subspace $W$ of $M$ complementary to $S$, hence of dimension $n-k-1$ and disjoint from $\Gamma'$. This contradicts $\Gamma'$ being $(n-k-1)$-large.

$(\ref{ideal})\Rightarrow(\ref{large})$. Let $\Gamma'$ again be the fixed point set. We first show that it is a nondegenerate subspace. Indeed, suppose for a contradiction that some point $x\in\Gamma'$ is collinear to all points of $\Gamma'$. Let $k\geq 1$ be the maximal dimension of singular subspaces in $\Gamma'$, and consider a subspace $U\subseteq\Gamma'$ of dimension $k$. Then $x\in U$. Let $W\subseteq U$ be a subspace of dimension $k-1$. By maximality of $k$, the singular subspace $W$ is contained in a unique singular subspace of dimension $k$ of $\Gamma'$ (namely, $U$), contradicting the fact that these should form an ovoid in the residue. Hence $\Gamma'$ is nondegenerate of rank $k+1$. We claim that every generator of $\Gamma$ contains a $k$-dimensional singular subspace of $\Gamma'$. Indeed, assume that some generator $M$ of $\Gamma$ intersects $\Gamma'$ in a subspace $U$ of dimension $i<k$ (possibly $i=-1$). By the nondegeneracy of $\Gamma'$, we find a singular subspace $S$ of dimension $k-1$ in $\Gamma'$ containing $U$. Then there is a unique generator $M'$ containing $S$ and intersecting $M$ in a subspace $S'$ of dimension $n-k-i$ containing $U$. By assumption $M'$ contains a unique subspace $K$ of $\Gamma'$ of dimension $k$ containing $S$. A dimension argument yields $(M\cap K)\setminus S\neq\emptyset$, a contradiction. The claim follows. Since every singular subspace of dimension $n-k-1$ is contained in a generator, we see that $\Gamma'$ is $(n-k-1)$-large.

$(\ref{ideal})\Rightarrow(\ref{kang})$. Suppose for a contradiction that some point $x$ is mapped onto a collinear point $x'\neq x$.  Let $M$ be a generator of $\Gamma$ containing $x$ and $x'$. We know from the previous paragraph that the fixed point set $\Gamma'$ is an $(n-k-1)$-large nondegenerate polar subspace of rank $k+1$, and hence that $M$ contains a $k$-dimensional subspace $S$ of $\Gamma'$. Let $R\subseteq S$ be a subspace of dimension $k-1$. Considering the residue of $R$, it suffices to show the assertion for $\Gamma'$ an ovoid. 

Now $x^\perp\cap\Gamma'=x'^\perp\cap\Gamma'=(xx')^\perp\cap\Gamma'$. But $x^\perp\cap\Gamma'$ defines an ovoid in the residue of $x$, whereas $(xx')^\perp\cap\Gamma'$ defines an ovoid in a singular hyperplane of that residue. Hence the assertion follows if we show that no ovoid of any polar space is contained in the perp of some point $p$. But this is obviously true as there exists at least one maximal singular subspace $M$ containing a given submaximal singular subspace $S$ in a maximal singular subspace through $p$, with $S$ and $M$ both disjoint from both the ovoid and $\{p\}$.  
\end{proof}

Note that $(\ref{longroot})$ expresses kangaroo behaviour of $\theta$ in the long root geometry of $\Gamma$, that is, the line Grassmannian geometry. In this geometry, there are five distinct mutual positions of two points: They can be equal, collinear, symplectic, special and opposite. Now $(\ref{longroot})$ says that points can be mapped only to itself, symplectic and opposite ones, skipping the possibilities of collinear and special. 

Whenever $\theta$ is lazy, the fixed point set is a polar space and, by the main result of \cite{CGP}, automatically arises as the intersection of the corresponding hyperbolic quadric with a subspace of the ambient projective space. We now investigate this behaviour for diligent kangaroos. 

\subsection{Diligent kangaroos}

Suppose $\theta$ is a diligent kangaroo of a polar space $\Gamma=(X,\cL)$ of type $\mathsf{D}_n$. We assume that $\Gamma$ is defined over a field $\K$ (hence it corresponds to a hyperbolic quadric in $\PG(2n-1,\K)$ and hence it can be viewed as a full subgeometry of this projective space). Note that $\theta$ extends to a collineation of $\PG(2n-1,\K)$ and as such is either \emph{linear} or \emph{semi-linear}, according to whether or not the companion field automorphism of the underlying vector space automorphism is trivial.

\begin{lemma}\label{busyovoid}
The fixed point set of $\theta$ is an ovoid.
\end{lemma}
\begin{proof}
This follows from Proposition~\ref{kangcharact}$(\ref{ideal})$. 
\end{proof}

\begin{lemma}\label{disjoint_msd}
Let $\Gamma=(X,\cL)$ be a full subgeometry of $\PG(2n-1,\K)$ as above, and let $S$ be a subspace of $\PG(2n-1,\K)$ with $\dim S < n$. Then there exists a generator $M$ of $\Gamma$ disjoint from $S$.
\end{lemma}
\begin{proof} We prove this by induction on $n$. 

For $n=2$, we may assume that $S$ is a line. If $S$ belongs to $\Gamma$, then we find a disjoint line of $\Gamma$; if $S$ does not belong to $\Gamma$, then it intersects $\Gamma$ in at most two points and we find a line of $\Gamma$ containing neither. 


Now suppose $n\geq 3$. We may obviously assume that $S\cap X$ is nonempty. For an arbitrary point $q\in S\cap X$, we select a point $p\in X\setminus S$ opposite $q$.  Then $S\cap \< p^\perp\>$ is a hyperplane of $S$ and has hence dimension at most $n-2$. Let $H$ be a hyperplane of $\<p^\perp\>$ not containing $p$. Projecting $p^\perp$ and $S\cap\<p^\perp\>$ onto $H$ we obtain a polar space $\Gamma'$ type $\mathsf{D}_{n-1}$ in $H$ and a subspace $S'$ of $H$ with $\dim S'<n-1$. Applying induction we find a generator $M'$ of $\Gamma'$ disjoint from $S'$. It follows that $M=\<p,M'\>$ is a generator of $\Gamma$ disjoint from $S$.
\end{proof}

A \emph{skeleton} of a projective space of dimension $d$ is a set of $d+2$ points no $d+1$ of which are contained in a common hyperplane. The next lemma is proved in full generality, although we only need it for hyperbolic polar spaces. 

\begin{lemma}\label{skel}
Let $O$ be an ovoid of a polar space $\Gamma$ with point set $X$ of rank at least $2$, which is a full subgeometry of $\PG(d,\K)$. If $|\K|>2$, then $O$ contains a skeleton of $\<O\>$. 
\end{lemma}
\begin{proof}
Suppose for a contradiction that $O$ does not contain a skeleton. Then the smallest closed substructure it is contained in is a degenerate projective space and is hence contained in two complementary subspaces $U$ and $W$. If $U\cup W$ contains points of $\Gamma$ outside $O$, then after projecting from a nonempty subspace $S$ entirely contained in $\Gamma$ maximal with respect to the property of being contained in $U\cap W$ and not containing a member of $O$, we may assume that $O=(U\cup W)\cap X$. 

There are two possibilities after this reduction: either the projection performed in the previous paragraph yields a polar space of rank 1, or  is again a polar space of rank at least 2. First suppose the latter.

Select $u\in U\cap O$ and $w\in W\cap O$. First suppose that the line $\<u,v\>$ of $\PG(d,\K)$ contains a least one more point $x\in X$.  Let $M$ be the intersection of an arbitrary maximal singular subspace of $\Gamma$ through $u$ and another arbitrary maximal singular subspace of $\Gamma$ through $w$. The maximal singular subspace $\<M,x\>$ generated by $M$ and $x$ contains a point of $O$, and, without loss of generality, we may assume it is a point $u'\in U$. The line $\<u,u'\>$ intersects $\<M,w\>$ in a point $u''\in X\cap U$ and hence belongs to $O$. Clearly $u''\perp w$ contradicts $O$ being an ovoid. 

So we may assume that $\Gamma$ is a quadric. Note that we may also assume that $\K$ is infinite, because in the finite case $|\K|=q$, one readily sees that $O=(U\cup W)\cap X$ can only happen if, for some $k,\ell,m$ with $k>\ell,m$, one has $q^k+1=(q^\ell+1)+(q^m+1)$, which is obviously impossible for $q>2$, except if $k=1$, $\ell=m=0$ and $q=3$. In the latter case $\Gamma$ is hyperbolic of rank $2$ and $X$ is a quadric in $\PG(3,3)$. However, every ovoid of $X$ lies in a plane, as is readily checked, and the four points of the ovoid form a skeleton of the plane.  

Assume now first that $\Gamma$ is not of hyperbolic type, that is, with the above notation, the singular subspace $M$ is contained in at least $3$ (and hence infinitely many) maximal singular subspaces.  Pick an arbitrary maximal singular subspace $A$ through $M$ not containing $u$ or $w$. As before we may assume that there exists $u'\in U\cap A$. The plane $\<u,u',w\>$ intersects $X$ in a conic $C$. Since $|C|$ is infinite, there exist three members $c_1,c_2,c_3\in C$ such that the intersections $d_i=\<M,c_i\>\cap O$ belong either to $U$ for all $i=1,2,3$, or to $W$, for all $i=1,2,3$.   Then all subspaces $\<M,c\>$, with $c\in C$, intersect $\<d_1,d_2,d_3\>$ and hence either $\<M,u\>$ intersects $W$ nontrivially, or $\<M,w\>$ intersects $U$ nontrivially, both contradictions. 

Finally we assume that $X$ is a hyperbolic quadric. In this case, we consider a hyperplane $H$ of $M$ (with notation as above). As before, we may assume that there exists three maximal singular subspaces through $H$, belonging to the same natural system (hence mutually intersecting precisely in $H$) intersecting $U$ in points (of $O$), say $u_1,u_2,u_3$. Any other maximal singular subspace through $H$ then intersects  the plane $\<u_1,u_2,u_3\>$ in a point, which necessarily belongs to $O$, contradicting the existence of $w\in W$ in one of these maximal singular subspaces. 

Now suppose that the reduction in the first paragraph yields a polar space of rank 1. We initially assume that the polar space is not hyperbolic. Then it contains at least three points. If $X$ is contained in a line, the assertion is trivial. So we may suppose that $W$ has dimension at least 1. We have $X\subseteq U\cup W$. Select $u\in U$ and let $H_u$ be the tangent hyperplane at $u$. Then $H_u$ does not contain any point of $X\cap W$, so it intersects it in a hyperplane. Pick $w\in W\setminus H_u$. then the line $\<u,w\>$ is not a tangent line and hence intersects $X$ in at least two points, implying $w\in X$. This of course contradicts $X$ being of rank 1 (remember we may assume $|\K|>2$). 

Now assume the polar space is hyperbolic. Then instead of projecting from a submaximal singular subspace $M$, we project from a hyperplane of $M$ and obtain a hyperbolic quadric $X$ in $\PG(3,\K)$ with an ovoid $O$ contained in either the union of two disjoint lines or the union of a point and a plane. In the former case, the ovoid only contains four points, which implies $|\K|=3$, and we noted above that in this case $O$ spans a plane. Hence all points of $O$ are contained in a plane $\pi$, except for one point $x$. Let $L$ be a lines of $\Gamma$ through $x$, and set $y=\pi\cap L$. Then $y\notin O$ and there is a unique line of $\Gamma$ through $y$ distinct from $L$.  Since $y\notin O$, the line $M$ contains a point $z$ of $O$ not contained in $\pi$. Obviously   $z\neq x$ and we obtain a contradiction. 

The assertion is proved. 
\end{proof}
We now return to our diligent kangaroo $\theta$. 
\begin{lemma}\label{linear_ovoid}
If $\theta$ is linear, then the fixed point set $O$ spans an $n$-dimensional subspace $S$ in $\PG(2n-1,\K)$, and $O=S\cap X$.
\end{lemma}
\begin{proof} First we claim that the subspace $\<O\>$ of $\PG(2n-1,\K)$ generated by $O$ is pointwise fixed by (the extension of) $\theta$. Indeed, this follows from the fact that $O$ contains a skeleton of $\<O\>$ by \cref{skel} if $|\K|>2$. If $|\K|=2$, then this is trivial (every collineation fixing a basis fixes everything else). The claim follows. Then
\cref{disjoint_msd} implies that $S$ is at least $n$-dimensional. Now suppose $\dim S > n$, then $S\cap \Delta$ contains lines, as generators of $\Delta$ intersect $S$ in at least lines, contradicting the fact that $O$ spanning $S$ implies that $S$ is pointwise fixed under the unique extension of $\theta$ to $\PG(2n-1,\K)$.
\end{proof}

The situation in the semilinear case is completely different, in fact, opposite. In the linear case, the span of the ovoid has minimal dimension, in the semi-linear case it has maximal dimension. Also, only involutions qualify!

\begin{lemma}\label{semilinear_ovoid}
If $\theta$ is semilinear, then it is an involution and an anti-oppomorphism, and the fixed point ovoid $O$ spans $\PG(2n-1,\K)$.
\end{lemma}
\begin{proof}
We prove this statement by induction on $n$. First let $n = 2$. Let $x\in X$ be a point that is not fixed under $\theta$. Then $x^\perp\cap(x^\theta)^\perp$ is fixed pointwise by $\theta$ as every line $L$ of $\Gamma$ through $x$ carries a fixed point $p_L$ and $L= \< x,p_L\>$ is mapped onto $\< x^\sigma,p_L\>$. This already implies that $
\theta$ is an anti-oppomorphism. Hence $(x^\perp \cap (x^\theta)^\perp)^\perp$, which equals $\{x,x^\theta\}$, is preserved under the action of $\theta$, yielding $(x^\theta)^\theta = x$. We conclude that $\theta$ is an involution. Now assume that all fixed points are contained in some plane $\pi$, in which $\theta$ then induces a Baer involution. Clearly, $\pi$ intersects $X$ in a conic $C$ and since $O$ is an ovoid of $\Gamma$, we have $C = O$. But no conic in a Pappian projective plane is pointwise fixed under a Baer involution. This settles the case $n = 2$.

Now let $n > 2$. We claim that $\theta$ is involutive and that, for each point $x$ with $x\neq x^\sigma$, and for each line $L$ through $x$, there exists a point $x'\in L$, $x'\neq x$, contained in $\<\cO\>$.

Indeed, let $x$ be a point not fixed by $\theta$. Set $X'=x^\perp\cap(x^\sigma)^\perp$, and let $\Gamma'=(X',\cL')$ be the corresponding hyperbolic polar space. Then $X'$ spans a subspace of dimension $2n-3$. 
We show first that $O'=O\cap X'$ is an ovoid of $\Gamma'$. Let $U'$ be a generator of $\Gamma'$. Then $U =\< x,U'\>$ is a generator of $\Gamma$ and so, by Proposition~\ref{kangcharact}$(\ref{point})$, $U\cap U^\sigma$ is a fixed point $p$, which is collinear to both $x$ (as $x,p \in U$) and $x^\theta$ (as $x^\theta,p \in U^\theta$). Hence $p\in U'$ (as $U'= U \cap (x^\theta)^\perp$). This shows that $O'$ is an ovoid of $\Gamma'$.

Now define the following collineation $\theta'$ of $\Gamma'$. Let $z\in X'$. Then $z^{\theta'}$ is by definition the unique point on $\<x^\theta ,z^\theta \>$ collinear to $x$. Since $\theta'$ is the composition of a semilinear map and a (linear) projection, it is a semilinear collineation of $\Gamma'$.  
Now $O'$ is exactly the set of fixed points of $\theta'$, since each fixed point $q$ is collinear to $q^\theta$, which is a contradiction to $\theta$ being a kangaroo if $q\neq q^\theta$. By induction, $O'$ generates $\<X'\>$. It follows that ${O'}^\perp = \{x,x^\theta\}$, and so $\{x,x^\theta\}^\theta = \{x,x^\theta\}$, implying $\theta$ is an involution (varying $x$). Also, every line $L\in\cL$ through $x$ has a unique point $x'$ in common with $X'$, and that point lies in $\<O'\>\subseteq \<O\>$. Our claim is proved.

Now suppose for a contradiction that $O$ is contained in some hyperplane $H$ of $\PG(2n-1,\K)$.
Obviously $H$ contains a point $x\notin O$ with $x^\perp\not\subseteq H$. Hence some line $L\in\cL$ through $x$ intersects $H$ in precisely $x$. But our claim above implies that $L$ contains a point $x'\in \<O\>\subseteq \cH$, with $x' \neq x$. Hence $L =\<x,x'\>\subseteq H$, a contradiction.

The fact that $\theta$ is an anti-oppomorphism follows from the fact that it maps each generator $M$ to a generator intersecting $M$ is precisely a point, hence adjacent to an opposite generator.
\end{proof}

\begin{remark}
Since the fixed point set of a lazy kangaroo contains full lines, a lazy kangaroo is always linear. Similarly as above one shows that, if the fixed point set has rank $k+1$ as a polar space, it arises as the intersection of $X\subseteq\PG(2n-1,\K)$ with a subspace of dimension $n+k$. This also follows directly from Lemma 3.1.2 of \cite{Lam-Mal:24}. 
\end{remark}

\begin{example}\label{residue2}
Let $\theta$ be a linear lazy kangaroo of $\mathsf{D_{5,1}}(\K)$ fixing points, lines and planes, but no $3$-spaces. The fix structure is a polar space of rank 3 embedded in a projective space of dimension $5+2=7$. Hence its equation is of the form $$X_{-3}X_3+X_{-2}X_2+X_{-1}X_1=X_0^2+aX_0X_0'+bX_0'^2,$$ with $a,b\in\K$. Since the polynomial $x^2+ax+b$ does not have solutions in $\K$, it defines a quadratic extension $\LL$ of $\K$ and we see that the fix structure is a polar space isomorphic to $\mathsf{B_{3,1}}(\K,\LL)$. 
\end{example}

\section{Kangaroos in $\mathsf{E_{6,1}}(\K)$}\label{sec:kang2}

In this section we let $\theta$ be a type preserving automorphism of $\mathsf{E_{6}}(\K)$ mapping no point of $\mathsf{E_{6,1}}(\K)$ to a collinear one, but having at least one fixed point and mapping at least one point to a point at distance 2. We call such an automorphism a \emph{kangaroo (collineation) of $\mathsf{E_{6,1}}(\K)$}. In order to describe the action of such a kangaroo, it is convenient to consider the standard representation of $\mathsf{E_{6,1}}(\K)$ as a full subgeometry in $\PG(26,\K)$. It is also the so-called \emph{universal one}, see \cite{Ron-Smi:85}, and as such every collineation of it extends to $\PG(26,\K)$. We denote it by $\cE_{6}(\K)$. Our aim is to prove the following classification.

\begin{theo}\label{kangaroo} Let $\theta$ be a collineation of $\mathsf{E_{6,1}}(\K)$ and consider it as a collineation of $\PG(26,\K)$ stabilizing $\cE_{6}(\K)$.
Then $\theta$ is a kangaroo if, and only if, its fixed point structure $F$ is a---separable or inseparable---quaternion Veronese variety in a $14$-dimensional subspace $U$ of $\PG(26,\K)$, with $F=U\cap\cE_6(\K)$, or a (separable) octonion Veronese variety (and then $\theta$ is a Galois involution) in a Baer subspace $\Sigma\cong\PG(26,\FF)$ of $\PG(26,\K)$ over a field $\FF$ with $\K/\FF$ a quadratic Galois extension and $F=\Sigma\cap\cE_6(\K)$. 
\end{theo}

We prove some lemmas which ultimately will culminate in a proof of the theorem. 
In the sequel, $\Delta=(X,\cL)$ is the parapolar space $\mathsf{E_{6,1}}(\K)$. Recall that the symp containing the noncollinear points $x,y$ is denoted by $\xi(x,y)$.
\subsection{General properties of kangaroo collineations}

\begin{lemma}\label{dualkangaroo}
The kangaroo collineation $\theta$ is also a dual kangaroo collineation. 
\end{lemma}

\begin{proof}
We have to show that $\theta$ does not map any symp to an adjacent one and that $\theta$ fixes at least one symp. 

Suppose first for a contradiction that $\theta$ maps the symp $\xi$ to an adjacent symp $\xi^\theta$. Set $U:=\xi\cap\xi^\theta$. Then $U^\theta$ is a $4$-space of $\xi^\theta$ and so $U\cap U^\theta\neq\emptyset$. But any point of $U\cap U^\theta$ is mapped onto a point of $U^\theta$, hence, if not fixed, sent to a collinear one, a contradiction. Consequently $U\cap U^\theta$ is fixed pointwise. 

Pick a $3$-space $W\subseteq U$ and let $V,V'$ be the unique $4'$-space in $\xi,\xi^\theta$, respectively, which contains $W$. \cref{factE6}$(iv)$ implies that $V$ and $V'$ are contained in a common $5$-space and hence all points of $V$ are collinear to all points of $V'$. Consequently all points of $V'\cap V^\theta$ are fixed. But all fixed points of $\xi\cup\xi^\theta$ lie in $U$, hence those of $V\cup V'$ lie in $W$, and consequently those of $V^\theta$ in $W^\theta$. Therefore, $V^\theta\cap V'\subseteq W\cap W^\theta$, so $V^\theta\cap V'= W\cap W^\theta$. Hence $W\cap W^\theta$ is a point or a plane (since $V'$ and $V^\theta$ are both $4'$-spaces of $\xi^\theta$). But we violate this requirement if we choose $W$ so that it intersects $U\cap U^\theta$ in a subspace of codimension 1. We obtain the desired contradiction.

Now let $x$ be a fixed point of $\theta$ and suppose no symp through $x$ is fixed. Then, in $\Res(x)$ viewed as a polar space, every point is mapped onto an opposite, and hence $\theta$ induces a type reversing automorphism. This is a contradiction as types in $\mathsf{E_{6}}(\K)$ are preserved. 
\end{proof}

Thanks to this lemma we may from now on assume that also the dual of everything we prove, holds. 

The next lemma generates a lot of fixed points in a fixed symp.

\begin{lemma}\label{generator}
If a symp $\xi$ is fixed under $\theta$, then every generator of $\xi$ contains a fixed point. Moreover, $\theta$ restricted to $\xi$ is a kangaroo. 
\end{lemma}

\begin{proof}
Let $U$ be a generator of the fixed symp $\xi$. Then, since $\theta$ preserves types, $U\cap U^\theta\neq\emptyset$ and any point in the intersection is mapped onto a point of $U^\theta$ and hence, if not fixed, to a collinear one, contradicting the definition of a kangaroo collineation. Since $\theta$ does not map points to collinear ones, and since there are fixed points, $\theta$ is a kangaroo by definition.
\end{proof}

The next lemma generates a lot of fixed symps, and hence also a lot of fixed points. 

\begin{lemma}\label{fixedsymp}
For every point $x$ with $x\neq x^\theta$, the symp $\xi(x,x^\theta)$ is fixed. Moreover, every point is collinear to a fixed point.
\end{lemma}

\begin{proof}
Let $x$ be a point which is not fixed under the kangaroo collineation $\theta$. Suppose first that $x$ is collinear to some fixed point $p=p^\theta$. Then the symps $\xi(x,x^\theta)$ and $\xi(x^\theta,x^{\theta^2})$ have the line $px^\theta$ in common and hence coincide by Lemma~\ref{dualkangaroo}. Hence the symp $\xi(x,x^\theta)$ is fixed.

Now suppose that $x$ is not collinear to any fixed point. Let $p$ be any fixed point of $\theta$ and let $y$ be collinear to both $p$ and $x$. By the previous paragraph we know that $\xi:=\xi(y,y^\theta)$ is fixed. Since $x\perp y$, the intersection 
$x^\perp\cap\xi$ is a $4'$-space by \cref{factE6}$(iii)$ and hence contains a fixed point by Lemma~\ref{generator}. This contradicts the assumption that $x$ is not collinear to a fixed point. The lemma is proved. 
%
\end{proof}

We note the following direct implications of Lemmas~\ref{dualkangaroo} and~\ref{fixedsymp}.
\begin{cor}\label{handycor}
Each point is contained in a fixed symp and, dually, every symp contains a fixed point.
\end{cor} 

\begin{lemma}\label{nocollE6}
No two fixed points are collinear. In particular, the fixed points contained in a fixed symp form an ovoid of the symp. 
\end{lemma}

\begin{proof}
Assume for a contradiction that $x$ and $y$ are two collinear points with $x=x^\theta$ and $y=y^\theta$. Every symp through the line $xy$ is fixed by Lemma~\ref{dualkangaroo}. Hence $\Res(xy)$ is fixed elementwise. Noting that points of a fixed singular subspace are either mapped to collinear ones or fixed, we see that all points collinear to $xy$ are fixed. By connectivity, all points are fixed and $\theta$ is the identity, a contradiction. 

Hence $\theta$ induces a diligent kangaroo in each fixed symp and the assertion follows from Lemma~\ref{busyovoid}. 
\end{proof}

We can now define a new incidence structure $(F,\Phi)$ as follows. The point set $F$ consists of all points fixed under $\theta$; the block set $\Phi$ consists of all symps fixed under $\theta$.  Incidence is natural.

\begin{lemma}
The point-block geometry $(F,\Phi)$ is a projective plane. 
\end{lemma}

\begin{proof}
Since two distinct fixed points are not collinear they define a unique fixed symp. Dually, two distinct fixed symps intersect in a unique fixed point. Since blocks contain at least three points, the lemma follows. 
\end{proof}

\subsection{Collineations of $\mathsf{E_{6,1}}(\K)$ that pointwise fix a quaternion Veronesean}

Now let $Y\subseteq X$ be a subset of points of the exceptional geometry $\mathsf{E_{6,1}}(\K)$ with the following properties: 
\begin{compactenum}[(VV1)]
\item Each symp $\xi$ containing at least two points of $Y$ intersects $Y$ in a ovoid of $\xi$; we call such an ovoid a $Y$-ovoid.
\item $Y$ is not contained in a symp. 
\end{compactenum}

Our main aim is to show that, if some collineation $\theta$ of $\mathsf{E_{6,1}}(\K)$ pointwise fixes $Y$, then it is a kangaroo. We proceed with a series of lemmas.

We will refer to the symps that contain a $Y$-ovoid as a \emph{host symp}. 

\begin{lemma}\label{twohost}
Two host symps intersect in a unique point, which automatically belongs to $Y$.
\end{lemma}

\begin{proof}
Suppose first for a contradiction that two host symps $\xi_1,\xi_2$ have a $4$-space $U$ in common. By (VV1), $U$ contains a unique point $y\in Y$. Select $y_2\in Y \cap\xi_2\setminus\{y\}$.  By (VV1), $y_2$ is not collinear to $y$, so $y_2^\perp\cap\xi_1$ is a $4'$-space which does not contain $y$, but which, by the definition of ovoid, contains another point $y_1\in Y$. Now any symp containing the collinear points $y_1,y_2\in Y$ violates (VV1).  Hence $\xi_1\cap\xi_2=\{p\}$, for some point $p\in X$. Suppose for a contradiction dat $p\notin Y$. Then there exists $x_1\in \xi_1\cap Y$ collinear to $p$, and $x_2\in\xi_2\cap x_1^\perp\cap Y$ (by (VV1)). Again any symp containing the collinear points $x_1,x_2\in Y$ violates (VV1). 
\end{proof}

Let $\Psi$ be the set of all host spaces. 

\begin{lemma}\label{PPP}
The incidence structure $(Y,\Psi)$, with natural incidence, is a projective plane.
\end{lemma}

\begin{proof}
This follows from (VV1) and Lemma~\ref{twohost}.
\end{proof}

We now also have:

\begin{lemma}\label{Dual}
The set $\Psi$ satisfies {\em (VV1)} in the dual $\mathsf{E_{6,6}}(\K)$ of $\mathsf{E_{6,1}}(\K)$.  
\end{lemma}

\begin{proof}
Let $y\in Y$ and select $\xi\in\Psi$ with $y\notin\xi$ (this exists by Lemma~\ref{PPP}). The members of $\Psi$ through $y$ are given by the symps $\xi(y,x)$, with $x\in \xi\cap Y$. Since $\xi\cap Y$ is an ovoid of $\xi$, the assertion follows. 
\end{proof}

Now comes the crux of the argument.

\begin{lemma}\label{dense}
Every point $x$ of $\mathsf{E_{6,1}}(\K)$ belongs to some host symp, unique if $x\notin Y$.
\end{lemma}

\begin{proof}
We apply approximately the same technique as in the proof of Lemma~\ref{fixedsymp}. So suppose first that $x\perp y\in Y$. By Lemma~\ref{Dual}, there is a unique member $\xi\in\Psi$ containing the line $xy$ (in the residue of $y$, the line $xy$ is a $4$-space of $\mathsf{E_{6,6}}(\K)$). Hence we may assume $x$ is not collinear to any member of $Y$. Let $y\in Y$ be arbitrary and select $z\in x^\perp\cap y^\perp$. Then $z$ is contained in a host symp $\xi$ by the foregoing. But $x^\perp\cap\xi$ is a $4'$-space, and so contains a member of $Y$, contradicting our hypothesis. The lemma is proved, taking into account Lemma~\ref{twohost}..  
\end{proof}

We can now show:

\begin{prop}\label{convers}
The collineation $\theta$ is a kangaroo.
\end{prop}

\begin{proof}
Let $x\in X$ be arbitrary. By Lemma~\ref{dense}, $x$ belongs to a symp $\xi\in\Psi$. Since $\xi$ contains at least two non-collinear fixed points, it is fixed by $\theta$. Then Proposition~\ref{kangcharact} implies that $\theta$ restricted to $\xi$ is a kangaroo. So $x^\theta\not\perp x$. 
\end{proof}

We have now everything in place to prove Theorem~\ref{kangaroo}.

\subsection{Proof of Theorem~\ref{kangaroo}}

Since $\cE_{6}(\K)$ is, by \cite{Ron-Smi:85}, the so-called \emph{universal embedding} of $\mathsf{E_{6,1}}(\K)$, the collineation $\theta$ extends to a collineation of $\PG(26,\K)$. We distinguish between two cases.

\textit{Case 1: Suppose $\theta$ is a linear collineation.} By Lemmas~\ref{linear_ovoid},~\ref{generator} and~\ref{nocollE6}, each member of $\Phi$ is a quadric of Witt index 1 in some $5$-space of $\PG(26,\K)$. Since different such $5$-spaces are contained in distinct symps, they pairwise intersect in unique points; also, each such $5$-space intersects $X$ in that quadric of Witt index 1. Hence Main Result~4.3 of \cite{ADS-HVM} implies that $(F,\Phi)$ is a quaternion Veronese variety in a $14$-dimensional subspace $U$ of $\PG(26,\K)$. Since $\theta$ fixes $U$ pointwise, no other point of $X$ is contained in $U$ and so $F=U\cap\cE_6(\K)$. 

Now, conversely,  suppose some collineation $\theta$ pointwise fixes a quaternion Veronese variety on $\cE_6(\K)$ arising as intersection with a subspace $U$.  Then Proposition~\ref{convers} implies that $\theta$ is a kangaroo.

\begin{remark}\label{remexE74}
We will classify all linear kangaroo collineations in Section~\ref{existenceE74E6}, that is, we will show that such a collineation exists whenever the field $\K$ admits a quaternion division algebra, or an inseparable extension of degree 4, hence whenever there exists a separabale or inseparable quaternion Veronesean over $\K$ (and then, there also exists one inside $\mathsf{E_{6,1}}(\K)$, and it will turn out that isomorphic quaternion algebras give rise to projectively equivalent Veroneseans and conjugate collineations). Moreover,  we will also determine the abstract isomorphism type of the group pointwise fixing a given quaternion Veronesean. 
\end{remark}

\textit{Case 2: Now let $\theta$ be a semilinear kangaroo collineation.} By Lemmas~\ref{semilinear_ovoid},~\ref{generator} and~\ref{nocollE6}, the restriction of $\theta$ to each subspace $\<\xi\>$ of $\PG(26,\K)$, with $\xi\in\Phi$, is a Baer involution whose fixed point set intersects $\xi$ in an ovoid spanning the $9$-space $\<\xi\>$. Let $x$ be any point of $X$. Lemma~\ref{fixedsymp} (for $x^\theta\neq x$) and the dual of Lemma~\ref{nocollE6}  (for $x^\theta=x$) imply that $x$ is contained in a fixed symp $\xi$. In the latter, $\theta$ induces an involution, hence $x^{\theta^2}=x$, implying that $\theta$ is a Baer involution in $\PG(26,\K)$, say fixing the projective space $\PG(26,\FF)$, with $\K/\FF$ quadratic. It follows that $F$ is the intersection of $X$ with $\PG(26,\FF)$, that each intersection $\xi\cap \PG(26.\FF)$ is a quadric $Q_\xi$ of Witt index 1 in some $9$-space, that, for each pair of such distinct quadrics $Q_\xi$ and $Q_\zeta$, the intersection $\<Q_\xi\>\cap\<Q_\zeta\>$ is a singleton, and that each quadric $Q_\xi$ has trivial nucleus.   Main Result~4.3 of \cite{ADS-HVM} again implies that $(F,\Phi)$ is a (separable) octonion Veronese variety in $\PG(26,\FF)$. The converse again follows from Proposition~\ref{convers}.


\section{Existence and uniqueness of quaternion Veroneseans in $\mathsf{E_{6,1}}(\K)$}  \label{existenceE74E6}

In this section we show a kind of converse, or rather `addition' to Theorem~\ref{kangaroo}, see also Remark~\ref{remexE74}. Roughly, our aim is to prove that every equivalence class of norm forms of a quaternion algebra over the field $\K$ gives rise to a projectively unique quaternion Veronese variety contained in  $\cE_6(\K)$ (via $\mathsf{E_{6,1}}(\K)$) and also in a subspace of dimension $14$ in $\PG(26,\K)$. Also, we determine the pointwise stabiliser, both as a subgroup of the automorphism group of $\mathsf{E_{6,1}}(\K)$ and as an abstract group. We start with the case of a separable quaternion Veronese variety.

Form here on, it is convenient to use the letter $X$ for elements of a quaternion algebra over $\K$, and not anymore for the point set of $\mathsf{E_{6,1}}(\K)$. This will cause no confusion. 

\subsection{Linear ovoids of polar spaces of type $\mathsf{D_5}$}

We have seen that a linear kangaroo collineation of $\mathsf{E_{6,1}}(\K)$ pointwise fixes a quaternion veronesean subvariety $\cV(\K,\HH)$ of $\cE_{6}(\K)$, with $\HH$ a quaternion algebra over $\K$. Recall the notion of \emph{host spaces} of $\cV(\K,\HH)$ from \cref{secQOVV}, which are $5$-dimensional subspaces of $\PG(26,\K)$ here.  Let $\PG(5,\K)$ be an arbitrary host space of $\cV(\K,\HH)$ with corresponding quadric $O$. Then we can choose coordinates in such a way that points are labeled with $(x_1,x_2,X)\in\K\times\K\times\HH$ and $O$ has equation $x_1x_2=X\overline{X}$. Let $\OO'$ be the split octonion algebra obtained by the Cayley-Dickson doubling process applied to $\HH$ with primitive element 1, that is, a generic element of $\OO'$ is a pair $(a,b)$ of quaternions, with multiplication $(a,b)\cdot(c,d)=(ac+d\overline{b},\overline{a}d+cb)$, where $x\mapsto\overline{x}$ is the standard involution in $\HH$. Consider the bilinear form $\beta:\OO'\times\OO'\rightarrow \K:(X,Y)\mapsto X\overline{Y}+Y\overline{X}$, where $\overline{(a,b)}=(\overline{a},-b)$, for all $a,b\in\HH$. Then $\beta$ is associated to---or is the linearization of---the quadratic form $q:\OO'\mapsto \K:X\mapsto X\overline{X}$, which is a split form. Hence coordinatizing $\PG(9,\K)$ as $(x_1,x_2,X)\in\K\times\K\times\OO'$, the equation $x_1x_2=X\overline{X}$ describes a hyperbolic polar space $Q$ isomorphic to $\mathsf{D_{5,1}}(\K)$. Clearly $O$ is contained in $Q$ in a standard way. By the free choice of coordinates we see that the embedding of $O$ in $Q$ is unique up to a collineation of $Q$.

\subsection{A construction of $\mathcal{E}_6(\K)$}\label{cartan}

Let $V$ be a $27$-dimensional vector space over the field $\K$, and write $V$ as $V=\K\times\K\times\K\times\OO'\times\OO'\times\OO'$. Write a generic vector in $V$ as $(x_1,x_2,x_3;X_1,X_2,X_3)$. Then, by the last section of \cite{SSMV}, the following equations determine 27 (degenerate) quadrics in $\PG(V)$, whose intersection is precisely $\mathcal{E}_6(\K)$:
$$\begin{array}{llll} &x_2x_3=X_1\overline{X}_1, \hspace{2cm} {}&X_2X_3=x_1 \overline{X}_1,\\
\mbox{}\hspace{3.5cm}\mbox{}&x_3x_1=X_2\overline{X}_2, &X_3X_1=x_2 \overline{X}_2,&\mbox{}\hspace{3.5cm}(*) \\
&x_1x_2=X_3\overline{X}_3, &X_1X_2=x_3 \overline{X}_3. 
\end{array}$$

We also use the notation $\mathcal{E}_6(\K)$ for the set of vectors of $V$ whose coordinates satisfy the 27 quadratic equations. It is shown in \cite{SSMV} for $|\K|>2$ that $\mathcal{E}_6(\K)$ is the \emph{projective closure} of the image of the (split octonion) Veronese map $\rho': \OO'\times\OO'\rightarrow V: (X,Y)\mapsto (X\overline{X},Y\overline{Y},1,Y,\overline{X}, X\overline{Y})$, that is, it is the smallest point set with the property that, whenever all points except possibly one of a line is contained in the point set, then all points of that line are contained in it. (This can also easily be seen directly: the given vectors satisfy the above equation, as can be checked by simple calculations, and since the $\mathcal{E}_6(\K)$ is projectively closed, the image of $\rho'$ is contained   in $\mathcal{E}_6(\K)$; a vector with coordinates $(*,*,1,Y,\overline{X},*)$ satisfying the above equations $(*)$ is readily seen to have coordinates $(X\overline{X},Y\overline{Y},1,Y,\overline{X}, X\overline{Y})$.)

Restricting $\OO'$ to $\HH$, we now see that the image $\cV(\K,\HH)$ of the  quaternion Veronesean map  $\rho:\HH\times\HH\times\HH\rightarrow V:(x,y,z)\mapsto (x\overline{x},y\overline{y},z\overline{z},y\overline{z},z\overline{x},x\overline{y})$ is fully contained in $\mathcal{E}_6(\K)$. Our goal is to show that this containment is projectively unique and that it is pointwise fixed by a non-trivial group of automorphisms isomorphic to the multiplicative group $\HH^\times$. We refer to this containment as a \emph{standard inclusion of the quaternion Veronese variety in $\mathsf{E_{6,1}}(\K)$}. Since it depends on the given representation of $\mathsf{E_{6,1}}(\K)$, the projective uniqueness is not obvious.
\subsection{Collineations fixing $O_1$ and $p_1$}\label{coll}
As a first step we determine the group of collineations of $\mathcal{E}_6(\K)$ fixing three points of a fixed ovoid $O_1$ of $\cV(\K,\HH)$ and an additional point $p_1$ of $\cV(\K,\HH)\setminus O_1$. We may choose $O_1$ to be the image $\{\rho(0,y,z)\mid y,z\in\HH\}$, the three points to be $\rho(0,1,0),\rho(0,0,1)$ and $\rho(0,1,1)$, and $p_1$ to be the point $(1,0,0,0,0,0)$. Then a general collineation of $\mathcal{E}_6(\K)$ fixing these four points can be written as $(x,y,z,X,Y,Z)\mapsto (kx,y,z,X^{\theta_1},Y^{\theta_2},Z^{\theta_3})$, with $k\in\K$ and $\theta_1,\theta_2,\theta_3$ linear maps acting on $\OO'$ (as a vector space over $\K$). In addition, $\theta_1$ stabilises $\HH\subseteq\OO'$ and fixes $1\in\HH$. Since the bilinear form $\beta$ must be preserved, up to a scalar, we can write $(a,b)^{\theta_1}=(a^\theta,b^{\theta^*})$. Then $(*)$ yields \begin{eqnarray}Y^{\theta_2}Z^{\theta_3}&=&k\overline{(\overline{YZ})^{\theta_1}}, \mbox{ for all }Y,Z\in\OO'.\label{eq1}\end{eqnarray} By the Principle of Triality and its consequences, see Theorem~3.2.1 and Lemma~3.3.2 of \cite{SV}, as soon as the first three equalities of $(*)$ are preserved, all equalities are preserved.  

Setting subsequently $Y=1$ and $Z=1$ yields $Y^{\theta_2}=\overline{\overline{Y}^{\theta_1}}C_2$ and $Z^{\theta_3}=C_3\overline{\overline{Z}^{\theta_1}}$, for all $Y,Z\in\OO'$, and some constants $C_2=k(1^{\theta_3})^{-1},C_3=k(1^{\theta_2})^{-1}\in\OO'$. Setting $Y=Z=1$ we achieve $C_2C_3=k$. We obtain, replacing $Z$ with $\overline{Z}$ and $Y$ with $\overline{Y}$ in Equation~\ref{eq1}, and setting $\overline{C}_3=C$,
$$(ZY)^{\theta_1}=(Z^{\theta_1}C)(C^{-1}Y^{\theta_1}), \mbox{ for all }Y,Z\in\OO'.$$
Set $C=(c_1,c_2)$, $c_1,c_2\in\HH$. Suppose first $Z=(z,0)\in\HH$ and $Y=(y,0)\in\HH$. Then
$$(zy)^\theta(c_1\overline{c}_1-c_2\overline{c}_2,0)=(z^\theta c_1,\overline{z^\theta}c_2)(\overline{c}_1y^\theta,-y^\theta c_2), \mbox{ for all }y,z\in\HH,$$  which yields \begin{eqnarray}(zy)^\theta(c_1\overline{c}_1-c_2\overline{c}_2)&=&z^\theta y^\theta c_1\overline{c_2}-y^\theta z^\theta c_2\overline{c}_2, \label{eq2}\\
\overline{c}_1(\overline{z^\theta}y^\theta-y^\theta\overline{z^\theta})c_2&=&0,\label{eq3}\end{eqnarray}
for all $y,z\in\HH$. Clearly, since $\HH^\times$ is not commutative, Equation~\ref{eq3} implies either $c_1=0$ or $c_2=0$. If $c_2=0$, then $\theta_1$ is an automorphism of $\HH$ and both $\theta_2$ and $\theta_3$ preserve $\HH$. If $c_1=0$, then $\theta$ is an anti-automorphism and both $\theta_2$ and $\theta_3$ interchange $\{(x,0)\mid x\in\HH\}$ and $\{(0,x)\mid x\in\HH\}$. Setting $c_2=k=1$, $c_1=0$ and $X^\theta=\overline{X}$, for all $X\in\OO'$, we see that we do obtain an automorphism of  $\mathcal{E}_6(\K)$ preserving $O_1\cup\{p_1\}$ and mapping the ovoid $O_2=\{\rho((x,0),0,(z,0))\mid x,z\in\HH\}$ to $O_2'=\{\rho((0,x),0,(0,z))\mid x,z\in\HH\}$. We denote the corresponding quaternion Veronesean by $\cV(\K,\HH)'$. 
\subsection{Quaternion Veroneseans on $\cE_6(\K)$ that contain $O_1$ and $p_1$}\label{Qp}
Next we show that there are precisely two quaternion Veronesean varieties contained in $\mathcal{E}_6(\K)$ that contain $O_1$ and $p_1$. One contains $O_2$ and the other one $O_2'$. Denote the symps of $\cE_6(\K)$ determined by $O_1,O_2,O_3$ by $\xi_1,\xi_2,\xi_3$, respectively (with $O_3=\{\rho((x,0),(y,0),0)\mid x,y\in\HH\}$). Also, for a point $p$ far from two symps $\zeta_1,\zeta_2$, we call the map $\zeta_1\to\zeta_2:q\mapsto\xi(p,q)\cap\zeta_2$ the \emph{projection} (of $\zeta_1$) from $p$ onto $\zeta_2$.

Indeed, we first claim that, if a quaternion Veronesean variety $\cH$ contained in $\mathcal{E}_6(\K)$ and containing itself $O_1\cup\{p_1\}$, contains any point $p$ of $O_2\setminus\{p_1,p_3\}$ (with $p_3=\rho(0,0,1)$), then $\cH=\cV(\K,\HH)$ as defined earlier. Likewise, if $\cH$ contains a point of $O'_2\setminus\{p_1,p_3\}$, then $\cH=\cV(\K,\HH)'$.  Indeed, for each point $q\in O_1$, the symp $\xi(p,q)$ intersects $\xi_3$ in a unique point of $\cH$ that belongs to $O_3$, and varying $q$, all points of $O_3$ are obtained. Likewise, all points of $O_2$ are also contained in $\cH$. An arbitrary point $r\in\cH\setminus(O_1\cup O_2\cup O_3)$ is the intersection of $\xi(p_1,q_1)$ and $\xi(p_2,q_2)$, with $q_i\in O_i$, $i=1,2$, such that $r\in\xi(p_i,q_i)$. But then $r\in\cV(\K,\HH)$. The claim follows. 

By the previous claim, it suffices to show that no point $t_2$ of $\xi_2\setminus(O_2\cup O_2'\cup p_1^\perp\cup p_3^\perp)$ lies in a quaternion Veronesean of $\cE_6(\K)$ together with $O_1$ and $p_1$. Such a point $t_2$ can be written as $(Y\overline{Y},0,1,0,\overline{Y},0)$, with $Y\in\OO'$ invertible. On the other hand, a generic point $t_1$ of $\xi_1\setminus (p_2^\perp\cup p_3^\perp)$ can be written as $(0,1,X\overline{X},\overline{X},0,0)$, with $X\in\OO'$ invertible. We now seek $Z\in\OO'$ such that $t_1,t_2$ and $t_3=(1,Z\overline{Z},0,0,0,\overline{Z})$ are contained in the same symp. 

Consider the cubic form $C$ on $V$ given by $$C(x_1,x_2,x_3,X_1,X_2,X_3)= x_1x_2x_3-x_1X_1\overline{X}_1-x_2X_2\overline{X}_2-x_3X_3\overline{X}_3+(X_1X_2)X_3+\overline{X}_3(\overline{X}_2\overline{X}_1).$$

A tedious calculation shows that the linear combination of any pair of images under $\rho'$ vanishes under $C$. Taking projective closure, this implies that $C$ is the cubic form associated to $\cE_6(\K)$ (see \cite{Asc:87} and \cite{Mal-Vic:22}), as $\cE_6(\K)$ is determined by its host spaces. One checks that, for a triple of non-collinear points $r_1,r_2,r_3$ of $\cE_6(\K)$, the symps $\xi(r_1,r_2),\xi(r_2,r_3)$ and $\xi(r_1,r_3)$ have a nontrivial common intersection if and only if $C(r_1+r_2+r_3)=0$. Applying this to $t_1,t_2,t_3$, we obtain, after an elementary calculation, 
$$C(t_1+t_2+t_3)=0 \Longleftrightarrow 1+X\overline{X}.Y\overline{Y}.Z\overline{Z}+(\overline{X}\overline{Y})\overline{Z}+Z(YX)=0,$$
which is equivalent to $(1+Z(YX))(1+\overline{Z(YX)})=0$. It follows that the set of points of $\xi_3$ collinear to $t_3$, but not to $p_2$, is given by $$\{(1,Z\overline{Z},0,0,0,\overline{Z})\mid \exists N\in\OO', N\overline{N}=0, Z=(N-1)X^{-1}Y^{-1}\}.$$ One sees that the point with $N=0$, namely $(1,Z\overline{Z},0,0,0,\overline{Z})$ with  $Z=-X^{-1}Y^{-1}$, is collinear to all others of that set, and so $t_3=(1,Z\overline{Z},0,0,0,\overline{Z})$, with $Z=-X^{-1}Y^{-1}$. Note that also the ``cyclically obtained'' conditions hold: $Y=-Z^{-1}X^{-1}$ and $X=-Y^{-1}Z^{-1}$ (each one is equivalent to $ZYX=-1$). 

It follows that the projection $O^*_3$ of $O_1$ from $t_2$ onto $\xi_3$ consists of the set of points $$\{(1,(Y\overline{Y}x\overline{x})^{-1},0,0,0,-\overline{Y}^{-1}\overline{x}^{-1})\mid x\in\HH^\times\}\cup\{(p_1,p_2\}.$$

Now we project $O_3^*$ from each point of $O_1$ back on $\xi_2$. We obtain the set of points
$$O_2^*=\{Y\overline{Y}x\overline{x}(x'\overline{x}')^{-1},0,1,0,\overline{x}'^{-1}(\overline{x}\overline{Y}),0)\mid x,x'\in\HH^\times\}\cup\{p_1,p_3\}.$$

Setting $Y=(y_1,y_2)\in\HH\times\HH$ we easily compute $\overline{x}'^{-1}(\overline{x}\overline{Y})=(\overline{x}'^{-1}\overline{x}\,\overline{y}_1,-x'^{-1}xy_2)$. 

If $y_2\neq 0$, then setting $x'=1$, we see that the second component of each point of $O_2^*$ determines the first one. However, now taking arbitrary $x'$, we have the $O_3^*$-points $(\overline{x}'^{-1}\overline{x}\,\overline{y}_1,-x'^{-1}xy_2)$ and, setting $x'=1$ again and replacing $x$ by $x'^{-1}x$, we also have the point $(\overline{x}\,\overline{x}'^{-1}\overline{y}_1,-x'^{-1}xy_2)$, which must hence coincide with $(\overline{x}'^{-1}\overline{x}\,\overline{y}_1,-x'^{-1}xy_2)$. This is only possible if $y_1=0$. But then $O_2^*=O_2'$. If $y_2=0$, then $O_2^*=O_2$. We have shown that there are exactly two quaternion Veroneseans in $\cE_6(\K)$ containing $O_1$ and $p_1$, and that they are mutually projectively equivalent. 

\subsection{Collineations of $\cE_6(\K)$ that pointwise fix $\cV(\K,\HH)$}\label{fixQP}
If a collineation pointwise fixes $\cV(\K,\HH)$, then, using the notation of Paragraph~\ref{coll}, the associated linear map $\theta$ on $\HH$ is the identity. Then Equation~\ref{eq2} yields $c_2=0$. We also have that both $\theta_2$ and $\theta_3$ pointwise fix $\HH$. In particular, they both fix $1$, and so $C_2=C_3=k=1$. We easily deduce now $\theta_1=\theta_2=\theta_3$. By linearity of $\theta_1$, Equation~\ref{eq1} implies that $\theta_1$ is an automorphism of $\OO'$. Hence we have, still with the notation of Paragraph~\ref{coll},
$ac+d\overline{b}=ac+d^{\theta^*}\overline{b^{\theta^*}}$ and $(\overline{a}d+cb)^{\theta^*}=\overline{a}d^{\theta^*}+cb^{\theta^*}$. We readily deduce that $x^{\theta^*}=xa$, with $a\overline{a}=1$, $a\in\HH$. Conversely, every such automorphism pointwise fixes $\cV(\K,\HH)$. Hence the group that pointwise fixes $\cV(\K,\HH)$ inside $\cE_6(\K)$ is isomorphic to the multiplicative group of norm 1 elements of $\HH$. 

Now we turn to the case where the fixed point structure is a Veronesean over an inseparable extension of degree 4.
\subsection{The inseparable case: Embedding into a split Cayley algebra}

The split Cayley algebra $\OO'$ over $\K$ can be defined with Zorn's matrices as follows. Write a generic element $X$ of the $8$-dimensional vector space $\OO'$ and its conjugate $\overline{X}$ in matrix form as follows (with $x_i\in\K$, $i=0,1,\ldots,7$).
$$X=\begin{pmatrix} x_0 & \begin{pmatrix} x_4 \\ x_5 \\ x_6 \end{pmatrix}\\ \begin{pmatrix} x_1 \\ x_2 \\ x_3 \end{pmatrix} & x_7 \end{pmatrix},\; \overline{X}=\begin{pmatrix} x_7 & \begin{pmatrix} -x_4 \\ -x_5 \\ -x_6 \end{pmatrix}\\
\begin{pmatrix} -x_1 \\ -x_2 \\ -x_3 \end{pmatrix} & x_0\end{pmatrix}.$$
We denote $X$ by $ (x_0,[x_1,x_2,x_3],[x_4,x_5,x_6],x_7)$.
 
Suppose $\K$ is a non-perfect field of characteristic $2$. Choose an element $\ell_1\in\K$ that is not a square and suppose that $\K^2+\ell_1\K^2\neq \K$.  Select an element $\ell_2\in\K\setminus (\K^2+\ell_1\K^2)$.
Then the quadruplets $(x_0,x_1,x_2,x_3)\in \K\times \K\times \K\times \K$, with componentwise addition and multiplication as follows, form an inseparable $4$-dimensional extension $\HH$ of $\K$:
\begin{eqnarray*}&&(x_0,x_1,x_2,x_3)\cdot(y_0,y_1,y_2,y_3)=\\
&&(x_0y_0+\ell_1x_1y_1+\ell_2x_2y_2+\ell_1\ell_2x_3y_3,
x_0y_1+x_1y_0+\ell_2x_2y_3+\ell_2x_3y_2,\\
&&x_0y_2+x_2y_0+\ell_1x_1y_3+\ell_1x_3y_1,
x_0y_3+x_1y_2+x_2y_1+x_3y_0).\end{eqnarray*}

In this structure, conjugation (the ``standard involution'') is the identity, and the norm is given by $\mathsf{n}(x_0,x_1,x_2,x_3)=x_0^2+\ell_1x_1^2+\ell_2x_2^2+\ell_1\ell_2x_3^2=(x_0,x_1,x_2,x_3)^2$.

Using the matrix representation, we can embed $\HH$ in $\cO'$ as follows. 
$$\HH\hookrightarrow\OO':(x_0,x_1,x_2,x_3)\mapsto  \begin{pmatrix} x_0 &  \begin{pmatrix} x_1 \\ x_2 \\ \ell_1\ell_2x_3 \end{pmatrix}\\ \begin{pmatrix}\ell_1x_1 \\ \ell_2x_2 \\ x_3 \end{pmatrix} & x_0 \end{pmatrix}.$$
Note that the restriction of the standard involution of $\OO'$ indeed acts trivially on $\HH$. Now the inseparable Veronese map $\rho$ is defined as follows:
$$\rho:\HH\times\HH\times\HH\rightarrow \K^{15}:(X,Y,Z)\mapsto (X^2,Y^2,Z^2,Y{Z},Z{X},X{Y}).$$
The image of $\rho$ in $\PG(14,\K)$ is an inseparable Veronese variety $\mathcal{V}(\K,\HH)$. It embeds in $\mathsf{E_{6,1}}(\K)$ just like the (inseparable) quaternion Veronese variety above.

Again, we want to determine all the automorphisms $\theta$ of $\mathcal{E}_6(\K)$  that fix the inseparable quaternion Veronese variety $\mathcal{V}(\K,\HH)$ pointwise. As with the separable case, one first shows that this is equivalent to finding the group of linear transformations of $\OO'$ (as a vector space over $\K$) pointwise fixing $\HH$ and acting as automorphisms of $\OO'$.  This can be done with elementary though tiresome calculations. Expressing that a general invertible $8\times8$ matrix $A$ with entries in $\K$ pointwise fixes $\HH$ and preserves  multiplication in $\OO'$, we obtain the following  expression for~$A$:


$$A=\begin{pmatrix}
1+a & \ell_1b & \ell_2c & d & b & c & \ell_1\ell_2d & a \\
b & 1+a & \ell_2d & \ell_1^{-1}c & {\ell^{-1}_1}a & d & \ell_2c & b \\
c & \ell_1d & 1+a & {\ell^{-1}_2}b & d & {\ell^{-1}_2}a & \ell_1b & c \\
\ell_1\ell_2d & \ell_1\ell_2c & \ell_1\ell_2b & 1+a & \ell_2c & \ell_1b & \ell_1\ell_2a & \ell_1\ell_2d \\
\ell_1b & \ell_1a & \ell_1\ell_2d & c & 1+a & \ell_1d & \ell_1\ell_2c & \ell_1b \\
\ell_2c & \ell_1\ell_2d & \ell_2a & b & d_2d & 1+a & \ell_1\ell_2b & \ell_2c \\
d & c & b & {\ell^{-1}_1\ell^{-1}_2}a & {\ell^{-1}_1}c & {\ell^{-1}_2}b & 1+a & d \\
a & \ell_1b & \ell_2c & d & b & c & \ell_1\ell_2d & 1+a \end{pmatrix}
$$
with $a,b,c,d\in\K$ and with $a+a^2+\ell_1b^2+\ell_2c^2+\ell_1\ell_2d^2=0$. We will use the notation $A(a,b,c,d)$ for this matrix,  where we consider the quadruple $(a,b,c,d)$ as a member of $\HH$.  Conversely, every such matrix belongs to an automorphism of $\OO'$ pointwise fixing $\HH$.
Let $G$ be the group of all such automorphisms. 
It is easily calculated that $A(a_1,b_1,c_1,d_1)A(a_2,b_2,c_2,d_2)=A((a_1,b_1,c_1,d_1)+(a_2,b_2,c_2,d_2))$. Hence $G$ is isomorphic to an additive subgroup of $\HH$. Let us call an element $(a,b,c,d)$ of $\HH$ \emph{admissible} if $a+a^2+\ell_1b^2+\ell_2c^2+\ell_1\ell_2d^2=0$, that is, $a=(a,b,c,d)^2$, and note that the square is the norm, which we denote by $\mathsf{n}(a,b,c,d)$. Then we can be more specific about the group $G$ (in paricular, prove that it is not trivial):

\begin{prop}
The set of admissible elements of $\HH$ equals $\{0\}\cup\{(1,b,c,d)^{-1}\mid b,c,d\in\K\}$. 
\end{prop}

\begin{proof}
Suppose first that $(a,b,c,d)$ is admissible and distinct from $0$. Then $$0=a^{-2}(a+a^2+\ell_1b^2+\ell_2c^2+\ell_1\ell_2d^2)=a^{-1}+(1+\ell_1b'^2+\ell_2c'^2+\ell_1\ell_2d'^2),$$
with $b'=ba^{-1}$, $c'=ca^{-1}$ and $d'=da^{-1}$. Hence on the one hand $a=\mathsf{n}(a,b,c,d)$, and on the other hand $a=\mathsf{n}(1,b',c',d')^{-1}$. Since the norm is just squaring, we obtain $$(a,b,c,d)=(1,b',c',d')^{-1}.$$

Conversely, for all $b,c,d\in\K$, on can write $(1,b,c,d)^{-1}=(n^{-1},bn^{-1},cn^{-1},dn^{-1})$, with $n=\mathsf{n}(1,b,c,d)$. Then clearly $$\mathsf{n}(n^{-1},bn^{-1},cn^{-1},dn^{-1})=n^{-2}n=n^{-1},$$ which implies that $(n^{-1},bn^{-1},cn^{-1},dn^{-1})$, and hence $(1,b,c,d)^{-1}$, is admissible. 
\end{proof}

So the group $G$ is a subgroup of $\HH$ each element of which depends on three free variables over $\K$. Compare this with the separable quaternion case, where $G$ was the subgroup of $\HH^\times$ of elements of norm 1, hence in a certain sense also has dimension $3$.

It remains to show that all Veronesean varieties isomorphic to $\mathcal{V}(\K,\HH)$ and contained in $\cE_6(\K)$ are mutually equivalent. In fact we will show that there is a unique $\mathcal{V}(\K,\HH)$ containing a given ovoid $O$ of a given symp $\xi$ of $\cE_6(\K)$ arising from the inclusion $\HH\subseteq\OO'$, and a given point $p$ of $\cE_6(\K)$ opposite $\xi$. 

\subsection{The projective uniqueness of $\mathcal{V}(\K,\HH)$ in $\cE_6(\K)$ in the inseparable case}
As in the separable quaternion case above, we are looking for elements $Y\in\OO'$ such that $S_Y:=\{xY\mid x\in\HH\}=\{x'(xY)\mid x',x\in\HH\}=:S_Y^*$. This time we will show that $Y$ itself has to belong to $\HH$, which implies that $\mathcal{V}(\K,\HH)$ inside $\cE_6(\K)$ is determined by the points in one symp (containing at least two points of $\mathcal{V}(\K,\HH)$) and an extra point. 

In order to do so, it is rather convenient to use the following form of $\OO'$. Let $\LL$ be the split (separable) extension $\K(t)$ of $\K$ defined by the element $t$ with $t^2=t$. Then the corresponding natural involution $x\mapsto\overline{x}$, $x\in\LL$, interchanges $t$ with $t+1$. Now we apply to $\LL$ the Cayley-Dickson process twice: first using $\ell_1$ as a primitive element, the second time using $\ell_2$. Then we obtain $\OO'$ as $\LL\times\LL\times\LL\times\LL$ with the following explicit form of the multiplication:
\begin{eqnarray*}(x_1,x_2,x_3,x_4)(y_1,y_2,y_3,y_4)&=&(x_1y_1+\ell_1 \overline{x}_2y_2+\ell_2\overline{x}_3y_3+\ell_1\ell_2 x_4\overline{y}_4, \\ &&\phantom{(}\overline{x}_1y_2+\phantom{\ell_1}{x}_2y_1+\ell_2\overline{x}_3y_4+\phantom{\ell_1}\ell_2 x_4\overline{y}_3,\\ &&\phantom{(} \overline{x}_1y_3+\ell_1 \overline{x}_2y_4+\phantom{\ell_1}x_3y_1+\phantom{\ell_2}\ell_1x_4\overline{y}_2, \\ && \phantom{(}x_1y_4+\phantom{\ell_1}{x}_2y_3+\phantom{\ell_1}{x}_3y_2+\phantom{\ell_1\ell_2}x_4\overline{y}_1),
\end{eqnarray*}
for all $x_1,\ldots,x_4,y_1,\ldots,y_4\in\LL$. The embedding $\HH\subseteq\OO'$ is via $\HH=\K\times\K\times\K\times\K$. 
It follows that,
\begin{eqnarray*}
(0,1,0,0)((0,0,1,0)(y_1,y_2,y_3,y_4))&=&(0,0,0,1)(\overline{y}_1,\overline{y}_2,\overline{y}_3,\overline{y}_4),
\end{eqnarray*}
and so, by the additivity of $S_Y$ (by its very definition), we see that $$(y_1+\overline{y}_1,y_2+\overline{y}_2,y_3+\overline{y}_3,y_4+\overline{y}_4)\in S_Y.$$ Hence $(S_Y\cap\HH)\setminus\{0\}\neq\emptyset$, say $0\neq u\in S_Y\cap \HH$, implying the existence of $x\in\HH^\times$ such that $xY=u$. Since $x$ is invertible, this yields $Y=x^{-1}u\in\HH$ and we are done. 

Taking everything together we have proved:

\begin{prop}\label{existE6}
Suppose the field $\K$ either admits a ---separable or inseparable--- quaternion division algebra $\HH$. Then any standard inclusion of the Veronese variety $\cV\cong\mathcal{V}(\K,\HH)$ in $\mathsf{E_{6,1}}(\K)$ is projectively unique and admits a nontrivial group of linear kangaroo collineations each nontrivial element of which pointwise fixes exactly $\cV$. 
\end{prop}

\begin{remark}\label{uniqueE6}
The results of \cref{Qp} show that, whenever a quaternion veronesean is embedded in $\mathsf{E_{6,1}}(\K)$ in such a way that each of its quadrics is contained in a symp as a classical ovoid (that is, the intersection of the symp with a subspace in its ambient projective space), and each point off a quadric is far from the corresponding symp of $\mathsf{E_{6,1}}(\K)$, then it is a standard inclusion.  
\end{remark}


\section{Collineations with opposition diagram $\sE_{7;4}$ fixing no chamber}\label{sec:E74nochamber}
We now apply the results of the previous sections to the classification of all domestic collineations of $\mathsf{E_7}(\K)$ with opposition diagram $\sE_{7;4}$ fixing no chamber. To that aim, we argue in this section within the parapolar space $\Delta:=\mathsf{E_{7,7}}(\K)$. We first describe two classes of point-domestic collineations which  have opposition diagram $\mathsf{E_{7;4}}$. Then we show that these are the only domestic collineations of $\Delta$  fixing no chamber with that opposition diagram. Since we are interested in large buildings, we assume $|\K|>2$ throughout. 

\subsection{Two classes of examples fixing no chamber}
There will be two classes of such domestic automorphisms; one related to every quadratic field extension of $\K$ and one related to every quaternion division algebra over $\K$.

First, let $\mathcal{I}$ be a full imaginary set of symps of $\Delta$ (see \cref{imE7} and \cref{E7im}). Then by \cref{E7im} the group $\PSL_2(\K)$ acts on $\cI$ pointwise fixing the equator geometry defined by any two members of $\cI$. Let $\theta$ be a collineation of $\Delta$ belonging to that group and acting fixed point freely on $\mathcal{I}$. Then we claim that $\theta$ is a point-domestic collineation. 

Indeed, if a point $x$ is collinear to some $5$-space $U$ of a member $\xi$ of $\mathcal{I}$, then the $6$-space generated by $x$ and $U$ contains a point of $\mathcal E$, and so $x$ is collinear to a fixed point and is mapped by $\theta$ to a point at most at distance $2$ from $x$. If $x$ is collinear to unique points $z_i$ of two members $\xi_i\in\mathcal I$, $i=1,2$, then, if $z_1$ is not collinear to $z_2$, they determine a unique symp $\zeta$ containing $x$ and fixed under $\theta$. So $x,x^\theta\in\zeta$. If $z_1\perp z_2$, then both $x$ and $x^\theta$ are collinear to $z_1z_2=(z_1z_2)^\theta$, and hence the distance between $x$ and $x^\theta$ is again at most $2$.

Note that a collineation like the previous one can only be found if $\K$ admits a quadratic extension, since it requires the existence of a member of $\PSL_2(\K)$ acting fixed point freely in $\PG(1,\K)$; hence it yields a $2\times 2$ matrix with imaginary eigenvalues, that is, eigenvalues in a quadratic extension. 

Now let $\HH$ be a separable or inseparable quaternion division algebra over $\K$ and denote by $\mathsf{C_{3,3}}(\K,\HH)$ the dual polar space arising from the polar space $\mathsf{C_{3,1}}(\HH,\K)$ by taking as points the singular planes of $\mathsf{C_{3,3}}(\HH,\K)$ and as lines the (full) plane pencils (all singular planes through a given line). Suppose $\mathsf{C_{3,3}}(\HH,\K)$ is a full subgeometry of $\mathsf{E_{7,7}}(\K)$ in such a way that the lines of $\mathsf{C_{3,3}}(\HH,\K)$ through a given point $x$ of $\mathsf{C_{3,3}}(\HH,\K)$ form a quaternion Veronesean in $\Res_\Delta(x)$ as in \cref{existenceE74E6}. We call this full subgeometry a \emph{dual polar quaternion Veronesean}.  Suppose also that  $\theta$ is a collineation of $\mathsf{E_{7,7}}(\K)$ with fixed point set a polar quaternion Veronesean $\mathcal{V}$, that is, $\theta$ only has fixed points, fixed lines and fixed symps, and every fixed line belongs to $\cV$ (hence is pointwise fixed) and every fixed symp contains a pointwise fixed full subgeometry isomorphic to a $\mathsf{B_{2,1}}(\K,\HH)$ (which is by definition isomorphic to any point residual in $\mathsf{B_{3,1}}(\K,\HH)$, and which we will refer to as a \emph{quaternion generalised quadrangle}). The existence of such a collineation whenever $\K$ admits a quaternion division algebra $\HH$ shall be proved in \cref{existenceE74}. 

We claim that every point $x$ of $\Delta$ is collinear to at least one fixed point. Indeed, let $\xi$ be any fixed symp. Then the fixed points in $\xi$ form a full subgeometry isomorphic to a generalised quadrangle. Hence, since $x$ is collinear to at least one point of $\xi$, there exists a fixed point $f$ at distance at most $2$ from $x$. Hence there exists a symp $\xi$ containing $x$ and the fixed point $f$. In $\Res_\Delta(f)$, every symp contains at least one fixed point; hence $\xi$ contains a line with all points fixed and so there is at least one fixed point collinear to $x$. 

Hence no point of $\Delta$ is mapped onto an opposite point and so $\theta$ is point-domestic. Since collineations with opposition diagram $\mathsf{E_{7;1}}$ or $\mathsf{E_{7;2}}$ always fix at least one chamber of $\Delta$, the opposition diagram of $\theta$ is necessarily $\mathsf{E_{7;4}}$. 

Note also the following consequence (which is also used in \cite{Ney-Par-Mal:23}).

\begin{cor}\label{sympfixed}
If a point $p$ is not fixed, then $p$ and $p^\theta$ are contained in a unique symp which is fixed under $\theta$. 
\end{cor}

\begin{proof}
Let $x$ be point fixed by $\theta$ and collinear with $p$. In the residue of $x$, the symp determined by $p$ and $p^\theta$ is fixed by Lemma~\ref{fixedsymp}. 
\end{proof}

Hence we have shown:

\begin{prop}\label{F4isdom}
Let $\theta$ be a nontrivial collineation of $\mathsf{E_{7,7}}(\K)$ pointwise fixing a polar quaternion Veronesean. Then $\theta$ is point-domestic with opposition diagram $\mathsf{E_{7;4}}$. More exactly, each point not fixed by $\theta$ is collinear to at least one fixed point.
\end{prop}

Now we turn to the classification. Before we make an assumption on the opposition diagram, we prove a general result which we will apply for both opposition diagrams $\mathsf{E_{7;3}}$ and $\mathsf{E_{7;4}}$. 

The next lemma is also useful for the opposition diagram $\mathsf{E_{7;3}}$ and hence will be stated in more general terms.

\begin{lemma}\label{pointlinefixed}
If $\theta$ is a collineation of $\Delta$ with at least one fixed point, and with the property that for every point $x\perp x^\theta\neq x$ the line $xx^\theta$ is fixed under $\theta$, then every fixed line contains at least one fixed point.
\end{lemma}

\begin{proof}
Let $L$ be a fixed line and let $x$ be a fixed point. If a unique point on $L$ is either incident, collinear or symplectic to $x$, then it is fixed and we are fine.  If all points on $L$ are collinear with $x$, then $\theta$ fixes the plane $\pi$ generated by $x$ and $L$. Pick $u\in\pi\setminus(L\cup\{x\})$. If $u$ is fixed, then so is $L\cap ux$; if $u$ is not fixed then $uu^\theta\cap L$ is fixed. 

At last assume that all points on $L$ are symplectic to $x$. Let $y_1,y_2\in L$, $y_1\neq y_2$, and set $\xi_i:=\xi(x,y_i)$, $i=1,2$. Then $\xi_1$ and $\xi_2$ have a line $M$ in common and since $\xi_1\neq\xi_2$ (as otherwise some point of $L$ would be collinear to $x$), $y_1$ and $y_2$ are collinear to the same point $u\in M$. Then $y_2$ is collinear to the line $uy_1$ of $\xi_1$ and so it is collinear to a $5$-space of $\xi_1$ which spans together with $y_2$ a $6$-space $W$. Also, $x$ is collinear to a $4$-space $U\subseteq W$. Clearly, $U$ is contained in every symp through $x$ and a point of $L$, and so the $5$-space $V$ spanned by $U$ and $x$ is the intersection of all such symps, proving it is uniquely determined by $x$ and $L$ and hence fixed by $\theta$. Now $W$ is the unique $6$-space through $L$ intersecting $V$ in a $4$-space. Hence $W$ is also fixed, and consequently also $U$ is fixed. Now the combination of Proposition~3.3 of \cite{PVMclass} and our assumption yields a point $u\in U$ fixed by $\theta$. Since $u$ is collinear to $L$, we are reduced to the situation of the previous paragraph, completing the proof of the lemma. 
\end{proof}

Now we let $\theta$ be an automorphism of the building $\mathsf{E_7}(\K)$ and we consider $\theta$ as a collineation of $\Delta$. We assume that $\theta$ does not fix any chamber and that it has opposition diagram $\mathsf{E_{7;4}}$.  We begin with some general properties until we make a case distinction each leading to a class of examples.

\subsection{General properties}

\begin{lemma}\label{pointwisefixedlines}
If a symp is mapped onto a symplectic one, then the intersection line is pointwise fixed. Consequently, there exist lines which are pointwise fixed. 
\end{lemma}

\begin{proof}
Let $\xi$ be a symp with $\xi^\theta\cap\xi=M$, with $M$ a line. Suppose for a contradiction that some point $x$ on $M$ is not fixed by $\theta$. Then we can find a point $y\in\xi\setminus M$ collinear to $x$ but not collinear to $x^{\theta^{-1}}$.    Then $y^\theta$ is not collinear to $x$, and so, by \cref{factE77symp}$(iii)$ the points $y$ and $y^\theta$ are opposite, a contradiction to the opposition diagram. Hence all points of $M$ are fixed.

Now we use some notation and terminology from \cref{oppsec}.  Let $L$ be a non-domestic line. Then $\theta_L$ is $3$-space- and $4$-space-domestic. Hence Theorem~6.1 of~\cite{TTVM} yields a symp $\xi$ through $L$ with $\{\xi,\xi^\theta\}$ symplectic.
\end{proof}

\begin{lemma}\label{allpointsfixed}
If a line is fixed, but not all points on it are fixed, then there exists a fixed panel of type $\{1,2,3,4,5,6\}$ containing the line. Consequently if a fixed line has at least one fixed point, then all of its points are fixed. 
\end{lemma}

\begin{proof}
Let $L$ be fixed, but not pointwise.  By Lemma~\ref{pointwisefixedlines}, no symp $\xi$ through $L$ has the property that $\xi\cap\xi^\theta=L$. Hence the map induced by $\theta$ on the irreducible factor of type $\mathsf{D_5}$ of $\Res_\Delta(L)$ is point-domestic (when viewed as a polar space of type $\mathsf{D_5}$, hence of odd rank 5). Lemma~3.16 of \cite{PVMclass}  now completes the proof of the first assertion of the lemma. The second one now follows since, if at least one point of a fixed line was fixed, but not all of them, then $\theta$ would fix a chamber.    
\end{proof}

\begin{lemma}\label{dilemma}
If a point $x$ is mapped onto a collinear point, then either the line $xx^\theta$ is fixed, or $x^{\theta^2}$ is collinear to $x$ and each symp through the plane $\pi$ containing $x,x^\theta$ and $x^{\theta^2}$ is fixed (in particular $\pi$ is fixed). 
\end{lemma}

\begin{proof}
Suppose that the point $x$ is mapped onto a collinear one, and suppose that the line $xx^\theta$ is not fixed under $\theta$. We first claim that $x\perp x^{\theta^2}$. Indeed, suppose not, then the lines $xx^\theta$ and $x^\theta x^{\theta^2}$ correspond to two noncollinear points $p,p'$, respectively, of the residual geometry $\Res_\Delta(x^\theta)$ of type $\mathsf{E_{6,1}}$. Select in $\Res_\Delta(x^\theta)$ a symp $\xi$ containing $p$ and opposite $p'$. Then $\xi$ intersects $\xi^\theta$ in a unique point. This means that, if $\xi$ corresponds to the symp $\zeta$ of $\Delta$, then $\zeta\cap\zeta^\theta$ is a line $L$ through $x^\theta$. Lemma~\ref{pointwisefixedlines} implies that $L$ is fixed pointwise, including $x^\theta$, a contradiction. The claim is proved. Hence $x,x^\theta,x^{\theta^2}$ define a unique plane $\pi$. 

We claim that all symps of $\Delta$ containing $\pi$ are fixed by $\theta$. Indeed, we again consider $\Res_\Delta(x^\theta)$, which is isomorphic to $\mathsf{E_{6,1}}(\K)$, and $p,p'$ as defined above. This time $p\perp p'$. If some symp $\xi$ of $\Res_\Delta(x^\theta)$ through $p$ is mapped onto a symp $\xi^\theta$ through $p'$ such that $\xi\cap\xi^\theta$ is a singleton, then the same argument as above yields a line through $x^\theta$ fixed pointwise, a contradiction. Hence every symp through $p$ is mapped onto an adjacent symp through $p'$. It is convenient to dualise. Dually, $p$ and $p'$ correspond to adjacent symps $\xi$ and $\xi'$ in $\mathsf{E_{6,6}}(\K)$, intersecting in a $4$-space $U$, and we have to show that each point of $U$ is fixed, knowing that each point of $\xi$ is mapped onto a collinear point in $\xi'$. Let $z$ be a point of $\xi\setminus\xi'$ and set $U_z=\<z,z^\perp\cap U\>$. Let $W_z$ be the unique $5$-space containing $U_z$. Then $W_z\cap\xi'$ is a $4'$-space $U'_z$ through $U\cap z^\perp$. Each point of $\xi'$ collinear to some point of $U_z\setminus U$ is contained in $U'_z$. Hence, since $U_z\setminus U$ generates $U_z$, the subspace $U_z$ is mapped onto $U'_z$. If $u\in \xi\setminus U_z$, then $U_u\cap U_z\subseteq U$ is mapped onto $U'_z\cap U'_u\subseteq U$. It follows that $z^\perp\cap U$ is mapped onto itself. Hence all $3$-spaces of $U$ are stabilised and consequently so is each point of $U$. The claim and the lemma follow. 
\end{proof}

Lemma~\ref{dilemma} permits a case distinction. Either for all points $x\perp x^\theta\neq x$ the line $xx^\theta$ is fixed (the \emph{``quaternion case''}), or there exists a point $z\perp z^\theta\neq z$ with $z^{\theta^2}\notin zz^\theta$ (the \emph{``quadratic case''}).

\subsection{The quaternion case}
In this subsection we assume that for all points $x\perp x^\theta\neq x$ the line $xx^\theta$ is fixed. 

We begin with proving that this in fact implies that no point is mapped onto a collinear one. 

\begin{lemma}\label{nocoll}
No point is mapped onto a collinear one. Hence every fixed singular subspace (in particular, line) is pointwise fixed.
\end{lemma}

\begin{proof}
Suppose for a contradiction that some point $x$ is mapped onto a collinear one. Then, by the main assumption of this subsection, the line $xx^\theta$ is fixed. \cref{pointlinefixed} yields a fixed point on $L$. Now Lemma~\ref{allpointsfixed} implies that $\theta$ fixed each point on $xx^\theta$, hence also $x$, a contradiction.  
\end{proof}

Now define the geometry $(F,\cL_F)$, with $F$ the set of fixed points of $\theta$ and $\cL_F$ the set of fixed lines of $\theta$. By Lemma~\ref{nocoll}, $F$ is a subspace, hence $\cL_F$ is just the set of lines all points of which are contained in $F$. 

\begin{lemma}\label{connected}
The geometry $(F,\cL_F)$ is connected.
\end{lemma}

\begin{proof}
By Lemma~\ref{pointwisefixedlines} there is at least one line $L$ that is pointwise fixed. We show that every point of $F$ is contained in the same connected component as $L$. Let $f\in F$ be arbitrary. Pick $x\in L$ arbitrary and consider the induced action of $\theta$ in $\Res_\Delta(x)$. By Lemma~\ref{nocoll}, no point is mapped onto a collinear one. If also no point is mapped onto a point at distance 2, then $\theta$ induces the identity on $\Res_\Delta(x)$, and so it fixes a chamber, a contradiction. Finally, $\theta$ fixes the point corresponding to the line $L$ in $\Res_\Delta(x)$. So $\theta$ induces a kangaroo collineation in $\Res_\Delta(x)$. In particular, $\theta$ fixes some symp $\xi$ containing $L$. If $f\in\xi$, then any line in $\xi$ joining $f$ with a point of $L$ is pointwise fixed.   If there is a unique point $f'\in\xi$ collinear to $f$, then $f'\in F$ and $ff'$ is fixed pointwise (by Lemma~\ref{nocoll}), and by the foregoing, $f'$ is connected to $L$ inside $F$, hence so is $f$. Finally, if $f$ is collinear to a unique $5$-space $U$ of $\xi$, then $U$ is fixed, and hence it is fixed pointwise by Lemma~\ref{nocoll}. This contradicts Theorem~\ref{kangaroo}. 
\end{proof}

\begin{lemma}\label{GQ}
The fixed point structure induced in any fixed symp is a (quaternion) generalised quadrangle.
\end{lemma}

\begin{proof}
Let $\xi$ be a fixed symp. We first claim that $\xi\cap F\neq\emptyset$. Indeed, let $f\in F$ arbitrary. If $f\in\xi$, then the claim follows. If $f^\perp\cap\xi$ is a singleton, then the claim also follows. So suppose that $f^\perp\cap\xi=U$ is a $5$-space of $\xi$. Then $U$ is pointwise fixed by \cref{nocoll} and the claim follows.

Next we claim that, for each $f\in F$, the collineation $\theta$ induces in $\Res_\Delta(f)$ is a kangaroo collineation $\theta_f$. Indeed, we already noted that  $\theta_f$ does not map points to collinear ones. If $\theta_f$ were the identity, then $\theta$ would fix a chamber of $\Delta$, a contradiction. Hence it remains to show that $\theta_f$ fixes at least one point.  But Lemmas~\ref{pointwisefixedlines} and~\ref{connected} yield a pointwise fixed line through $f$, hence the claim. Note also that, since $\theta$ fixes lines pointwise, $\theta_f$ is a linear kangaroo collineation. 

Hence Theorem~\ref{kangaroo} implies that, if $f\in F\cap \xi$, then the set of fixed points in $f^\perp\cap\xi$ is a cone over a quadric of Witt index 0 in a $5$-dimensional projective space (when viewing $\xi$ as a hyperbolic quadric in $\PG(11,\K)$); so the cone spans a $6$-dimensional projective space. This holds for every point of $F\cap\xi$. Hence $F\cap\xi$ is a full subgeometry of $\xi$ isomorphic to a generalised quadrangle. Since, by the above, the tangent spaces are $6$-spaces, $F\cap\xi$ spans a $7$-space $U_\xi$ and it is not difficult to see that $F\cap\xi=U_\xi\cap\xi$. Since the perp of two noncollinear points is a quadric of Witt index 0 and stems from a quaternion Veronesean (see Theorem~\ref{kangaroo}), $F\cap\xi$ is a quaternion quadrangle. 
\end{proof}

Now let $(\Xi_\theta,\cL_F)$ be the point-line geometry with point set the set $\Xi_\theta$ of fixed symps of $\theta$ and line set $\cL_F$, with natural incidence relation. If we show that $(\Xi_\theta,\cL_F)$ is a polar space, then $(F,\cL_F)$ is a quaternion dual polar space. 

\begin{prop}
The geometry $(\Xi_\theta,\cL_F)$ is a polar space. 
\end{prop}

\begin{proof}
It suffices to show that $(i)$ no point is collinear to all others, and that $(ii)$ for a given point-line pair, either all points on the line are collinear to the given point, or exactly one is. 

{$(i)$ } Let $\xi\in\Xi_\theta$ and let $L\in\cL_F$ with $L\subseteq\xi$. By considering the residue $\Res_\Delta(x)$, for any $x\in L$, we see that there exists $\xi'\in\Xi_\theta$ with $\xi\cap\xi'=L$. By Lemma~\ref{GQ}, there is a line $L'\in\cL_F$ in $\xi'$ opposite $L$ (opposite in the sense of polar spaces). Again we can select $\xi''\in\Xi_\theta$ with $\xi'\cap\xi''=L'$. Then $\xi$ and $\xi''$ are opposite in $\Delta$, and so $\xi$ is not collinear to all points of $\Xi_\theta$.

{$(ii)$ } Now let $(\xi,L)\in\Xi_\theta\times\cL_F$. If $L\subseteq\xi$, then the assertion is trivial.
Now suppose $L\cap\xi=\{x\}$, with $x\in F$. Every symp $\xi'\in\Xi_\theta$ containing $L$ intersects $\xi$ in a fixed line $M$ (use $\Res_\Delta(x)$ to see this). Hence $\xi'\perp\xi$ in $(\Xi_\theta,\cL_F)$.

Finally assume $L\cap\xi=\emptyset$. Since no $5$-space of $\xi$ is fixed, every point of $L$ is collinear to a unique point of $\xi$. This yields a unique fixed symp $\xi'$ containing $L$ and intersecting $\xi$ in a fixed line.
\end{proof}

We record an interesting consequence.
\begin{cor}
No symp is mapped onto an adjacent symp.
\end{cor}
\begin{proof}
Suppose the symps $\xi$ and $\xi^\theta$ are adjacent and let $U=\xi\cap\xi^\theta$. If $U$ is stabilised, then it is fixed pointwise,  contradicting the fact that no subspace larger than a line is contained in a quaternion dual polar space. Hence some point of $U$ is moved by $\theta$, implying that it is mapped onto a symplectic point (by Lemma~\ref{nocoll}). But now by Corollary~\ref{sympfixed}, $\xi^\theta$ is fixed, a contradiction.
\end{proof}

\subsection{The quadratic case}

Here we assume that there is a point $z\perp z^\theta\neq z$ such that $z^{\theta^2}\notin zz^\theta$. By Lemma~\ref{dilemma}, $z^{\theta^2}\perp z$ and the plane $\pi$ generated by $z,z^\theta,z^{\theta^2}$ is fixed by $\theta$, as well as every symp containing $\pi$. 

\begin{lemma}\label{L}
\begin{compactenum}[$(i)$]
\item Every singular subspace containing $\pi$ is fixed. \item The collineation $\theta$ fixes a unique line $L\subseteq \pi$ and a unique point $f\in\pi\setminus L$. \item  Every singular $k$-space $U$ containing $\pi$ contains a pointwise fixed $(k-2)$-dimensional subspace $U'$ disjoint from $L$. \item Every symp containing $L$ is fixed under $\theta$. \end{compactenum}
\end{lemma}

\begin{proof}
{$(i)$ } By Lemma~\ref{dilemma} every symp containing $\pi$ is fixed, hence by taking intersections, every singular subspace containing $\pi$ is fixed.  

{$(ii)$ } First note that $\theta$ does not fix an incident point-line pair in $\pi$, as $(i)$ would lead to a fixed chamber.  Consider now a singular $3$-space $U$ of $\Delta$ containing $\pi$ and pick any point $u\in U\setminus\pi$ not fixed under $\theta$ (if $\theta$ fixes all points of $U\setminus\pi$, then also all points of $\pi$, a contradiction). Then either $\theta$ fixes $uu^\theta$, and hence also the point $f:=uu^\theta\cap\pi$, or it fixes the plane $\alpha$ spanned by $u,u^\theta,u^{\theta^2}$, and hence the line $L:=\pi\cap\alpha$. 

We now claim that, if a linear collineation $\sigma$ in the projective plane $\PG(2,\K)$ fixes a line, then it fixes a point. Indeed, we may assume, up to duality, that $\sigma$ fixes the line $L$ without fixed points. Choose coordinates so that $(1,0,0),(0,1,0)\in L$ and let $(0,0,1)^\sigma=(x_1,x_2,x_2)$, $x_i\in\K$, $i=1,2,3$.  Then the associated matrix $M$ of $\sigma$ can be written as 
$$\left[ \begin{array}{ccc} a & b & 0 \\ c & d & 0\\ x_1 & x_2 & x_3 \end{array}\right], $$ where $\lambda^2-(a+d)\lambda+(ad-bc)=0$ has no $\K$-solutions. Hence $x_3$ is the only $\K$-eigenvalue of $M$ and consequently $\sigma$ fixes a unique point. The claim is proved.

The previous claim now implies that $\theta$ fixes an antiflag $(f,L)$ in $\pi$.  

{$(iii)$ } Let $U$ be a singular $k$-space containing $\pi$, $3\leq k\leq 6$. Then $U^\theta= U$. Pick any point $u\in U\setminus\pi$ and note that the $3$-space $\Sigma$ spanned by $u$ and $\pi$ is fixed by $\theta$.  As before, we may assume that $u\neq u^\theta$. If for all choices of $u$, the points $f,u,u^\theta$ are aligned, then $f$ is a center and hence there is an axis (of fixed points), which intersects $\pi$ in a line, a contradiction. So we may assume that $uu^\theta\cap\pi\neq \{f\}$. It follows by the uniqueness of $f$ that $u,u^\theta,u^{\theta^2}$ determine a plane $\pi'$ which, by uniqueness of $L$, intersects $\pi$ in $L$. As before, $L$ is the unique fixed line of $\pi'$ and hence there is a unique fixed point $f'\notin L$ in $\pi'$. By Lemma~\ref{allpointsfixed}, the line $ff'$ is fixed pointwise. So every $3$-space in $U$ containing $\pi$ contains a unique line pointwise fixed. Since by Lemma~\ref{allpointsfixed}, the set of fixed points in $U$ is a subspace, we see that all these lines generate a $(k-2)$-space $U'$ of $U$ pointwise fixed. 

{$(iv)$ }   Let $\xi$ be any symp containing $L$. Then $f$, being collinear to all points of $L$, is collinear to all points of a certain $5$-space $W$ of $\xi$, and the singular subspace defined by $f$ and $W$ contains $\pi$. Then $(iii)$ implies that $\theta$ fixes every plane $\alpha$ in $W$ containing $L$.   Also, $\alpha$ contains a unique fixed point $f_\alpha\notin L$. It follows easily that for every point $v\in\alpha\setminus(L\cup\{f_\alpha\})$, the line $vv^\theta$ does not contain $v^{\theta^2}$. Hence we can let $\alpha$ play the role of $\pi$; in particular $\theta$ fixes each symp containing $\alpha$, and so $\xi^\theta=\xi$.  
\end{proof}

\begin{lemma}\label{grid}
Let $\xi$ be a symp containing $L$. Then $\xi$ pointwise fixes a subpolar space $Q\subseteq L^\perp$ of type $\mathsf{D_4}$; also $\theta$ fixes all members of the system of generators of the grid $Q^\perp$ which contains $L$. 
\end{lemma}

\begin{proof}
Let $U,U'$ be two singular $5$-spaces in $\xi$ with $U\cap U'=L$. By Lemma~\ref{L}$(iii)$, $U$ and $U'$ contain $3$-spaces $S,S'$, respectively, pointwise fixed under $\theta$, and disjoint from $L$. Then $S$ and $S'$ are opposite in $\xi$. Every point lying on a line which intersects $S\cup S'$ in at least two points is fixed, according to Lemma~\ref{allpointsfixed}. All such points together form a subspace $Q$ of type $\mathsf{D_4}$. Now $\theta$ does not fix any other points, as otherwise it would fix a point of the grid $Q^\perp$, which contains $L$; hence $\theta$ would fix a point on $L$, a contradiction. Hence the first assertion is proved.

Now let $Q'$ be a subspace of $\xi$ of type $\mathsf{D_3}$ containing $Q^\perp$. Then $Q\cap Q'$ contains exactly two points. Applying inverse Klein correspondence to $Q'$, we obtain a map, which we shall also denote by $\theta$, of $\PG(3,\K)$ fixing exactly two disjoint lines $K_1,K_2$, acting without fixed points on $K_1$ and fixing at least one point $a$ on $K_2$. For each line $N$, either $N^\theta$ is disjoint from $N$, or $\theta$ fixes the intersection $N\cap N^\theta$ or the span $\<N,N^\theta\>$. 

Assume for a contradiction that $\theta$ does not fix some point $z_2\in K_2$. Pick any point  $z_1\in K_1$ and pick any point $z\in \<z_1,z_2\>$. Then $\<z^\theta_1,z^\theta_2\>$ is disjoint from $\<z_1,z_2\>$ because otherwise both $K_1$ and $K_2$ would be contained in the plane spanned by $z_1,z_2,z_1^\theta,z^\theta_2$. For a similar reason $\<z,z^\theta\>$ is disjoint from both $K_1$ and $K_2$. Now we claim that $\<z,z^\theta\>^\theta$ is distinct from $\<z,z^\theta\>$. Indeed, if $\<z,z^\theta\>$ was fixed, then so would be the unique line $A$ through $a$ intersecting $\<z,z^\theta\>$ and $K_1$ nontrivially, and hence also $A\cap K_1$, contradicting the fact that $\theta$ acts without fixed points on $K_1$. The claim is proved. Since the intersection $z^\theta=\<z,z^\theta\>\cap\<z,z^\theta\>^\theta$ is not fixed, the plane $\<z,z^\theta,z^{\theta^2}\>$ is fixed (this follows from applying \cref{dilemma} to the intersection point of the planes of $Q'$ corresponding under the Klein correspondence to $z$ and $z^\theta$). But since $\<z,z^\theta\>$ is disjoint from $K_1$, this plane does not contain $K_1$ and hence intersects it in a unique point, which is then fixed, a contradiction. 

We conclude that $\theta$ fixes each point of $K_2$. Translate back to $Q'$, this proves the second assertion. 
\end{proof}
We can now pin down $\theta$.

\begin{prop}
The collineation $\theta$ pointwise fixes the equator geometry defined by two opposite symps $\xi,\xi'$. Also, $\theta$ acts fixed point freely on the full imaginary set $\mathcal{I}$ of symps defined by $\xi$ and $\xi'$ and fixes every line intersecting every member of $\mathcal{I}$. 
\end{prop}

\begin{proof}
Let $\zeta$ be any symp containing the line $L$; then $\zeta^\theta=\zeta$ by Lemma~\ref{L}$(iv)$. By Lemma~\ref{grid}, we find another line $L'$ in $\zeta$ fixed under $\theta$, containing no fixed points, and such that $L'$ is contained in a plane containing exactly one fixed point. Hence everything that holds for $L$ also holds for $L'$. We can then consider a symp $\zeta'$ with $\zeta\cap\zeta'=L'$. We again find a line $L''$ in $\zeta'$ with similar properties as $L$. Note that $L''$ is $\zeta'$-opposite $L'$. It follows that $L$ and $L''$ are opposite in $\Delta$. Pick any point $x\in L$ and let $x''$ be the unique point on $L''$ not opposite $x$. Set $\xi:=\xi(x,x'')$ and $\xi'=\xi^\theta$. Clearly, $\theta$ acts fixed point freely on the  imaginary set $\mathcal{I}$ of symps defined by $\xi$ and $\xi'$. 

Let $y$ be any point in $x^\perp\cap{x''}^\perp$. Let $\zeta$ be the symp defined by $xy$ and $L$, and let $\zeta''$ be the symp defined by $x''y$ and $L''$.  Since $L$ and $L''$ are opposite, the symps $\zeta$ and $\zeta''$ are symplectic and intersect in a unique line $L_y$, with $y\in L_y$. Every point $y_0'$ of $L_y$ is collinear to a point  $y_0\in L$ and a point $y_0''\in L''$; since $y_0$ and $y_0''$ are not opposite, they are contained in the same member of $\mathcal{I}$. Hence $L_y$ intersects each member of $\mathcal{I}$. Note that by Lemma~\ref{L}$(iv)$ the symps $\zeta$ and $\zeta''$ are fixed, and hence so is $L_y$. It follows that $\theta_\xi$ pointwise fixes $x^\perp\cap{x''}^\perp$, and also $x$ and $x''$. This implies that every singular subspace of $\xi$ of dimension at least $2$ is domestic. Since $\theta$ belongs to the opposition diagram $\mathsf{E_{7;4}}$, $\theta_\xi$ is also point-domestic. Consequently $\theta_\xi$ has opposition diagram $\mathsf{D_{6;1}^2}$ unless it is the identity. By  Proposition~3.11 of \cite{PVMclass}, $\theta_\xi$ is either identity or pointwise fixes the perp of a line. he latter contradicts the fixed points of $\theta_\xi$ already found. So $\theta_\xi$ is the identity.  

Now let $p$ be any point of the equator geometry defined by $\mathcal{I}$. Then $p$ is collinear to singular $5$-spaces in the members of $\mathcal{I}$. By the previous paragraph, all these $5$-spaces are mapped onto each other by $\theta$, and hence $p$ is fixed. Hence $\theta$ fixes pointwise the equator geometry defined by the two opposite symps $\xi,\xi'$.
\end{proof}

\subsection{Existence and uniqueness of the quaternion case}  \label{existenceE74}

Let $\HH$ be a|separable or inseparable|quaternion division algebra over $\K$, with standard involution $x\mapsto\overline{x}$.  Let us call a polar space isomorphic to $\mathsf{C_{3,1}}(\HH,\K)$ a \emph{quaternionic polar space}. Recall it is defined by the pseudo-quadratic form $$\overline{X}_{-3}X_3+\overline{X}_{-2}X_2 +\overline{X}_{-1}X_1\in\K$$ and lives in $\PG(5,\HH)$. The quads of the corresponding dual polar space are isomorphic to the polar space $\mathsf{B_{2,1}}(\K,\HH)$ of rank 2 arising from the quadric $Q(\HH)$ in $\PG(7,\K)$ with equation $$X_{-1}X_{-2}+X_{-3}X_{-4}=\mathsf{n}(X_0,X_1,X_2,X_3),$$ with $\mathsf{n}$ the norm mapping in $\HH$. This quadric arises as the intersection of the hyperbolic quadric $Q\cong\mathsf{D_{6,1}}(\K)$ in $\PG(11,\K)$ with equation $X_{-1}X_{-2}+X_{-3}X_{-4}-X_0X_7+X_1X_4+X_2X_5+X_3X_6=0$ with a $7$-dimensional subspace $S$ having equations
$$
X_4 =\ell_1X_1,  \; 
X_6 =-\ell_1\ell_2X_3, \; 
X_{7}=X_0+X_1, \; 
X_5=\ell_2(X_2+X_6) 
$$
in the separable case, and 
$$X_0=X_7,\; X_4=\ell_1X_1,\; X_5=\ell_2 X_2,\; X_6=\ell_1\ell_2 X_3$$ in the inseparable case. This follows from the explicit expressions of the norm form above. 

Now let $p_{-3}$ and $p_{-4}$ be the points with only nonzero coordinate $X_{-3}$ and $X_{-4}$, respectively. Then both points belong to $Q(\HH)$. Let $Q_0(\HH)$ and $Q_0$ be the quadrics contained in $Q(\HH)$ and $Q$, respectively, consisting of the points of $Q(\HH)$ and $Q$, respectively, collinear to both $p_{-3}$ and $p_{-4}$ (hence, $Q_0(\HH)=p_{-3}^\perp\cap p_{-4}^\perp$, with $\perp$ the collinearity relation in $Q(\HH)$, and $Q_0=p_{-3}^{\perp'}\cap p_{-4}^{\perp'}$, with $\perp'$ the collinearity relation in $Q$). Then $U:=\<Q_0\>$ is a $9$-dimensional subspace of $\PG(11,\K)$. 
With this notation, we have the following lemma.
\begin{lemma}\label{extendD5D6}
Let $\varphi_0$ be a collineation of $U$ stabilizing $Q_0$ and fixing $Q_0(\HH)$ pointwise. Then there exists a unique collineation $\varphi$ of $\PG(11,\K)$ stabilizing $Q$, fixing $Q(\HH)$ pointwise and such that $\varphi$ coincides with $\varphi_0$ over $U$.  
\end{lemma}
\begin{proof}
The collineation $\varphi$ fixes the subspace $\<Q(\HH)\>$ pointwise and is determined on the subspace $U$ by $\varphi_0$. Since $\<U,Q(\HH)\>=\PG(11,\K)$ and $U\cap \<Q(\HH)\>=\<Q_0(\HH)\>$, we deduce that $\varphi$ is unique, if it exists. 

Now since $\varphi_0$ fixes $\<Q_0(\HH)\>$ pointwise, $\varphi_0$ stems from a linear transformation of the underlying vector space fixing the vectors corresponding to $\<Q_0(\HH)\>$. Extending $\varphi_0$ by declaring the basis vectors corresponding to $p_{-3}$ and $p_{-4}$ to be fixed, the obtained collineation $\varphi$ satisfies all the stated requirements. 
\end{proof}

We also need the following result.

\begin{lemma}\label{flagsopp}
Let $\{p,L,\xi\}$ be a flag of $\Delta\cong\mathsf{E_{7,7}}(\K)$, with $p$ a point, $L$ a line and $\xi$ a symp. Let $q$ be a point in $\xi$ symplectic to $p$. Let $r$ be a point collinear to $q$ such that the line $rq$ is opposite $\xi$ in the point residual at $q$. Let $\zeta$ be a symp containing $r$ opposite $rq$ in the point residual at $r$, and let $R$ be any line in $\zeta$ containing $r$ and not containing a point that is not opposite all points of $L$. Then $p$ and $r$ are opposite, $L$ and $R$ are opposite, and $\xi$ and $\zeta$ are opposite. 
\end{lemma}

\begin{proof}
The choice of $r$ immediately implies that $r$ is far from $\xi$, hence $\{q\}=r^\perp\cap\xi$ and so Fact~\ref{factE77} implies that $p$ and $r$ are opposite. Suppose now that $\xi$ and $\zeta$ intersect nontrivially. Then the intersection contains a line $M$ which, by assumption, does not contain $q$. Then $r^\perp M\subseteq\xi$ is nonempty, contradicting $r$ being far from $\xi$.  Assume that $\xi\cap\zeta$ is empty but that they are not opposite. Then by Fact~\ref{factE77symp}, and since $r$ is far from $\xi$, the point $q$ is close to $\zeta$, contradicting the choice of $\zeta$ as opposite $rq$ in the point residual at $r$. Hence $\xi$ and $\zeta$ are opposite. Finally, $R$ and $L$ are opposite by the assumptions on $R$. 
\end{proof}

Recall that we denote the dual polar space corresponding to the quaternion polar space by $\mathsf{C_{3,3}}(\HH,\K)$.  The corresponding building is denoted by $\mathsf{C_3}(\HH,\K)$. 
The rest of this section is dedicated to the proof of the following existence result.

\begin{prop}\label{B33inE77}
Suppose the field $\K$ admits a (separable or inseparable) quaternion division algebra $\HH$. Then $\mathsf{C_{3,3}}(\HH,\K)$ is a full subgeometry of $\mathsf{E_{7,7}}(\K)$, unique up to projectivity, such that every point residual of $\mathsf{C_{3,3}}(\HH,\K)$ is a standard inclusion of $\cV(\K,\HH)$ in the corresponding point residual of $\mathsf{E_{7,7}}(\K)$. Also, this embedding arises as the fixed point set of each nontrivial collineation of a group of collineations isomorphic to the group $G$ above. 
\end{prop}
\begin{proof}
Let $\Delta$ be a parapolar space isomorphic to $\mathsf{E_{7,7}}(\K)$ and let $\Gamma$ be a dual polar space isomorphic to $\mathsf{C_{3,3}}(\HH,\K)$. 

Select a chamber of $\Delta$, that is, a set $C$ consisting of a point $p$, a line $L$, a plane $\pi$, a solid $\Sigma$, a maximal $5$-space $W$, a maximal $6$-space $U$ and a symp $\xi$. The fact that these subspaces form a chamber translates in the conditions $p\in L\subseteq \pi\subseteq \Sigma\subseteq W\subseteq \xi$ and $\dim (U\cap W)=4$. These conditions imply that $U\cap\xi$ is a $5$-space. We have represented $C$ on the Coxeter diagram of $\Delta$, using Bourbaki labelling, see Figure~\ref{fig1}. 
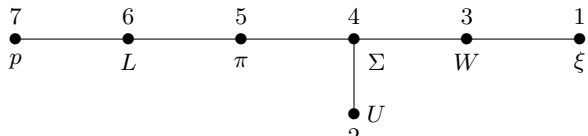
\begin{figure}[ht]
\begin{center}
\begin{tikzpicture}[scale=0.5]
\node at (0,0.3) {};
\node [inner sep=0.8pt,outer sep=0.8pt] at (-6,0) (1) {$\bullet$};
\node [inner sep=0.8pt,outer sep=0.8pt] at (-6,0.6) (1) {\footnotesize $7$};
\node [inner sep=0.8pt,outer sep=0.8pt] at (-6,-0.6) (1) {\footnotesize $p$};

\node [inner sep=0.8pt,outer sep=0.8pt] at (-3,0) (3) {$\bullet$};
\node [inner sep=0.8pt,outer sep=0.8pt] at (-3,0.6) (3) {\footnotesize 6};
\node [inner sep=0.8pt,outer sep=0.8pt] at (-3,-0.6) (3) {\footnotesize $L$};

\node [inner sep=0.8pt,outer sep=0.8pt] at (0,0) (4) {$\bullet$};
\node [inner sep=0.8pt,outer sep=0.8pt] at (0,0.6) (4) {\footnotesize 5};
\node [inner sep=0.8pt,outer sep=0.8pt] at (0,-0.6) (4) {\footnotesize $\pi$};

\node [inner sep=0.8pt,outer sep=0.8pt] at (3,0) (5) {$\bullet$};
\node [inner sep=0.8pt,outer sep=0.8pt] at (3,0.6) (5) {\footnotesize 4};
\node [inner sep=0.8pt,outer sep=0.8pt] at (3.6,-0.6) (5) {\footnotesize $\Sigma$};

\node [inner sep=0.8pt,outer sep=0.8pt] at (6,0) (6) {$\bullet$};
\node [inner sep=0.8pt,outer sep=0.8pt] at (6,0.6) (6) {\footnotesize 3};
\node [inner sep=0.8pt,outer sep=0.8pt] at (6,-0.6) (6) {\footnotesize $W$};

\node [inner sep=0.8pt,outer sep=0.8pt] at (9,0) (7) {$\bullet$};
\node [inner sep=0.8pt,outer sep=0.8pt] at (9,0.6) (7) {\footnotesize 1};
\node [inner sep=0.8pt,outer sep=0.8pt] at (9,-0.6) (7) {\footnotesize $\xi$};

\node [inner sep=0.8pt,outer sep=0.8pt] at (3,-2) (2) {$\bullet$};
\node [inner sep=0.8pt,outer sep=0.8pt] at (3,-2.6) (2) {\footnotesize 2};
\node [inner sep=0.8pt,outer sep=0.8pt] at (3.6,-2) (2) {\footnotesize $U$};

\phantom{\draw [line width=0.5pt,line cap=round,rounded corners] (1.north west)  rectangle (1.south east);}
\phantom{\draw [line width=0.5pt,line cap=round,rounded corners] (6.north west)  rectangle (6.south east);}
\draw (-6,0)--(9,0);
\draw (3,0)--(3,-2);
\end{tikzpicture}
\end{center} \vspace{-0.7cm}
\caption{The Dynkin diagram of $\Delta$ showing the chamber $C$ \label{fig1}}
\vspace{0.3cm}
\end{figure}

The residue $\Res_\Delta(p)$ of $p$ in $\Delta$ is a parapolar space isomorphic to $\mathsf{E_{6,1}}(\K)$. By Proposition~\ref{existE6}  we may select an arbitrary linear kangaroo collineation $\theta_1$ of $\Res_{\Delta}(p)$ pointwise fixing a quaternion Veronese variety $\cV(\K,\HH)$, and such that both $L$ and $\xi$ are fixed by $\theta_1$. Let $G$ be the pointwise stabiliser of $\cV(\K,\HH)$ in $\Res_{\Delta}(p)$. This now implies that the action of $\theta_1$ on $\xi$ satisfies the assumptions of Lemma~\ref{extendD5D6} and so there is a unique collineation $\theta_2$ of $\Res_\Delta(\xi)$, which we identify with the symp $\xi$ itself, such that the action of $\theta_2$ on $\Res_\xi(p)$ coincides with the action of $\theta_1$ on $\Res_\Delta(\{p,\xi\})$. Denote by $\theta_3$ the union of $\theta_1$ and $\theta_2$ (which is a well defined function since $\theta_1$ and $\theta_2$ coincide over the intersection of their domains).  Then the domain of definition of $\theta_3$ contains the union $E_2(C)$ of the rank 2 residues defined by the flags in $C$ of size $5$. Set $C'=\theta_3(C)$. Then $\theta_3:E_2(C)\rightarrow E_2(C')$ is adjacency preserving. 

Set $F=\{p,L,\xi\}\subseteq C$.  Select a line $M$ through $p$ in $\xi$ fixed under $\theta_3$, with $M$ not collinear to all points of $L$. Select a point $q$ in $\xi$ not collinear to $p$ and fixed under $\theta_2$.  Finally, select a line $K$ through $p$, fixed under $\theta_1$ and opposite $\xi$ in $\Res_\Delta(p)$. Considering any apartment $\mathcal{A}_1$ containing $K,M,L,q$ (which exists by grouping these elements together as flags $\{K,\xi'\}$ and $\{q,L',\xi\}$, with $\xi'$ the symp determined by $K$ and $M$, and $L'$ the line in $\xi$ through $q$ meeting $L$ nontrivially, and noting that $\xi\cap\xi'=M$), we can select a point $r$ collinear to the unique point $q$ of $\xi$, and opposite $p$.  In $\mathcal{A}_1$, there is a unique line $R$ (containing $r$) opposite $L$ and a unique symp $\zeta\supseteq R$ opposite $\xi$. Denote by $H$ the flag $\{r,R,\zeta\}$. We claim that there is a unique apartment $\mathcal{A}$ containing $C,\theta_3(C)$ and $H$. Indeed, first we notice that, since $\theta_1$ is a kangaroo, the flag $C\setminus F$ is opposite the flag $\theta_3(C)\setminus F$ in $\Res_\Delta(F)$. Hence, by Proposition~3.29 of~\cite{Tits:74}, the projection $D$ of $C$ onto $H$ is opposite $\theta_3(C)$. Let $\mathcal{A}$ be an apartment determined by $D$ and $\theta_3(C)$. Then $\mathcal{A}$ is unique and contains $C$, since $C$ is the projection of $D$ onto $F$ (reversing the roles of $F$ and $H$). The claim follows. 

Now, similarly, there is a unique apartment $\mathcal{A}'$ containing $\theta_3(C),\theta_3^2(C)$ and $H$. Since both $\mathcal{A}$ and $\mathcal{A}'$ contain $L,M$ and $q$, which are all fixed under $\theta_3$, the intersections of $\mathcal{A}$ and $\mathcal{A}'$ with $\xi$ are independent of the choice of $H$ and this implies that the image under $\theta_3$ of $\mathcal{A}\cap\xi$ is exactly $\mathcal{A}'\cap\xi$. Similarly for $\Res_\Delta(p)$. We can therefore extend $\theta_3$ to $\mathcal{A}$ as the unique isomorphism (or adjacency-preserving map) $\mathcal{A}\rightarrow\mathcal{A}'$ extending the action of $\theta_3$. Note that $\theta_3$ then fixes $H$ elementwise.  

Hence we have everything in place to apply Proposition~4.16 of~\cite{Tits:74}.  We obtain that $\theta_3$ is the restriction of a collineation $\theta$ globally defined on $\Delta$ and coinciding with $\theta_3$ over $E_2(C)\cup\mathcal{A}$. Now $\theta_3$ pointwise fixes the complex arising as the convex closure of $F$ and $H$, which is a Coxeter complex of type $\mathsf{C}_3$. Also, the fixed point structures of $\theta_3$ in $\xi$ and $\Res_\Delta(p)$ imply that the global fixed point structure $\Gamma$ of $\theta$ is a building of type $\mathsf{C}_3$ isomorphic to $\mathsf{C_3}(\HH,\K)$.  

We show that $\Gamma$ is a full subgeometry. All points of $\Delta$ on the line $L$ are fixed by $\theta$. If $L'$ is a line of $\Gamma$ opposite $L$ (so $L'$ is fixed under $\theta$), then each point of $L'$ is the unique point not opposite a certain point of $L$, hence fixed by $\theta$. So all points of $L'$ belong to $\Gamma$. Now each line of $\Gamma$ is opposite some line opposite $L$, so the same argument shows the assertion.

Note that $\Gamma$ is the convex hull of $(E_1(C)\cup\Sigma)\cap\Gamma$ (by the arguments in \S4.3 of \cite{Tits:74}), so that we would obtain the same building $\Gamma$ if we replaced $\theta_1$ with any other nontrivial member of $G$. This shows the last assertion of the Proposition. 

Finally, we prove that $\Gamma$ is unique up to projectivity. Let $\Gamma'\cong\Gamma$ be another full subgeometry with the property that every point residual of $\Gamma'$ is a standard inclusion of $\cV(\K,\HH)$ in the corresponding point residual of $\mathsf{E_{7,7}}(\K)$. By the projective uniqueness of the point residuals, we may assume that $p$ belongs to $\Gamma'$ and that the point residuals at $p$ coincide.  We may also assume (using Lemma~\ref{extendD5D6}) that respective symps of $\Gamma$ and $\Gamma'$ in $\xi$ coincide. In the above arguments, we then skip the selection of $\cA_1$ and instead choose the point $r_0$ arbitrary on a line of $\Gamma'$ through $q$, choose $\zeta_0$ through $r_0$ such that it hosts a symp of $\Gamma'$ not containing the line $r_0q$, and choose $R_0$ any line of $\Gamma'$ through $r_0$ in $\zeta_0$ and $\Gamma'$-opposite $L$. Our assumption on the point residuals being standard embeddings and Lemma~\ref{flagsopp} imply that the distance of points of $\Gamma'$ measured in $\Gamma'$ is the same as the distance measured in $\Delta$. Then Lemma~\ref{flagsopp} implies that $H_0:=\{r_0,R_0,\zeta_0\}$ and $\{p,L,\xi\}$ are opposite. Let $\cA_0$ be the apartment containing $C,\theta(C)$ and $H$.  There exists a collineation $\varphi$ fixing $E_2(C)$ pointwise and mapping $\cA$ to $\cA_0$ (use again Proposition~4.16 of~\cite{Tits:74}). Then $E_2(F)$ in $\Gamma$ coincides with $E_2(F)$ in $\varphi(\Gamma)$ and in $\Gamma'$. Since $H$ represents a chamber of both $\varphi(\Gamma)$ and $\Gamma'$, Proposition~4.1.1 of~\cite{Tits:74} implies that $\varphi(\Gamma)$ and $\Gamma'$ coincide. Whence the assertion.  
\end{proof}

We now provide a more or less direct construction of the group $G$ of \cref{B33inE77} by explicitly establishing it (but without detailed proof). To that aim, we recall the following construction of $\mathsf{E_{7,7}}(\K)$. Let $V$ be a vector space of dimension $56$ over $\K$. Then we define the mapping\\ $\nu:\K\times\K\times\K\times\OO'\times\OO'\times\OO'\rightarrow V: (\ell_1,\ell_2,\ell_3,X_1,X_2,X_3)\mapsto$ 
\begin{eqnarray*}& (1,\ell_1,\ell_2,\ell_3,X_1,X_2,X_3,\\& X_1\overline{X_1}-\ell_2\ell_3, X_2\overline{X_2}-\ell_3\ell_1, 
X_3\overline{X_3}-\ell_1\ell_2,\\ & \ell_1\overline{X}_1-{X_2}X_3, \ell_2 \overline{X}_2 -X_3X_1, \ell_3\overline{X}_3-X_1{X_2},\\ &  \ell_1X_1\overline{X_1}+\ell_2 X_2\overline{X_2}+\ell_3X_3\overline{X_3}-\overline{X_3}(\overline{X}_2\overline{X_1}) 
- (X_1{X_2})X_3-\ell_1\ell_2\ell_3),\end{eqnarray*}
and call this the \emph{dual polar affine octonion Veronese map}. Its image $\mathcal{AV}(\K,\OO')$ is contained in and spans $\PG(V)\cong\PG(55,\K)$. For $|\K|>2$, let $\mathcal{V}(\K,\OO')$ be the projective closure of $\mathcal{AV}(\K,\OO')$ (see \cref{cartan} for the definition). 
We then have

\begin{prop}[see \cite{SSMV}] For $|\K|>2$, the full subgeometry $\mathcal{V}(\K,\OO')$, endowed with all projective lines contained in it, is isomorphic to $\mathsf{E_{7,7}}(\K)$. 
\end{prop}

Now let $\sigma$ be an automorphism of $\OO'$ fixing the center $\K$ pointwise (so that $X\overline{X}=X^\sigma\overline{X}^\sigma$ and $X+\overline{X}=X^\sigma+\overline{X}^\sigma$). Then clearly the following linear map \begin{eqnarray*}\varphi_\sigma:&V\rightarrow V: &(x,x_1,x_2,x_3,X_1,X_2,X_3,y_1,y_2,y_3,Y_1,Y_2,Y_3,y)\mapsto\\&&(x,x_1,x_2,x_3,X_1^\sigma,X_2^\sigma,X_3^\sigma,y_1,y_2,y_3,Y_1^\sigma,Y_2^\sigma,Y_3^\sigma,y),\end{eqnarray*}
with $x,x_i,y,y_i\in\K$, $X_i,Y_i\in\OO'$, $i=1,2,3$, defines an automorphism, which we also denote by $\varphi_\sigma$, of $\mathcal{AV}(\K,\OO')$, and hence of $\mathcal{V}(\K,\OO')$. Now we have the following proposition, the proof of which is left to the reader.

\begin{prop}
Suppose that $\sigma$ is an automorphism of $\OO'$ the fixed point structure of which is a separable or inseparable quaternion division subalgebra $\HH$. Then the fix structure in $\mathcal{V}(\K,\OO')$ of the collineation $\varphi_\sigma$ is a full subgeometry of $\mathcal{V}(\K,\OO')\cong\mathsf{E_{7,7}}(\K)$ isomorphic to the dual polar space $\mathsf{C_{3,3}}(\HH,\K)$.
\end{prop}


\section{Collineations with opposition diagram $\sE_{7;3}$ fixing no chamber}\label{sec:E73nochamber}
\subsection{The examples}
By Proposition~5.7 of \cite{DSV}, each quadratic extension $\LL$ of $\K$ yields a subbuilding $\mathfrak{B}$ of $\mathsf{E_7}(\K)$ isomorphic to $\mathsf{F_4}(\K,\LL)$ with the following properties. 

\begin{compactenum}[$(i)$]
\item The associated Lie incidence geometry $\Gamma^*:=\mathsf{F_{4,1}}(\K,\LL)$ is a full subgeometry of the Lie incidence geometry $\Delta^*:=\mathsf{E_{7,1}}(\K)$.
\item Each symp of $\Gamma^*$ is embedded in a unique symp of $\Delta$.
\item The embedding is \emph{isometric}, that is, a pair of points $p,q$ in $\Gamma^*$ is collinear, symplectic, special or opposite in $\Gamma^*$ if, and only if, it is collinear, symplectic, special or opposite, respectively, in $\Delta^*$. 
\item There is a group $G$ isomorphic to $\LL^\times/\K^\times$ of automorphisms of $\Delta^*$ each nontrivial member of which has $\Gamma^*$ as fix structure. If some automorphism of $\Delta^*$ pointwise fixes $\Gamma^*$, then it belongs to $G$. 
\end{compactenum}

Conversely, if, for some quadratic extension $\LL$ of $\K$, the Lie incidence geometry $\Delta^*$ contains a full subgeometry isomorphic to the Lie incidence geometry $\mathsf{F_{4,1}}(\K,\LL)$, then, up to a projectivity, it is precisely the corresponding point-line geometry of $\mathfrak{B}$.  Note that $\LL/\K$ is allowed to be an inseparable extension and also $|\K|=2$ is allowed. 

We now show the following proposition.

\begin{prop}\label{E73isdom}  Given a full subgeometry of $\Delta^*$ isomorphic to the Lie incidence geometry $\Gamma^*$  with corresponding group $G$ as above, then every nontrivial member $\theta$ of $G$ is domestic and has opposition diagram  $\mathsf{E_{7;3}}$.
\end{prop}

It is convenient to consider the ``dual'' situation, that is, to consider $\Gamma=\mathsf{F_{4,4}}(\K,\LL)$ inside $\Delta=\mathsf{E_{7,7,}}(\K)$. Then it follows from \cite{DSV} that  the fixdiagram of $\theta$ is $\mathsf{E_{7;4}}$ and that the collineation induced by $\theta$ in any fixed symp of $\Delta$ is point-domestic and does not fix a chamber. We state this and some consequences as a lemma for further reference.

\begin{lemma}\label{D6fix}
Let $\xi$ be a symp of $\Delta$ fixed by $\theta$. Then $\theta$ induces a point-domestic collineation in $\xi$ without fixed points. It has both fixdiagram and opposition diagram $\mathsf{D_{6;3}^2}$. Also, 
\begin{compactenum}[$(i)$] \item The set of fixed lines is a spread of $\xi$; \item Two fixed lines of $\xi$ are either contained in a common (fixed) $3$-space, or are $\xi$-opposite. \item All $5$-spaces of $\xi$ that are maximal singular subspaces of $\Delta$ that contain a fixed $3$-space, are fixed themselves by $\theta$.  \item A $5'$-space $U$ of $\xi$ is mapped to a $5'$-space $U^\theta$ that is either disjoint from $U$ or intersects $U$ in a fixed $3$-space.
\item No $5$-space of $\xi$ is mapped onto a disjoint one. 
\end{compactenum}
\end{lemma}

\begin{proof} First note that fixed $5$-dimensional singular subspaces of $\xi$ are maximal singular subspaces (hence $5$-spaces and not $5'$-spaces) since they are contained in at least two fixed symps (by assumption of the fix structure). 

Now we note that 
$(i)$ follows from point-domesticity and Proposition~3.1 of \cite{PVMclass}. Heading for $(ii)$, let $L_1$ and $L_2$ be two arbitrary fixed lines of $\xi$ and assume $L_1$ is not opposite $L_2$. Then some point $p_1\in L_1$ is collinear to all points of $L_2$. Then also $p_1^\theta$ is collinear to all points of $L_2$, and so $L_1$ and $L_2$ are contained in an automatically fixed $3$-space, proving $(ii)$.

Now assume that a $5$-dimensional subspace $U$ containing a fixed $3$-space $S$ is not fixed.   Let $p\in U\setminus S$. Since $U$ is not fixed, $p^\theta\notin U$. But $\<S,p,p^\theta\>$ is a fixed $5$-space, intersecting $U$ in the $4$-space $\<S,p\>$. Hence $U$ is not maximal in $\Delta$, proving $(iii)$.

Now we prove $(iv)$. By $(i)$, the collineation $\theta$ induces a line-spread in $\xi$. Suppose first that $U$ does not contain a line of this spread, then $U$ and $U^\theta$ are disjoint, as $\theta$ has no fixed points. Thus we can suppose that $U$ contains a line $R$ of the spread. Choose an arbitrary point $p\in U\setminus R$, let $L$ be the line of the spread that contains $p$. Since $p$ and $R$ lie in a singular subspace, it follows from $(ii)$ that $L$ and $R$ are collinear and thus span a $3$-space $S$, that is fixed by $\theta$.

If $L$ and thus also $S$ lie in $U$, then $U\cap U^\theta$ contains $S$, a fixed $3$-space.
Thus suppose $L$ does not lie in $U$. Then we consider the singular $5$-space $W$ spanned by $S$ and the points of $U$ that are collinear to all points of $S$. This is a $5$-space that intersects $U$ in a $4$-space, thus $W$ is a maximal singular subspace that contains $S$. It now follows from $(iii)$ that $W$ is a fixed space. 

Since $W$ and $U$ intersect in a hyperplane, it suffices to show that every $4$-space $A$ of $W$ contains a fixed $3$-space. One verifies that $A\cap A^\theta$ is fixed, proving $(iv)$.

In order to prove $(v)$ we consider two opposite fixed $5$-spaces $U$ and $U'$. It is easy to find a plane $\alpha\subseteq U$ containing no fixed line. Then $\alpha$ and $\alpha^\theta$ are disjoint and the $5'$-space spanned by $\alpha$ and $\alpha^\perp\cap U'$ is mapped onto an opposite. By the opposition diagram, no $5$-space is mapped onto an opposite.
\end{proof}


\begin{lemma}\label{pointsymp}
For every point $p$ of $\Delta$, there is a symp $\xi$ of $\Delta$ close to $p$ fixed by $\theta$.
\end{lemma}
\begin{proof}
Choose an arbitrary fixed symp $\zeta$ and suppose $p$ is collinear to a unique point $x$ in $\zeta$ (otherwise, by \cref{factE77}, there is nothing to prove). In $\zeta$ there is a fixed $5$-space $U$ through $x$ and all symps through $U$ are fixed by $\theta$. We claim that one of these symps contains a $5$-space collinear to $p$.

To prove this, we look at the residue $\Res_\Delta(x)$ at $x$: in this geometry of type $\mathsf{E_{6,1}}$, we have a point $p$ and a $4$-space $U$ such that $p$ is not collinear to any point of $U$, and we have to prove that there exists a symp through $U$ that is close to $p$. Because this residue is self-dual, it suffices to show the dual of this: if we have a symp $P'$ and a line $U'$, there exists a point on $U'$ that is close to $P'$. Each line is contained in a symp and every two symps intersect in at least a point, thus we find a symp $\xi$ through $U'$ that intersects $P'$. The point $y$ in this intersection is collinear to a point $q$ of $U'$, thus $q$ has to be close to $P'$.
\end{proof}

\begin{lemma}\label{notsympl}
A point $p$ of $\Delta$ is either mapped to an opposite point, or to a collinear point. As a consequence, no point of $\Delta$ is mapped to a point at distance $2$.
\end{lemma}
\begin{proof}
From Lemma~\ref{pointsymp} we know that $p$ is collinear to a $5$-space $U$ of a fixed symp $\xi$. From lemma \ref{D6fix}$(iv)$ follows that $U \cap U^\theta$ is a fixed $3$-space $S$ or $U$ is opposite $U^\theta$.

Suppose first that $U\cap U^\theta=S$, then $p$ and $p^\theta$ are both collinear to $S$. Thus $p$ and $p^\theta$ can not be opposite. If $p$ and $p^\theta$ are symplectic, then the symp $\xi(p,p^\theta)$ determined by $p$ and $p^\theta$ intersects $\xi$ in a $5$-space that contains $S$. Since $\Gamma^*$ is a full subgeometry of $\Delta^*$, this $5$-space through a fixed $3$-space has to be fixed, and the symp $\xi(p,p^\theta)$ through this fixed $5$-space also has to be fixed. Thus $\theta$ induces a spread in $\xi(p,p^\theta)$ and every point of this symp has to be mapped to a collinear point. Hence $p$ and $p^\theta$ can not be symplectic. It follows that $p$ and $p^\theta$ are collinear.

Suppose now that $U$ and $U^\theta$ are disjoint. Then it is easy to see that $p$ and $p^\theta$ can not be collinear (a symp through $p$ and a point of $U^\theta$ then has disjoint $4$-spaces lying in $U\cup U^\theta$ in common with $\xi$, a contradiction). Suppose they are symplectic and consider a point $q$ that is collinear to both $p$ and $p^\theta$ (note $q$ does not lie in $\xi$). The point $q$ has to be collinear to at least one point $y$ of $\xi$. If $y\notin U$, then $p$ and $y$ determine a symp that contains $q$ and a $4$-space of $U$. So, either $q$ is collinear to a unique point $y$ of $\xi$ that lies in $U$, or $q$ is collinear to a $3$-space in $U$. We can repeat this reasoning with $U^\theta$ instead of $U$. We conclude that $q$ should be collinear to a $3$-space in $U$ and a $3$-space in $U^\theta$. This is a contradiction, $q$ can only be collinear to a $5$-space of $\xi$. Thus $p$ and $p^\theta$ are opposite.
\end{proof}

\begin{lemma}
In every plane $\alpha$ of $\Delta$, there is a point that is mapped to a collinear point. In particular, no plane is mapped to an opposite plane by $\theta$.
\end{lemma}
\begin{proof}
Choosing a point $p$ in $\alpha$, we can suppose that $p$ and $p^\theta$ are opposite. \cref{pointsymp} yields a fixed symp $\xi_1$ close to $p$ and $U:=p^\perp\cap \xi_1$ is disjoint from $U^\theta$. Now choose a point $q\neq p$ in $\alpha$. Then one verifies that $q$ is either collinear to a $5$-space of $\xi_1$ that intersects $U$ in a $3$-space, or $q$ is collinear to a unique point $x$ of $\xi_1$ that lies in $U$. In the first case we set $\xi_2:=\xi_1$, in the second case we find $\xi_2$ as follows.
Choose a point $y$ in $U$ that is collinear to $x^\theta$. Then $x,y,x^\theta$ span a plane and thus, by Lemma~\ref{D6fix}$(ii)$, the points  $x,x^\theta,y,y^\theta$ span a fixed $3$-space $S$.

Choose a fixed $5$-space through $S$. From the proof of Lemma~\ref{pointsymp}, we know that there is a fixed symp $\xi_2$ through this $5$-space such that $q$ is collinear with a $5$-space of $\xi_2$. Also the point $p$ is collinear to a $5$-space of $\xi_2$, as $p$ is collinear to the line $\langle x, y \rangle$ in $S$, which is contained in $\xi_2$. In both cases we now have a symp $\xi_2$ such that $p$ and $q$ are collinear to $5$-spaces in $\xi_2$. Because $p$ and $q$ are collinear, these $5$-spaces intersect in a $3$-space $A$.

Consider a third point $r$ on the plane $\alpha$, that does not lie on the line $pq$. If $r$ is collinear to a $5$-space $R$ in $\xi_2$ then we set $\xi_3:=\xi_2$.
We can now suppose that $r$ is collinear to a unique point of $\xi_2$, which has to belong to $A$. 
Choose a point $b$ in $A\setminus {a}$ that is collinear to $a^\theta$. Now we find a fixed $3$-space $F$ spanned by $a,b,a^\theta,b^\theta$ as before.
We choose an arbitrary fixed $5$-space through $F$ and again as in the proof of Lemma~\ref{pointsymp} we find a fixed symp $\xi_3$ through this $5$-space, such that $r$ is collinear to a $5$-space of $\xi_3$. The points $p$ and $q$ are collinear to the line $\langle a,b\rangle$, thus they are also collinear to a $5$-space of $\xi_3$.

We now have a fixed symp $\xi_3$ such that $p$, $q$ and $r$ are all collinear to a $5$-space ($P$, $Q$ and $R$, respectively) in $\xi_3$. The $5$-space $R$ intersects both $P$ and $Q$ in a $3$-space distinct from $P\cap Q$, since otherwise we have two different $6$-spaces---the one spanned by $p$ and $P$ and the one spanned by $\alpha$ and $P\cap Q$---intersecting in a $4$-space, contradicting \cref{factE6}$(vi)$ in $\Res_\Delta(p)$.

Consequently we may suppose that there exists a point $a\in (P\cup Q)\setminus R$. The symp $\xi(a,r)$ contains $p,q,r$ and a $4$-space of $R$. The points $p$ and $q$ are collinear to at least a common plane of that $4$-space, and so we conclude with the previous paragraph that  $P\cap Q\cap R$ is a plane $\beta$ and $P\cap Q$, $Q\cap R$ and $R\cap P$ span a $5$-space $W$ containing $\beta$.

Now each point of the plane $\alpha$ is collinear to a $5'$-space containing $\beta$, that intersects $W$ in a $4$-space, an conversely, every $5'$-space containing $\beta$ and intersecting $W$ in a $4$-space is collinear to a point of $\alpha$.

If $\beta$ and $\beta^\theta$ had a point in common, then every point of $\alpha$ would go to a collinear point. Thus we suppose $\beta$ and $\beta^\theta$ are disjoint. 

Now $W\cap W^\theta$ is nonempty by \cref{D6fix}$(v)$. Let $x\in W\cap W^\theta$. Then $x^\theta\in W^\theta$ and $x^{\theta^{-1}}\in W$, hence $L:=xx^{\theta^{-1}}=xx^\theta\subseteq W\cap W^\theta$. Since $L$ and $\beta$ generate at most a $4$-space, we find a $5'$-space $Z$ containing $\beta$ and $L$, hence intersecting $W$ in at least a $3$-space (as $L$ and $\beta$ generate at least a $3$-space---indeed, $L$ is not contained in $\beta$ by the previous paragraph). Consequently there is a point $z\in\alpha$ collinear to $Z$. Since $L\subseteq Z\cap Z^\theta$, the point $z$ is not opposite $z^\theta$ and hence $z\perp z^\theta$ by \cref{notsympl}.
\end{proof}

It now follows that $\theta$ has opposition diagram $\mathsf{E_{7;3}}$ since the previous lemma implies that $\theta$ is domestic, and since we know that $\theta$ maps some point to an opposite point (namely, any point in a $6$-space containing a $5$-space $U$ in a fixed symp with the property that $U$ and $U^\theta$ are disjoint). Proposition~\ref{E73isdom} is proved.

\begin{remark}\label{remark02}Note that, in the above proofs, we only used the hypotheses that there are no fixed even-dimensional subspaces, that there are fixed lines, that every maximal $5$-space and every symp through a fixed $3$-space is also fixed, and that the collineation induced in a fixed symp is point-domestic. 
\end{remark}

\subsection{The characterisation}
In this section, let $\theta$ be a collineation of $\Delta=\mathsf{E_{7,7}}(\K)$ with opposition diagram $\mathsf{E_{7;3}}$ fixing no chamber (however some lemmas below are independent of the latter condition). Again, since we restrict to large buildings we may assume that $|\K|>2$. 

Our aim is to show that the fix structure of $\theta$ is a subbuilding $\mathsf{F_{4}}(\K,\LL)$, with $\LL$ a quadratic extension of $\K$, such that the corresponding geometry $\mathsf{F_{4,1}}(\K,\LL)$ is a full subgeometry of $\mathsf{E_{7,1}}(\K)$.

\begin{lemma}\label{notspecial}
For each symp $\xi$ of $\Delta$, the pair $\{\xi,\xi^\theta\}$ is not special.
\end{lemma}

\begin{proof}
Suppose for a contradiction that $\{\xi,\xi^\theta\}$ is a special pair of symps. 
Let $U\subseteq\xi$ and $U'\subseteq\xi^\theta$ be the unique $5$-spaces contained in a common symp. Select a $5$-space $W$ in $\xi$ disjoint from both $U$ and ${U'}^{\theta^{-1}}$ (this exists by Proposition~3.30 of \cite{Tits:74}). Then $W^\theta$ is disjoint from $U'$ and hence is not contained in the perp of any point of $U'$. It follows from \cref{factE77symp}$(iv)$ combined with \cref{factE77}$(i)$ that each point of $W$ is opposite some point of $W^\theta$.  Hence, by \cref{oppositionU}, $W$ is opposite $W^\theta$, contradicting the opposition diagram. 
\end{proof}

\begin{lemma}\label{colllinefixed}
If a point $x$ is mapped onto a collinear point $x^\theta$, then the line $xx^\theta$ is fixed. 
\end{lemma}

\begin{proof}
Suppose not. By the dual of \cref{closefar}, there is a symp $\xi^\theta$ locally opposite the line $x^\theta x^{\theta^2}$ at $x^\theta$ and locally close to the line $xx^\theta$ at $x^\theta$. It follows that $\xi$ is locally opposite $xx^\theta$ at $x$. We claim that $\{\xi,\xi^\theta\}$ is special. 

Since $\xi^\theta$ is locally close to $xx^\theta$ at $x^\theta$, the perp $x^\perp$ intersects $\xi^\theta$ in a maximal singular subspace. Since $x\in\xi$, this implies that $\xi$ and $\xi^\theta$ are not opposite. Now suppose that $\xi$ and $\xi^\theta$ are not disjoint. Then, by \cref{factE77}, they share at least a line, and there exists some point $p\in\xi\cap\xi^\theta$ collinear to $x^\theta$. Hence $p\perp x$ and $xx^\theta$ and $xp$ are collinear points in $\Res_\Delta(x)$, contradicting the fact that $xx^\theta$ is far from $\xi$ in $\Res_\Delta(x)$.  This shows the claim.

But now we run against a contradiction to Lemma~\ref{notspecial}. 
\end{proof}

\begin{lemma}\label{linefixed}
If for a  symp $\xi$ the pair $\{\xi,\xi^\theta\}$ is symplectic, then the line $\xi\cap\xi^\theta$ is fixed. 
\end{lemma}

\begin{proof}
Set $\xi\cap\xi^\theta=:L$.

We show first that we may assume that some point $x\in\xi$ is mapped onto an opposite. Assume for a contradiction that no point of $\xi$ is mapped onto an opposite one; then the set $Y$ of points of $\xi$ collinear to a point $y\in L$ but not in $L^\perp$ is mapped into the set of points of $\xi^\theta$ collinear to $y$, but also into the set of points of $\xi^\theta$ collinear to $y^\theta$, hence into the set $Y'$ of points of $\xi^\theta$ collinear to both $y$ and $y^\theta$.  If $y$ and $y^\theta$ are not collinear, then $Y$ contains some affine part of a $5$-space, whereas $Y'$ does not, a contradiction. If $y\perp y^\theta$, with $y\neq y^\theta$, then $Y'$ does not contain pairs of such affine $5$-spaces the projective completion of which intersects precisely in a point (they always contain the line $yy^\theta$), whereas $Y$ does (in the perp of $y$ one finds $5$-spaces only intersecting in $y$).    It follows that $y=y^\theta$, hence $L$ is fixed pointwise. 

So we may assume that there exists a non-domestic point $x\in\xi$. If $M$ is the line in $\xi$ through $x$ intersecting $L$ nontrivially (say, in the point $u$), then $\xi^{\theta_x}=M$.  Lemma~4.18 of~\cite{PVM} implies that the line $K$ through $u$ intersecting $M^\theta$ nontrivially is fixed. Since $K\subseteq \xi^\theta$, we have $K=K^{\theta^{-1}}\subseteq\xi$, and so $K=L$.
\end{proof}

\begin{lemma}\label{oppsymp}
Suppose at least one fixed point exists. If a symp $\xi$ is non-domestic, then it does not contain points $x$ with $\{x,x^\theta\}$ symplectic.  
\end{lemma}

\begin{proof}
Let $\xi$ be opposite $\xi^\theta$. The opposition diagram of $\theta_\xi$ is a ``subset'' of the residue of a node of type $1$ in $\mathsf{E_{7;3}}$, that is, $\mathsf{D_{6;0}}$, $\mathsf{D_{6;1}^1}$, $\mathsf{D_{6;1}^2}$ or $\mathsf{D_{6;2}^1}$.  Then Theorem~6.1 of \cite{TTVM} yields noncollinear fixed points $x_1,x_2$ of $\theta_\xi$. The lines $x_1x_1^\theta$ and $x_2x_2^\theta$ are fixed lines by \cref{colllinefixed}; Lemma~\ref{pointlinefixed} yields a fixed point $p_1$ on $x_1x_1^\theta$. The unique symp $\xi'$ through $p_1$ intersecting $x_2x_2^\theta$ nontrivially (see \cref{imE7}) is itself fixed and hence intersects $x_2x_2^\theta$ in a fixed point $p_2$. Moreover, if $x$ is any point of $\xi$ fixed under $\theta_\xi$, then the unique point $x'$ of $\xi'$ collinear to $x$ is fixed under $\theta$. Hence the fixed point structures of $\theta_\xi$ and of $\theta$ in $\xi'$ are isomorphic. Now suppose, for a contradiction, that some point $y\in\xi'$ is mapped onto a collinear but distinct point $y^\theta$. Let $U$ be a $5$-space containing $yy^\theta$. Theorem~6.1 of \cite{TTVM} yields a $3$-space $S\subseteq U$ pointwise fixed by $\theta$. \cref{pointlinefixed} asserts that $yy^\theta$ is stabilised and \cref{colllinefixed} yeilds a fixed point $z\in yy^\theta$. If $z\in S$, then $\<S,y\>$ is a fixed $4$-space and $U$ a fixed $5$-space. This easily yields a fixed chamber, a contradiction. If $z\notin S$, then $\<S,z\>$ is a fixed $4$-space and this leads to the same contradiction. Hence no point of $\xi'$ is mapped to a collinear one, meaning that for each point $x\in\xi$, the pair $\{x,x^\theta\}$ is either collinear or opposite, but not symplectic.
\end{proof}

\begin{lemma}\label{allpoints}
If a fixed line $L$ contains a fixed point $p$, then all points on $L$ are fixed. 
\end{lemma}

\begin{proof}
Suppose some point $x\in L$ is not fixed. Assume for a contradiction that some symp $\xi$ through $L$ intersects $\xi^\theta$ precisely in $L$. Let $y$ be a point in $\xi$ collinear to $x$ but not to $p$. Then $y$ is non-domestic. Let $M$ be a line in $\xi$ through $y$ not contained in $x^\perp$. Then $M$ contains a point $u$ mapped to a point at distance 2 (namely, the point on $M$ collinear to $p$). Now $\theta_y$ is a domestic duality in $\Res_\Delta(y)$; by the main result of \cite{Mal:12} and Section 3.2 of \cite{Sch-Sch-Mal:24}, there exists some symp $\zeta$ through $M$ mapped to an opposite $\zeta^\theta$. But $u\in\zeta$, contradicting Lemma~\ref{oppsymp}. 

Hence the map induced on $\Res_\Delta(L)$ is point-domestic on the irreducible component of type $\mathsf{D_5}$, viewed as polar space $\mathsf{D_{5,1}}(\K)$.  Lemma~3.16 of \cite{PVMclass}  implies that some chamber $C$ of that residue is fixed. But now $C\cup\{p,L\}$ is a fixed chamber of $\Delta$, a contradiction.  
\end{proof}

We are now ready to show that there are no fixed points.

\begin{lemma}\label{nofixedpoints}
There are no fixed points.
\end{lemma}

\begin{proof}
The opposition diagram ensures that some symp $\xi$ is mapped onto an opposite. Then, as in the proof of \cref{oppsymp}, Theorem~6.1 of \cite{TTVM} yields at least one fixed point of $\theta_\xi$. Hence we find a point $p$ mapped to a collinear point and the line $pp^\theta$ is fixed by Lemma~\ref{colllinefixed}. If there was some fixed point of $\theta$, then Lemma~\ref{pointlinefixed} would yield a fixed point on the line $pp^\theta$, and the latter is then fixed pointwise by Lemma~\ref{allpoints}. Clearly, this contradicts $p$ not being fixed. 
\end{proof}

We now prepare for proving Lemma~\ref{oppsymp} without the condition of the existence of a fixed point. 

\begin{lemma}\label{sympno2}
If a symp $\xi$ is fixed, then $\theta$ acts point-domestically on it. 
\end{lemma}

\begin{proof}
Suppose for a contradiction that $p$ is a point of $\xi$ mapped onto a point at distance 2. Let $\sigma$ be the restriction of $\theta$ to $\xi$. Then we claim that $\sigma_p$ has no fixed points. 

Indeed, if $L$ through $p$ is mapped onto $L^\sigma$ through $p^\sigma$ with $L\cap L^\sigma={x}$, then $x$ is mapped onto a point of $L^\sigma$ (remember $\sigma$ has no fixed points by Lemma~\ref{nofixedpoints}). But then Lemma~\ref{colllinefixed} implies that $L^\theta=xx^\theta$ is fixed, a contradiction. 

Next, we claim that each line through $p$ in $\xi$ is mapped onto an opposite line in $\xi$. Indeed, suppose not, let $L$ be a line through $p$ not mapped onto an opposite in $\xi$ and consider a symp $\zeta$ which intersects $\xi$ in just $L$. Then $\zeta^\theta$ intersects $\xi$ in just $L^\theta$. Since some point of $L\subseteq\zeta$ is collinear to all points of $L^\theta\subseteq\zeta^\theta$, the pair $\{\zeta,\zeta^\theta\}$ is not opposite. It is not special either by Lemma~\ref{notspecial}, and is not an adjacent pair since this would imply that $p^\perp\cap{p^\theta}^\perp$ contains a plane in $\zeta\cap\zeta^\theta\cap\xi$. Hence they are symplectic and $\zeta\cap\zeta^\theta=M$ is a line. If $p\perp M$, then also $p^\theta\perp M^\theta=M$ (by Lemma~\ref{linefixed}) and $p^\perp\cap{p^\theta}^\perp$ contains a line in $\zeta\cap\zeta^\theta\cap\xi$, which then necessarily coincides with $L$ and $L^\theta$, a contradiction.  Hence $p$ is collinear to a unique point of $L$, which is not fixed by Lemma~\ref{nofixedpoints}. This implies that $p$ is opposite $p^\theta$ in $\Delta$, a contradiction. The claim is proved.

Now let $U$ be an arbitrary $5$-space in $\xi$ containing $p$. Suppose for a contradiction that $U\cap U^\theta$ contains a point $u$. Then $u$ is collinear to $(pu)^\theta$, contradicting the previous claim. Hence $U$ is disjoint from $U^\theta$.  Let $\zeta$ be a symp intersecting $\xi$ in $U$. Then $\{\zeta,\zeta^\theta\}$ is not opposite, adjacent or symplectic---a point in the intersection would have to be collinear to the span of two $4$-spaces (one in each of $U$ and $U^\theta$) of $\xi$, clearly impossible. We conclude that $\{\zeta,\zeta^\theta\}$ is special, but this contradicts Lemma~\ref{notspecial}. 

This completes the proof of the lemma.
\end{proof}

\begin{lemma}\label{oppsympnofixed}
If a symp $\xi$ is non-domestic, then it does not contain points $x$ with $\{x,x^\theta\}$ symplectic.  
\end{lemma}

\begin{proof}
Let $\xi$ be a non-domestic symp. Suppose for a contradiction that $\theta_\xi$ maps a point $x$ to a collinear (distinct) point $x'$. We first claim that $xx'$ is stabilised by $\theta_\xi$. If $xx'$ contains a fixed point, then this is obvious. If not, then consider two $5$-spaces $U,U'$ through $xx'$ intersecting precisely in $xx'$. By the opposition diagram, $\theta_\xi$ is plane-domestic and solid-domestic (solids being singular $3$-spaces), hence Theorem~6.1 of \cite{TTVM} yields pointwise fixed solids $S\subseteq U$ and $S'\subseteq U'$, which, by assumption, are disjoint from $xx'$. Since $\theta_\xi$ is type preserving, it stabilised both $U$ and $U'$ and hence $xx'$. The claim is proved.  

Next we claim that $xx'$ contains a fixed point. Indeed, let $\pi$ be any plane through $xx'$. Theorem 6.1 of \cite{TTVM} yields a fixed point $f\in\pi$. If there is a second fixed point $f'$ in $\pi\setminus xx'$, then $ff'\cap xx'$ is a fixed point $p$ for $\theta_\xi$. If there is no further fixed point in $\pi\setminus xx'$, then for an arbitrary point $y\in\pi\setminus (xx'\cup\{f\})$ the line $yy^\theta$ is fixed (applying the previous with $y$ in place of $x$)  and so $yy^\theta\cap xx'$ again defines a fixed point $p$ on $xx'$. The claim is proved.

Now the symp $\xi(x^\theta,p)=\xi(x,p^\theta)$ is fixed under $\theta$ (as $p^{\theta^2}$ is contained in $pp^\theta$ by Lemma~\ref{colllinefixed}). But this now contradicts Lemma~\ref{sympno2}.
\end{proof}

\begin{lemma}\label{notsymplectic}
If $x$ is a point, then $\{x,x^\theta\}$ is not symplectic.
\end{lemma}

\begin{proof}
Suppose for a contradiction that $x$ is mapped onto a point at distance 2 from $x$. Let $\xi_x$ be the symp through $x,x^\theta$. We select a symp $\xi^\theta$ through $x^\theta$ such that in $\Res(x^\theta)$, the symp $\xi^\theta$ is far from an arbitrary  line through $x^\theta$ in $\xi_x\cap\xi_x^\theta$; then $\xi$ and $\xi_x$ intersect in precisely a line $L$ just like $\xi^\theta$ and $\xi_x$ also just intersect in some line $M$. 

Now since $\xi\neq\xi_x$, we have $\xi\neq\xi^\theta$; also the pair $\{\xi,\xi^\theta\}$ is not opposite by Lemma~\ref{oppsympnofixed}, it is not special by Lemma~\ref{notspecial}, and it is not adjacent since otherwise each of $\xi$ and $\xi^\theta$ shares the subspace $x^\perp\cap {x^\theta}^\perp\cap \xi\cap\xi^\theta$ of dimension at least 3 with $\xi_x$, a contradiction to our choice of $\xi$. Hence $\xi$ and $\xi^\theta$ share exactly a line $K$, which is preserved under $\theta$ by Lemma~\ref{linefixed}. Assume for a contradiction that $x$ is collinear to a unique point of $K$.  By Lemma~\ref{nofixedpoints}, $x^\theta$ is not collinear to that point and so $\{x,x^\theta\}$ is, by \cref{factE77symp}$(iii)$, an opposite pair, a contradiction. Hence both $x$ and $x^\theta$ are collinear to all points of $K$. This yields $K\subseteq \xi_x\cap\xi\cap\xi^\theta$, implying $K=L=M$, which is clearly a contradiction.

We exhausted all possibilities for $\{\xi,\xi^\theta\}$ and each one led to a contradiction. This proves the lemma.   
\end{proof}

\begin{lemma}\label{nottoadjacent}
If $\xi$ is a symp, then $\xi$ is not adjacent to $\xi^\theta$.
\end{lemma}

\begin{proof}
Suppose for a contradiction that the symp $\xi$ is mapped onto an adjacent one and set $U:=\xi\cap\xi^\theta$. 
By \cref{notsymplectic}, each point $x$ of $U$ is mapped to a collinear point and by \cref{colllinefixed}, the line $xx^\theta$ is fixed, implying $xx^\theta\subseteq U$. Hence $U^\theta=U$. Let $S$ be an arbitrary $4$-space of $U$ and let $W$ be the unique $5'$-space of $\xi$ containing $S$. 
The set $W\setminus S$, is mapped into the unique $5'$-space $W'$ of $\xi^\theta$ containing $S$, since the only points in $\xi^\theta$ collinear to a point of $W\setminus S$ are contained in $W'$. Hence $W'\setminus S=W^\theta\setminus S^\theta$. Since $W'\setminus S$ generates $W$, and likewise $W^\theta\setminus S^\theta$ generates $W^\theta$, this implies $W'=W^\theta$. Hence $S=S^\theta$. Since $S$ was arbitrary in $U$, this easily implies that $U$ is fixed pointwise by $\theta$. This, however, contradicts Lemma~\ref{nofixedpoints}. 
\end{proof}

\begin{lemma}\label{noeven}
There are no fixed even-dimensional singular subspaces. In particular, $\theta$ does not fix elements of types $2$, $5$ or $7$.
\end{lemma}

\begin{proof}
There are no fixed points by Lemma~\ref{nofixedpoints}. Suppose a subspace $W$ of even dimension $d>0$ is fixed. Pick  a point $p_0\in W$. Then $p_0^\theta\neq p_0$ and $p_0p_0^\theta$ is a fixed line, so $d\geq 2$. Pick $p_1\in W\setminus p_0p_0^\theta$. If $d=2$, then $p_0p_0^\theta\cap p_1p_1^\theta$ is a fixed point, a contradiction. Hence $d\geq 4$ and $S:=\<p_0p_0^\theta,p_1,p_1^\theta\>$ is a fixed $3$-space. Pick $p_2\in W\setminus S$. If $d=4$, then $S\cap p_2p_2^\theta$ is a fixed point, a contradiction. Hence $d=6$. But then $p_3p_3^\theta\cap\<S,p_2,p_2^\theta\>$ is a fixed point, for $p_3\in W\setminus\<S,p_2,p_2^\theta\>$, the final contradiction. 
\end{proof}

\begin{prop}\label{thechar}
The set of fixed elements of $\theta$, that is, the fixed lines, $3$-spaces, $5$-spaces and symps, form a geometry $\mathcal{F}$ of type $\mathsf{F_{4,1}}$ related to a building $\mathsf{F_{4}}(\K,\LL)$, with $\LL$ a quadratic extension of $\K$. 
\end{prop}

\begin{proof}
The opposition diagram ensures that there is a non-domestic symp $\xi$. Then by Theorem 6.1 of \cite{TTVM}, $\theta_\xi$ fixes at least two collinear points $p,q$. This implies that the lines $pp^\theta$ and $qq^\theta$ are fixed. Now, since $p\perp p^\theta\perp q^\theta\perp q\perp p$, the pair $\{p,q^\theta\}$ is symplectic and so there is a unique symp $\zeta$ containing $p,p^\theta,q$ and $q^\theta$.  Since $\theta$ fixed $pp^\theta$ and $qq^\theta$, also $\zeta$ is fixed by $\theta$. Then, by \cref{sympno2}, $\theta$ acts point-domestically on $\zeta$.  Since by \cref{nofixedpoints}, $\theta$ does not admit any fixed point, Theorem~8 of \cite{PVMclass} implies the fixed point structure of $\theta$ consists of lines, 3-spaces and 5-spaces, which, declaring the lines as points, form a polar space $\mathsf{C_{3,1}}(\LL,\K)$, with $\LL$ a (separable or inseparable) quadratic extension of $\K$. Every point of $\zeta$ lies on a fixed line. 

Now consider a fixed line $L$. We claim that $L$ is contained in a fixed symp. Indeed, if $L\subseteq \zeta$, then the claim is trivial. Otherwise, let $p\in L$ be arbitrary. If $p$ is far from $\zeta$, then so is $p^\theta$ and the points $p,p^\theta, p^\perp\cap\zeta$ and $(p^\theta)^\perp\cap\zeta$ define a fixed symp containing $L$. If $p$ and $\zeta$ are close, then so are $p^\theta$ and $\zeta$. If $p^\perp\cap\zeta=(p^\theta)^\perp\cap\zeta$, then the line $L$ is contained in the unique $6$-space containing $p^\perp\cap\zeta$ and hence intersects $\zeta$ in a fixed point, contradicting \cref{nofixedpoints}.   Hence we can select a point $x\in(p^\perp\cap\zeta)\setminus(p^\theta)^\perp$. The symp through $p^\theta$ and $x$ contains $p$ and $x^\theta$ and is hence fixed. 

Dualizing the situation to $\Delta^*=\mathsf{E_{7,1}}(\K)$, we find, using \cref{nottoadjacent}, that $\theta$ induces in the symp $\xi_L$ corresponding to $L$ a kangaroo fixing points, lines and planes (use \cref{noeven} to see that no other subspaces are fixed). \cref{residue2} shows that this corresponds to a polar space $\mathsf{B_{3,1}}(\K,\LL)$, for some quadratic extension $\LL$ of $\K$. (We use the same notation for $\LL$ as before since they will turn out to be the same in the next paragraph.)


We have shown, viewing the fixed symps in $\Delta$ as points, that $\mathcal{F}$ is a thick full subgeometry of type $\mathsf{F_{4}}$ of $\Delta^*$ with irreducible rank 3 residues isomorphic to $\mathsf{C_{3,1}}(\LL,\K)$ and $\mathsf{B_{3,1}}(\K,\LL)$ ($\LL$ still depends on the given residue). Hence all irreducible spherical residues of rank 3 are buildings and so $\cF$ is quotient of a building of type $\mathsf{F_4}$. However, it is easily verified that $\cF$ satisfies the conditions of Proposition~9 of \cite{Tits:83}. Hence $\cF$ defines a building $\Gamma$ of type $\mathsf{F_4}$, obviously isomorphic to $\mathsf{F_4}(\K,\LL)$, for some quadratic extension $\LL$ of $\K$, when the fixed points in $\Delta^*$ are viewed as the type $1$ elements of $\Gamma$. 
\end{proof}


\subsection{An alternative characterization of collineations with opposition diagram~$\mathsf{E_{7;3}}$}
We now show that the necessary conditions for a collineation to have opposition diagram $\mathsf{E_{7;3}}$ proved in Lemmas~\ref{nofixedpoints} and~\ref{notsymplectic}  are also sufficient. This can be considered an analogue of the kangaroo collineation classification in buildings of type $\mathsf{E_6}$, or various alternative characterizations of domestic dualities and trialities in buildings of type $\mathsf{A}_n$, $\mathsf{D_4}$ and $\mathsf{E_6}$.

So, henceforth, let $\theta$ be a collineation of $\mathsf{E_{7,7}}(\K)$ mapping no point to a symplectic one, fixing no point (these two requirements are equivalent to saying that $\theta$ is a $\{0,2\}$-kangaroo), and mapping at least one point to a collinear one (clearly, due to the classification in \cite{Tits:66} there exist anisotropic collineations, that is, collineations mapping every point to an opposite). 

We begin with the analogue of \cref{colllinefixed}, but the arguments are inspired by the proof of Proposition~3.1 of \cite{PVMclass}.

\begin{prop}\label{colllinefixed2}
For every point $p$ in $\Delta$ with $p^\theta \perp p$, the line $\<p,p^\theta\>$ is fixed under $\theta$. 
\end{prop}
\begin{proof}
Suppose for a contradiction that there exists a point $p$ with $p^\theta\perp p$ and $\<p,p^\theta\>$ not fixed under $\theta$.
Pick $x\in\<p,p^\theta\>\setminus\{p,p^\theta\}$. Then $x^\theta\in\<p^\theta,p^{\theta^2}\>\setminus\{p^\theta,p^{\theta^2}\}$.
Because $p^\theta$ is collinear with both $x$ and $x^\theta$, $x$ is collinear to $x^\theta$ and thus to all points of $\<p^\theta,p^{\theta^2}\>$. This implies that $p,p^\theta,p^{\theta^2}$ are contained in a plane $\pi$ of $\Delta$, which they span. Let $\alpha$ be a plane containing $\<p,p^\theta\>$ and so that $\alpha$ and $\pi$ are not contained in a common singular subspace of $\Delta$. Pick $q\in\alpha\setminus\pi$ and $r\in\<q,p^\theta\>\setminus\{q,p^\theta\}$ and note that $q$ is not collinear to $p^{\theta^2}$. Since $r\perp p$ (both are contained in $\alpha$), we have $r^\theta\perp p^\theta$ and so $r\perp r^\theta$. So $r^\theta$ is collinear to all points of $\<r,p^\theta\>$. Likewise, $q^\theta$ is collinear to all points of $\<q,p^\theta\>=\<r,p^\theta\>$. Hence all points of $\<q^\theta,r^\theta\>$ (in particular, $p^{\theta^2}$) are collinear to all points of $\<q,p^\theta\>$ (in particular, $q$). This contradiction concludes the proof.
\end{proof}

Since at least one point is mapped to a collinear point, there is at least one fixed line.

\begin{lemma}\label{noeven2}
The collineation $\theta$ fixes no singular spaces of even dimension in $\Delta$.\label{noevenfix}
\end{lemma}
\begin{proof}
If a singular space $U$ of $\Delta$ is fixed by $\theta$, then every point of this space is mapped to a collinear point. As there are no fixed points, we can, in view of \cref{colllinefixed2}, copy the proof of \cref{noeven}.
\end{proof}

\begin{lemma}\label{3to5}
There are $3$-spaces in $\Delta$ that are fixed by $\theta$. Every maximal $5$-space through a fixed $3$-space is fixed.
\end{lemma}
\begin{proof}
Let $L$ be a fixed line and $p$ a point that is collinear to all points on $L$. Then $p^\theta$ is also collinear to all points on $L$ and so $p$ has to be collinear to $p^\theta$, as they are not opposite. Thus $M:=\langle p,p^\theta \rangle$ is a fixed line, that can not intersect $L$ because there are no fixed points. Obviously $M$ and $L$ span a fixed singular space of dimension 3.

Let $S$ be a fixed $3$-space and $D$ an arbitrary $4$-space through $S$, $p$ a point in $D\setminus S$. The point $p$ is collinear to the fixed $3$-space $S$, thus $p\perp p^\theta$ and $p^\theta \perp S$. Thus $S$, $p$ and $p^\theta$ span a $5$-space $U$ that has to be fixed because $S$ and $\langle p,p^\theta\rangle$ are fixed. If $U$ would lie in a $6$-space $W$, then $W\cap W^\theta = U$, because a $6$-space can not be fixed by \cref{noeven2}. But considering $\Res_\Delta(x)$ for a point $x\in U$, this contradicts \cref{factE6}$(vi)$. We conclude that $U$ is a maximal singular subspace.

There is exactly one maximal $5$-space through each $4$-space in $\Delta$, so, by the arbitrariness of $D$, we have proved that all maximal $5$-spaces through a fixed $3$-space are fixed.
\end{proof}

\begin{lemma}
Every symp through a fixed $3$-space, or through a fixed $5$-space in $\Delta$ is fixed.
\end{lemma}

\begin{proof}
Since every fixed $5$-space contains a fixed $3$-space spanned by two fixed lines obtained by applying \cref{colllinefixed2} to two suitable points of the $5$-space, it suffices to show that each symp through a fixed $3$-space $U$ is fixed itself. But this follows from Lemma~\ref{3to5} as each symp through $U$ is determined by any two maximal $5$-spaces sharing $U$.
%
%
%
\end{proof}

In view of Remark~\ref{remark02}, we have shown:

\begin{theo}
Let $\theta$ be an automorphism of the building $\mathsf{E_7}(\K)$. Then the following are equivalent.
\begin{compactenum}[$(1)$]
\item $\theta$ does not fix any chamber and has opposition diagram $\mathsf{E_{7;3}}$.
\item The fixed point structure of $\theta$ induced in $\mathsf{E_{7,1}}(\K)$ is a full subgeometry isomorphic to $\mathsf{F_{4,1}}(\K,\LL)$, for some quadratic extension $\LL$ of $\K$, isometrically embedded.
\item The collineation induced in $\mathsf{E_{7,7}}(\K)$ does not map any point to a point at even distance and maps at least one point to a collinear one.  
\end{compactenum}
\end{theo}

\section{Domestic collineations of $\sE_7(\K)$ fixing a chamber}\label{sec:chamberfixing}

In this section we give an explicit classification of all domestic automorphisms of a large $\sE_7$ building fixing a chamber. The simpler cases of opposition diagrams $\sE_{7;1}$ and $\sE_{7;2}$ are dealt with in~\cite{PVMexc}, and so here we focus on the remaining more involved cases $\sE_{7;3}$ and $\sE_{7;4}$. There is a growing complexity of examples as one moves through the diagrams $\sE_{7;k}$ with $k=0,1,2,3,4$ (for both chamber-fixing, and non-chamber fixing examples). 

\subsection{Preliminary results}

We begin with the following general setup. Let $G_0$ be an adjoint Chevalley group of arbitrary irreducible (spherical) type over a commutative field $\K$. We adopt the notation and conventions outlined in \cite[Section~1.1]{PVMexc}, and so in particular $\Phi$ is the root system of $G_0$, with simple roots $\alpha_1,\ldots,\alpha_n$ and $(W,S)$ is the associated Coxeter system. The fundamental coweights of $\Phi$ are denoted $\omega_1,\ldots,\omega_n$, and the coweight lattice of $\Phi$ is $P=\ZZ\omega_1+\cdots+\ZZ\omega_n$. The highest root of $\Phi$ is $\varphi$. The \textit{polar type} of $\Phi$ is the subset $\wp\subseteq\{1,2,\ldots,n\}$ given by 
$$
\wp=\{1\leq i\leq n\mid \langle\alpha_i,\varphi\rangle\neq0\}.
$$
Let $G=G_{\Phi}(\K)$ be the subgroup of $\mathrm{Aut}(G_0)$ generated by the inner automorphisms of $G_0$ and the diagonal automorphisms, as in \cite{Hum:69,St:60}, and let $x_{\alpha}(a)$, $s_{\alpha}(t)$, and $h_{\lambda}(t)$ be the elements of $G$ described in \cite[Section~1.1]{PVMexc} (with $\alpha\in\Phi$, $\lambda\in P$, $a\in\K$, and $t\in\K^{\times}$). In particular, we have the relation
\begin{align}\label{eq:fold}
s_{\alpha}(t)=x_{\alpha}(t)x_{-\alpha}(-t^{-1})x_{\alpha}(t)
\end{align}
for $\alpha\in\Phi$ and $t\in\K^{\times}$.

For each $\alpha\in\Phi$ let $U_{\alpha}$ be the subgroup of $G$ generated by the elements $x_{\alpha}(a)$ with $a\in \K$. For $A\subseteq \Phi$ let $U_{A}$ be the subgroup generated by the groups $U_{\alpha}$, $\alpha\in A$, and let $U^+=U_{\Phi^+}$. Let $H$ be the subgroup generated by the diagonal elements $h_{\lambda}(t)$ with $\lambda\in P$ and $t\in\K^{\times}$. Let $B$ be the subgroup of $G$ generated by $U^+$ and $H$.

Let $\Delta=\Delta_{\Phi}(\K)$ be the standard split spherical building associated to $G$. Thus $\Delta$ has chamber set $G/B$ and Weyl distance function $\delta(gB,hB)=w$ if and only if $g^{-1}h\in BwB$.

Suppose that $\sX$ is a polar closed type preserving admissible Dynkin diagram (see \cite[\S3]{PVMexc} or \cite[\S1.4]{PVMclass} for the definition), let $J$ be the set of encircled nodes, and let $\varphi_1,\ldots,\varphi_N$ be the associated highest roots. The relevant examples for this paper are $\sX=\sE_{7;3}$ (in which case $J=\{1,6,7\}$ and $\varphi_1=\varphi_{\sE_7}$, $\varphi_2=\varphi_{\sD_6}$, and $\varphi_3=\alpha_7$) and $\sX=\sE_{7;4}$ (in which case $J=\{1,3,4,6\}$ and $\varphi_1=\varphi_{\sE_7}$, $\varphi_2=\varphi_{\sD_6}$, $\varphi_3=\varphi_{\sD_4}$, and $\varphi_4=\alpha_3$). Define
\begin{align}\label{eq:Psi}
\Psi_J=\{\beta\in\Phi\mid \langle\alpha,\beta\rangle=0\text{ for all $\alpha\in\Phi_{S\backslash J}$}\},
\end{align}
and let $\Psi_J^+=\Psi_J\cap \Phi^+$. Note that $\varphi_1,\ldots,\varphi_N\in\Psi_J$. 

The following theorem severely restricts the form of automorphisms of $\Delta$ with an opposition diagram~$\sX$. While we expect that a version of the theorem holds for non-simply laced buildings (and indeed for non-split buildings), we shall only require the theorem in the simply laced case, and hence we restrict to this simpler setting (in particular step 4 of the proof is simplified in this setting).

\begin{thm}\label{thm:partgeneral} Let $\Delta=\Delta_{\Phi}(\K)$ be the split spherical simply laced building associated to $G=G_{\Phi}(\K)$ with $|\K|>2$. Let $\theta$ be a type preserving automorphism of $\Delta$, and suppose that $\theta$ has polar closed opposition diagram $\sX$ with highest root sequence $\varphi_1,\ldots,\varphi_N$ and encircled nodes $J\subseteq S$. Then $\theta$ is linear (that is, $\theta$ does not involve a field automorphism), and $\theta$ is conjugate in $G$ to an element of the form
$$
x_{-\varphi_1}(1)\cdots x_{-\varphi_N}(1)uh
$$
where $u\in U_{\Psi_J}^+$ and $h\in H_J$, where $H_J$ is the subgroup of $H$ generated by the elements $h_{\omega_j}(t)$ with $j\in J$ and $t\in\K^{\times}$
\end{thm}

\begin{proof}
By \cite[Lemma~3.5]{PVMexc} we have $s_{\varphi_1}\cdots s_{\varphi_N}=w_{S\backslash J}w_0$. Since $\theta$ is capped (as $|\K|>2$) it follows from \cite[Theorem~5]{PVMclass} that there exists a chamber of $\Delta$ mapped to Weyl distance $s_{\varphi_1}\cdots s_{\varphi_N}$ by~$\theta$. Since $G$ acts strongly transitively on $\Delta$ we may (up to conjugation) assume that the chamber $B$ is mapped to Weyl distance $s_{\varphi_1}\cdots s_{\varphi_N}$, and hence $\theta B\in Bs_{\varphi_1}\cdots s_{\varphi_N}B$. Moreover, since the stabiliser of $B$ is transitive on each $w$-sphere $\{gB\mid gB\subseteq BwB\}$ we may assume, up to conjugation, that 
$$
\theta B=x_{\varphi_1}(1)\cdots x_{\varphi_N}(1)s_{\varphi_1}\cdots s_{\varphi_N}B.
$$
 Let $M=\ell(s_{\varphi_1}\cdots s_{\varphi_N})$. By \cite[Theorem~5]{PVMclass} we have $\ell(\delta(gB,\theta gB))\leq M$ with equality if and only if $\delta(gB,\theta gB)=s_{\varphi_1}\cdots s_{\varphi_N}$.

By~(\ref{eq:fold}) and the fact that $\varphi_1,\ldots,\varphi_N$ are mutually perpendicular we have
$$
x_{\varphi_1}(1)\cdots x_{\varphi_N}(1)s_{\varphi_1}\cdots s_{\varphi_N}B=x_{-\varphi_1}(1)\cdots x_{-\varphi_N}(1)B,
$$
and hence 
$$
\theta=x_{-\varphi_1}(1)\cdots x_{-\varphi_N}(1)uhf\quad\text{for some $u\in U^+$, $h\in H$, and $f\in\mathrm{Aut}(\K)$}.
$$
If $g\in G$ then the chamber $gB$ is mapped to Weyl distance $w$, where $g^{-1}\theta g\in BwBf$. The strategy to restrict $u,h,f$ is as follows: If $u,h,f$ do not satisfy certain conditions, then we will exhibit elements $g\in G$ with $g^{-1}\theta g\in BwBf$ and $\ell(w)>M$, yielding a contradiction as $\delta(gB,\theta gB)=w$. It is helpful to note that if $v\in W_{S\backslash J}$ then $\ell(s_{\varphi_1}\cdots s_{\varphi_N}v)=\ell(s_{\varphi_1}\cdots s_{\varphi_N})+\ell(v)$.

\medskip

\noindent\textit{Claim $1$: We have $u\in U_{\Phi^+\backslash\Phi_{S\backslash J}}$.} Write $u=u_1u_2$ with $u_1\in U_{\Phi_{S\backslash J}^+}$ and $u_2\in U_{\Phi^+\backslash\Phi_{S\backslash J}}$. Then
\begin{align*}
w_{S\backslash J}^{-1}\theta w_{S\backslash J}&=x_{-\varphi_1}(\pm 1)\cdots x_{-\varphi_N}(\pm 1)u_1^-u_2'h'f,
\end{align*}
where $u_1^-=w_{S\backslash J}^{-1}u_1w_{S\backslash J}\in U_{\Phi_{S\backslash J}^-}$, $u_2'\in U_{\Phi^+\backslash \Phi_{S\backslash J}}$, and $h'\in H$. Since $u_1^-\in BW_{S\backslash J}B$ we have $u_1^-\in BvB$ for some $v\in W_{S\backslash J}$. But then
$$
w_{S\backslash J}^{-1}\theta w_{S\backslash J}\in Bs_{\varphi_1}\cdots s_{\varphi_N}B\cdot BvBf=Bs_{\varphi_1}\cdots s_{\varphi_N}vBf.
$$
This forces $v=e$, and so $u_1^-\in B\cap U_{\Phi_{S\backslash J}^-}$, giving $u_1^-=1$, hence $u_1=1$.

\medskip

\noindent\textit{Claim $2$: We have $h\in H_J$.} Write $h=\prod_{j\in S}h_{\omega_j}(t_j)$ with $t_j\in\K^{\times}$. Suppose that $j\in S\backslash J$. Then
\begin{align*}
x_{-\alpha_j}(-1)\theta x_{-\alpha_j}(1)&=x_{-\varphi_1}(1)\cdots x_{-\varphi_N}(1)x_{-\alpha_j}(-1)ux_{-\alpha_j}(t_j^{-1})hf\\
&=x_{-\varphi_1}(1)\cdots x_{-\varphi_N}(1)x_{-\alpha_j}(t_j^{-1}-1)u'hf,
\end{align*}
with $u'\in U^+$ (here we have used the fact, from Claim 1, that $x_{\alpha_j}(a)$ does not appear as a factor in $u$). Thus, if $t_j\neq 1$, relation~(\ref{eq:fold}) gives
$$
x_{-\alpha_j}(-1)\theta x_{-\alpha_j}(1)\in Bs_{\varphi_1}\cdots s_{\varphi_N}B\cdot Bs_jBf=Bs_{\varphi_1}\cdots s_{\varphi_N}s_jBf,
$$
a contradiction (as $j\in S\backslash J$). Hence $t_j=1$ for all $j\in S\backslash J$ and so $h\in H_J$.

\medskip

\noindent\textit{Claim $3$: We have $f=\mathrm{id}$.} The arguments of Claims 1 and 2 prove the following: If $\theta'=x_{-\varphi_1}(a_1)\cdots x_{-\varphi_N}(a_N)uhf$ has opposition diagram $\sX$, where $a_1,\ldots,a_N\neq 0$, $u\in U^+$, $h\in H$, and $f\in\mathrm{Aut}(\K)$, then $u\in U_{\Phi^+\backslash\Phi_{S\backslash J}}$ and $h\in H_J$. Let $j\in S\backslash J$  and $t\in\K^{\times}$, and consider the element
$$
\theta'=h_{\omega_j}(t)^{-1}\theta h_{\omega_j}(t)=x_{-\varphi_1}(t^{\langle\varphi_1,\omega_j\rangle})\cdots x_{-\varphi_N}(t^{\langle\varphi_N,\omega_j\rangle})u'hh_{\omega_j}(t^{f}t^{-1})f,
$$
where $u'\in U_{\Phi^+\backslash\Phi_{S\backslash J}}$ and $h\in H_J$ (by Claims 1 and 2). But $\theta'$ has opposition diagram $\sX$ (as it is a conjugate of $\theta$) and hence by Claim 2 we have $t^{f}t^{-1}=1$. Since $t\in\K^{\times}$ was arbitrary we have $f=\mathrm{id}$.
\medskip

\noindent\textit{Claim $4$: We have $u\in U_{\Psi_J}^+$.} Let $Q_{S\backslash J}^+$ be the $\ZZ_{\geq 0}$-span of $\Phi_{S\backslash J}^+$, and define
$$
\Omega_J=\{\beta\in\Phi^+\mid (\beta-Q_{S\backslash J}^+)\cap \Phi=\{\beta\}\}.
$$
Note that $\{\varphi_1,\ldots,\varphi_N\}\subseteq \Psi_J^+\subseteq \Omega_J$. 

So far we have shown that, up to conjugation, $\theta=x_{-\varphi_1}(1)\cdots x_{-\varphi_N}(1)uh$ with $u\in U_{\Phi\backslash\Phi_{S\backslash J}}^+$ and $h\in H_J$. Write $u=u_1u_2$ with $u_1\in U_{\Psi_J}^+$ and $u_2\in U_{\Phi\backslash(\Phi_{S\backslash J}\cup\Psi_J)}^+$. We must show that $u_2=1$. Suppose, for a contradiction, that $u_2\neq 1$. 

Note that if $\alpha\in\Phi_{S\backslash J}$ then the root subgroup $U_{\alpha}$ commutes with $u_1$ and $x_{-\varphi_1}(1)\cdots x_{-\varphi_N}(1)$ (by definition of $\Psi_J$), and also it commutes with $h$ (as $h\in H_J$). Thus by conjugating $\theta$ by an element of $U_{\Phi_{S\backslash J}}^-$ we may assume that 
$
u_2=\cdots x_{\beta}(a)
$
with the product in decreasing root height, with $\beta\in\Omega_J$ and $a\neq 0$, and there is $j\in S\backslash J$ with $\beta+\alpha_j\in\Phi$ (see the proof of Claim 4 in \cite[Theorem~2.4]{PVMexc} for a similar argument). Now conjugate $\theta$ by $x_{-\beta-\alpha_j}(z)$, with $z\in\K$ yet to be chosen. We have, for some $\lambda\neq 0$,
\begin{align*}
x_{-\beta-\alpha_j}(-z)u_2hx_{-\beta-\alpha_j}(z)&=x_{-\beta-\alpha_j}(-z)u_2x_{-\beta-\alpha_j}(\lambda z)h. 
\end{align*}
Writing $u_2=u_2'x_{\beta+\alpha_j}(b)x_{\beta}(a)$ we have
\begin{align*}
x_{-\beta-\alpha_j}(-z)u_2hx_{-\beta-\alpha_j}(z)&=x_{-\beta-\alpha_j}(-z)u_2'x_{\beta+\alpha_j}(b)x_{\beta}(a)x_{-\beta-\alpha_j}(\lambda z)\\
&=x_{-\beta-\alpha_j}(-z)u_2'x_{\beta+\alpha_j}(b)x_{-\beta-\alpha_j}(\lambda z)x_{-\alpha_j}(\pm \lambda az)x_{\beta}(a).
\end{align*}
Recall the identity, for all $\alpha\in\Phi$,
$$
x_{\alpha}(a)x_{-\alpha}(b)=x_{-\alpha}(b(1+ab)^{-1})x_{\alpha}(a(1+ab))h_{\alpha^{\vee}}((1+ab)^{-1})
$$
whenever $1+ab\neq 0$ (this relation is easiest verified using matrices in $\mathsf{SL}_2(\K)$). Thus, choosing $z\in\K$ such that $1+\lambda bz\neq 0$ (here we use $|\K|>2$) we have
\begin{align*}
x_{-\beta-\alpha_j}(-z)u_2hx_{-\beta-\alpha_j}(z)&\in 
x_{-\beta-\alpha_j}(-z)u_2'x_{-\beta-\alpha_j}(\lambda z(1+\lambda bz)^{-1})x_{\beta+\alpha_j}(b(1+\lambda bz))x_{-\alpha_j}(\pm \lambda az)B\\
&\subseteq
x_{-\beta-\alpha_j}(-z)u_2'x_{-\beta-\alpha_j}(\lambda z(1+\lambda bz)^{-1})x_{-\alpha_j}(\pm \lambda az)B\\
&\subseteq x_{-\beta-\alpha_j}(-z+\lambda z(1+\lambda bz)^{-1})u_2''x_{-\alpha_j}(\pm \lambda az)B,
\end{align*}
where $u_2''\in U^+$ (by the assumption that $\beta$ is of minimal height in the expression for $u_2$). Thus, writing $\mu=\lambda z(1+\lambda bz)^{-1}$ we have
\begin{align*}
x_{-\beta-\alpha_j}(-z)u_2hx_{-\beta-\alpha_j}(z)\in x_{-\beta-\alpha_j}(-z+\mu)Bs_jB.
\end{align*}
Write $\theta'=x_{-\beta-\alpha_j}(-z)\theta x_{-\beta-\alpha_j}(z)$. It follows that
\begin{align*}
\theta'\in x_{-\beta-\alpha_j}(-z)x_{-\varphi_1}(1)\cdots x_{-\varphi_N}(1)u_1x_{-\beta-\alpha_j}(\mu)Bs_jB.
\end{align*}
Since the highest roots pairwise commute, and since $\varphi_1,\ldots,\varphi_N\in \Psi_J$, it follows that 
\begin{align*}
\theta'\in U_{-\beta-\alpha_j}U_{\Psi_J}^+s_{\varphi_1}\cdots s_{\varphi_N}U_{\varphi_1}\cdots U_{\varphi_N}U_{-\beta-\alpha_j}Bs_jB.
\end{align*}

We claim that
\begin{align}\label{eq:partial}
U_{-\beta-\alpha_j}U_{\Psi_J}^+s_{\varphi_1}\cdots s_{\varphi_N}U_{\varphi_1}\cdots U_{\varphi_N}U_{-\beta-\alpha_j}\subseteq Bw_{S\backslash J}w_0B. 
\end{align}
Before proving~(\ref{eq:partial}), note that the theorem follows from this equation, because it implies that
\begin{align*}
\theta'\in Bw_{S\backslash J}w_0B\cdot Bs_jB=Bw_{S\backslash J}w_0s_jB,
\end{align*}
and since $\ell(w_{S\backslash J}w_0s_j)=\ell(w_{S\backslash J}w_0)+1$ we arrive at the desired contradiction.

We now prove~(\ref{eq:partial}). First note that $s_{\varphi_1}\cdots s_{\varphi_N}=w_{S\backslash J}w_0=w_0w_{S\backslash J}$ (the fact that these longest elements commute follows from the fact that $J$ is stable under opposition). Moreover, note the following:
\begin{compactenum}[$(1)$]
\item If $\beta\in\Omega_J$ then $w_{S\backslash J}w_0\beta<0$ (this is because $\beta\in\Phi^+\backslash \Phi_{S\backslash J}=\Phi(w_{S\backslash J}w_0)$). 
\item If $j\in S\backslash J$ then $w_{S\backslash J}w_0\alpha_j=\alpha_j$ (to see this, note that by the list of diagrams in \cite{PVM}, if $\sX=(\Gamma,J,\sigma)$ is the opposition diagram of a type preserving automorphism, then $w_0$ and $w_{S\backslash J}$ induce the same permutation of $S\backslash J$).
\item Suppose $\gamma\in \Psi_J^+$ is such that $\gamma-\beta-\alpha_j$ is a root (with $\beta\in\Omega_J$ and $j\in S\backslash J$ such that $\beta+\alpha_j$ is a root). Then either $\gamma-\beta-\alpha_j$ is positive, or $w_{S\backslash J}w_0(\gamma-\beta-\alpha_j)$ is positive. For if $\gamma-\beta-\alpha_j$ is negative, then $\gamma-\beta$ is a nonpositive linear combination of roots. Moreover, $\gamma-\beta\notin Q_{S\backslash J}$ (as $\beta\notin \Phi_{S\backslash J}$) and so $\langle\gamma-\beta,\omega_k\rangle<0$ for some $k\in J$. Let $k'\in J$ be given by $w_0\omega_k=-\omega_{k'}$. Then $\langle w_{S\backslash J}w_0(\gamma-\beta),\omega_{k'}\rangle=\langle \gamma-\beta,w_0w_{S\backslash J}\omega_{k'}\rangle=\langle \gamma-\beta,w_0\omega_{k'}\rangle=-\langle \gamma-\beta,\omega_{k}\rangle>0$. Then, by (2), $w_{S\backslash J}w_0(\gamma-\beta-\alpha_j)=w_{S\backslash J}w_0(\gamma-\beta)-\alpha_j$ has positive coefficient of $\alpha_{k'}$, and hence is positive. 
\end{compactenum} 
It follows that
$$
U_{\varphi_1}\cdots U_{\varphi_N}U_{-\beta-\alpha_j}=U_{-\beta-\alpha_j}YU_{\varphi_1}\cdots U_{\varphi_N}X\subseteq U_{-\beta-\alpha_j}YB,
$$
with $Y$ a product of negative roots of the form $\varphi_i-\beta-\alpha_j$ for some $i$, and $X$ a product of positive roots of the form $\varphi_i-\beta-\alpha_j$ for some~$i$. 

Similarly, we have
$$
U_{-\beta-\alpha_j}U_{\Psi_J}^+=X'U_{\Psi_J}^+Y'U_{-\beta-\alpha_j}\subseteq BY'U_{-\beta-\alpha_j}
$$
where $X'$ (respectively $Y'$) is a product of positive (respectively negative) roots of the form $\gamma-\beta-\alpha_j$ with $\gamma\in\Psi_J^+$. Thus 
\begin{align*}
U_{-\beta-\alpha_j}U_{\Psi_J}^+s_{\varphi_1}\cdots s_{\varphi_N}U_{\varphi_1}\cdots U_{\varphi_N}U_{-\beta-\alpha_j}&\subseteq BY'U_{-\beta-\alpha_j}w_{S\backslash J}w_0U_{-\beta-\alpha_j}YB
\end{align*}
The strategy is now as follows. First the terms $Y'U_{-\beta-\alpha_j}$ are sent to the right, where they will ultimately be absorbed into~$B$. Then the terms $U_{-\beta-\alpha_j}Y$, along with any negative roots generated by sending $Y'U_{-\beta-\alpha_j}$ to the right, are sent to the left, where they are absorbed into~$B$. 

The details are as follows. The terms in $Y'U_{-\beta-\alpha_j}$ are of the form $\gamma-\beta-\alpha_j$ with $\gamma\in\{0\}\cup\Psi_J^+$ and $\gamma-\beta-\alpha_j\in-\Phi^+$. We have
\begin{align*}
U_{\gamma-\beta-\alpha_j}w_{S\backslash J}w_0U_{-\beta-\alpha_j}YB&=w_{S\backslash J}w_0U_{w_{S\backslash J}w_0(\gamma-\beta)-\alpha_j}U_{-\beta-\alpha_j}YB
\end{align*}
Note that the root $w_{S\backslash J}w_0(\gamma-\beta)-\alpha_j$ is necessarily positive (by the observations above), and it is not equal to the negative of any root appearing in $U_{-\beta-\alpha_j}Y$. To verify this, if $w_{S\backslash J}w_0(\gamma-\beta)-\alpha_j=-(\gamma'-\beta-\alpha_j)$ for some $\gamma'\in\{0\}\cup\Psi_J^+$ then for all $k\in S\backslash J$ it follows that 
$$
\langle \gamma-\beta,\alpha_k\rangle-\langle\alpha_j,\alpha_k\rangle=-\langle\gamma'-\beta,\alpha_k\rangle+\langle\alpha_j,\alpha_k\rangle,
$$
and hence (since $\langle\gamma,\alpha_k\rangle=\langle\gamma',\alpha_k\rangle=0$) we have
$
2\langle\beta+\alpha_j,\alpha_k\rangle=0.
$
Thus $\beta+\alpha_j\in\Psi_J$, a contradiction (as $s_j\beta=\beta+\alpha_j$). 

Therefore commutator relations can be used to move the positive root subgroup $U_{w_{S\backslash J}w_0(\gamma-\beta)-\alpha_j}$ to the right, past $U_{-\beta-\alpha_j}Y$, where it is absorbed into~$B$. It follows that 
\begin{align*}
BY'U_{-\beta-\alpha_j}w_{S\backslash J}w_0U_{-\beta-\alpha_j}YB&\subseteq Bw_{S\backslash J}w_0U_{-\beta-\alpha_j}YY''B,
\end{align*}
where $Y''$ consists of any negative roots resulting from the commutator relations (any positive roots are absorbed into $B$). One now moves the terms $U_{-\beta-\alpha_j}YY''$ across to the left, past $w_{S\backslash J}w_0$, where they become positive roots and are absorbed into $B$, and~(\ref{eq:partial}), and hence the theorem, follows. 
\end{proof}

The following elementary lemma is useful for studying domestic automorphisms that fix a chamber.

\begin{lemma}\label{lem:fixing}
Let $\Delta$ be a building and let $R$ be a residue. Suppose that $\theta$ is an automorphism of $\Delta$ stabilising $R$. Then $\theta$ fixes a chamber of $\Delta$ if and only if $\theta|_R$ fixes a chamber of $R$. 
\end{lemma}

\begin{proof}
Suppose a chamber $C$ of $\Delta$ is fixed. Since $R$ is stabilised, the projection $\mathrm{proj}_R(C)$ is fixed. The converse is clear (as each chamber of $R$ is also a chamber of $\Delta$).
\end{proof}

\subsection{The $\sE_7$ case}

We now specialise to the case of $\sE_7$. We will use standard Bourbaki~\cite{Bou:02} root system conventions, and when explicit calculations are required we adopt the sign conventions for root subgroups in the Chevalley group $\sE_7(\K)$ used in $\mathsf{MAGMA}$~\cite{CMT:04,MAGMA}. The list of positive roots of $\sE_7$ are listed in \cite[Appendix]{PVMexc} for reference.

The following lemma records basic facts about the system $\Psi_J$ from~(\ref{eq:Psi}) for the diagrams $\sE_{7;3}$ and $\sE_{7;4}$. For brevity, we sometimes write $i$ in place of $s_i$ (for example, $s_1s_2s_7s_4=1274$).

\begin{lemma}\label{lem:magicelement}
We have the following.
\begin{compactenum}[$(1)$]
\item Let $\sX=\sE_{7;3}$. Then $\Psi_J$ is of type $\sA_1\times\sA_1\times\sA_1$ with simple roots $\gamma_1=\varphi_{\sE_7}$, $\gamma_2=\varphi_{\sD_6}$, $\gamma_3=\alpha_7$. The element $\su=134265423143765423143546$ has the property $\su^{-1}\gamma_1=\alpha_7$, $\su^{-1}\gamma_2=\alpha_5$, and $\su^{-1}\gamma_3=\alpha_2$, and thus conjugates the system $\Psi_J$ to the $\sA_1\times \sA_1\times \sA_1$ system generated by $\{\alpha_7,\alpha_5,\alpha_2\}$. 
\item Let $\sX=\sE_{7;4}$. Then $\Psi_J$ is of type $\sD_4$ with simple system $\gamma_1=\varphi_{\sD_6}$, $\gamma_2=\alpha_1$, $\gamma_3=\varphi_{\sD_4}$, and $\gamma_4=\alpha_3$. 
The element $
\su=4 3 1 5 4 3 6 5 4 2 3 1 4 3 5 4 6 5 7 6 5 4 3 1
$
satisfies $\su^{-1}\gamma_1=\alpha_2$, $\su^{-1}\gamma_2=\alpha_4$, $\su^{-1}\gamma_3=\alpha_3$, and $\su^{-1}\gamma_4=\alpha_5$, and thus conjugates $\Psi_J$ to the standard $\sD_4$ parabolic subsystem.
\end{compactenum}
\end{lemma}

\begin{proof}
This is verified by direct calculation. 
\end{proof}

\subsubsection{The $\sE_{7;3}$ diagram}

We now consider the $\sE_{7;3}$ diagram.

\begin{thm}\label{thm:E73chamberfixing}
Let $\theta$ be be an automorphism of the $\sE_7(\K)$ building with $\Diag(\theta)=\sE_{7;3}$, where $|\K|>2$. Then $\theta$ is conjugate to an element of the form
\begin{align}
\label{eq:theta}
x_{\varphi_1}(t_1^{-1}t_2^{-1}a)x_{\varphi_2}(t_2^{-1}a)x_{\varphi_3}(a)s_{\varphi_1}^{-1}s_{\varphi_2}^{-1}s_{\varphi_3}^{-1}h_{\omega_1}(t_1)h_{\omega_6}(t_2)h_{\omega_7}(t_3)
\end{align}
with $a\in\K$ and $t_1,t_2,t_3\in\K^{\times}$, and where $\varphi_1=\varphi_{\sE_7}$, $\varphi_2=\varphi_{\sD_6}$, and $\varphi_3=\alpha_7$. Moreover, $\theta$ fixes a chamber (and hence is conjugate to a member of $B$) if and only if the polynomial $p(Y)=Y^2+aY+t_3^{-1}$ has a root $y\in\K$. If $y\in\K$ is a root of $p(Y)$ then:

\begin{compactenum}[$(1)$]
\item If $t_3y^2=1$ then $\theta$ is conjugate to the unipotent element $x_{\varphi_1}(t_1t_2t_3y)x_{\varphi_2}(t_2t_3y)x_{\varphi_3}(t_3y)$. 
\item If $t_3y^2\neq 1$ then $\theta$ is conjugate to the homology $h_{\omega_7}(t_3^{-1}y^{-2})$. 
\end{compactenum}
\end{thm}

\begin{proof} Let $J=\{1,6,7\}$. We have $\Psi_J^+=\{\varphi_1,\varphi_2,\varphi_3\}$ (a system of type $\sA_1\times\sA_1\times \sA_1$), and thus by Theorem~\ref{thm:partgeneral} we may assume, up to conjugation, that 
$$
\theta=x_{-\varphi_1}(1)x_{-\varphi_2}(1)x_{-\varphi_3}(1)uh
$$
where $u=x_{\varphi_1}(a_1)x_{\varphi_2}(a_2)x_{\varphi_3}(a_3)$ and $h=h_{\omega_1}(t_1)h_{\omega_6}(t_2)h_{\omega_7}(t_3)$, with $a_1,a_2,a_3\in\K$ and $t_1,t_2,t_3\in\K^{\times}$. Using~(\ref{eq:fold}), and the fact that $s_{\varphi_i}^{-1}x_{\varphi_j}(1)s_{\varphi_i}=x_{\varphi_j}(1)$ for $i\neq j$, it follows that
\begin{align*}
\theta_1&=x_{\varphi_1}(1)x_{\varphi_2}(1)x_{\varphi_3}(1)s_{\varphi_1}^{-1}s_{\varphi_2}^{-1}s_{\varphi_3}^{-1}x_{\varphi_1}(a_1')x_{\varphi_2}(a_2')x_{\varphi_3}(a_3')h_{\omega_1}(t_1)h_{\omega_6}(t_2)h_{\omega_7}(t_3)\\
&=x_{\varphi_1}(1)x_{\varphi_2}(1)x_{\varphi_3}(1)s_{\varphi_1}^{-1}s_{\varphi_2}^{-1}s_{\varphi_3}^{-1}h_{\omega_1}(t_1)h_{\omega_6}(t_2)h_{\omega_7}(t_3)x_{\varphi_1}(t_1^{-2}t_2^{-2}t_3^{-1}a_1')x_{\varphi_2}(t_2^{-2}t_3^{-1}a_2')x_{\varphi_3}(t_3^{-1}a_3'),
\end{align*}
where $a_i'=a_i+1$ for $i=1,2,3$. If follows that $\theta$ is conjugate to 
\begin{align*}
\theta'&=x_{\varphi_1}(b_1)x_{\varphi_2}(b_2)x_{\varphi_3}(b_3)s_{\varphi_1}^{-1}s_{\varphi_2}^{-1}s_{\varphi_3}^{-1}h_{\omega_1}(t_1)h_{\omega_6}(t_2)h_{\omega_7}(t_3),
\end{align*}
where $b_1=1+t_1^{-2}t_2^{-2}t_3^{-1}a_1'$, $b_2=1+t_2^{-2}t_3^{-1}a_2'$, and $b_3=1+t_3^{-1}a_3'$. 

Let $\alpha=(1010000)$ and $\beta=(1234321)$. A direct calculation shows that if $t_1b_1-b_2\neq 0$ then 
\begin{align*}
x_{-\beta}(-1)x_{-\alpha}(-1)\theta' x_{-\alpha}(1)x_{-\beta}(1)\in Bs_{\varphi_1}s_{\varphi_2}s_{\varphi_3}s_3B,
\end{align*}
contradicting the fact that $\theta$ has opposition diagram~$\sE_{7;3}$. Similarly, writing $\gamma=(0000110)$ and $\delta=(0112211)$ we see that if $t_2b_2-b_3\neq 0$ then 
$x_{-\delta}(-1)x_{-\gamma}(-1)\theta' x_{-\gamma}(1)x_{-\delta}(1)\in Bs_{\varphi_1}s_{\varphi_2}s_{\varphi_3}s_5B$, 
again a contradiction. Writing $a=b_3$ it follows that $b_2=t_2^{-1}a$ and $b_1=t_1^{-1}t_2^{-1}a$, and~(\ref{eq:theta}) is proved.

We claim that if there exists $y\in \K$ with 
$
y^2+ay+t_3^{-1}=0
$
then $\theta$ is conjugate to
\begin{align}\label{eq:thetaprime}
\theta''&=x_{\varphi_1}(t_1t_2t_3y)x_{\varphi_2}(t_2t_3y)x_{\varphi_3}(t_3y)h_{\omega_7}(t_3^{-1}y^{-2}).
\end{align}
The equation $y^2+ay+t_3^{-1}=0$ gives $a=-y-t_3^{-1}y^{-1}$. A direct calculation shows that the chamber $g_1B$ is fixed by $\theta'$, where
$g_1=x_{\varphi_1}(-t_1^{-1}t_2^{-1}y)x_{\varphi_2}(-t_2^{-1}y)x_{\varphi_3}(-y)s_{\varphi_1}s_{\varphi_2}s_{\varphi_3}.$ 
We compute $g_1^{-1}\theta' g_1=x_{\varphi_1}(t_1t_2t_3y)x_{\varphi_2}(t_2t_3y)x_{\varphi_3}(t_3y)h_{\omega_7}(t_3^{-1}y^{-2})$ as required.

Next we claim that if $y^2+ay+t_3^{-1}=0$ then statements $(1)$ and $(2)$ in the statement of the theorem hold. Indeed statement (1) is immediate from~(\ref{eq:thetaprime}). For (2), note that if $t_3y^2\neq 1$ then a direct calculation shows that $\theta''$ fixes the chamber $g_2B$, where
$$
g_2=x_{\varphi_1}(t_1t_2t_3^2y^3(t_3y^2-1)^{-1})x_{\varphi_2}(t_2t_3^2y^3(t_3y^2-1)^{-1})x_{\varphi_3}(t_3^2y^3(t_3y^2-1)^{-1})s_{\varphi_1}s_{\varphi_2}s_{\varphi_3}.
$$ 
Direct calculation gives $g_2^{-1}\theta'' g_2=h_{\omega_7}(t_3^{-1}y^{-2})$ as required.

Finally, we claim that if $p(Y)=Y^2+aY+t_3^{-1}$ is irreducible over $\K$ then $\theta$ does not fix any chamber. With $\su$ as in Lemma~\ref{lem:magicelement}, and with $\theta'$ as above, we have 
$$
\theta'''=\su^{-1}\theta'\su=x_{\alpha_7}(t_1^{-1}t_2^{-1}a)x_{\alpha_5}(t_2^{-1}a)x_{\alpha_2}(a)s_7^{-1}s_5^{-1}s_2^{-1}h_{\alpha_7}(t_1)h_{\alpha_5+\alpha_7}(t_2)h_{\frac{1}{2}(\alpha_2+\alpha_5+\alpha_7)}(t_3)
$$
(note that $\frac{1}{2}(\alpha_2+\alpha_5+\alpha_7)\in P$). In particular, $\theta'''$ stabilises the residue $R$ of type $\{2,5,7\}$ containing the base chamber $B$. Thus by Lemma~\ref{lem:fixing} $\theta'''$ fixes a chamber of $\Delta$ if and only if it fixes a chamber of $R$, and a simple calculation shows that this occurs if and only if $p(Y)$ splits over~$\K$.
\end{proof}

We now prove Theorem~\ref{thm:E73Classification1}, which we restate for convenience. 

\begin{thm}\label{thm:E73Classification}
An automorphism $\theta$ of the building $\sE_7(\K)$ with $|\K|>2$ has opposition diagram $\sE_{7;3}$ and fixes a chamber if and only if $\theta$ is conjugate to one of the following elements
\begin{compactenum}[$(1)$]
\item $x_{\varphi_1}(1)x_{\varphi_2}(1)x_{\varphi_3}(1)$;
\item $h_{\omega_7}(t)$ with $t\neq 0,1$.
\end{compactenum}
\end{thm}

\begin{proof}
If $\theta$ fixes a chamber and has opposition diagram $\sE_{7;3}$ then, by Theorem~\ref{thm:E73chamberfixing} $\theta$ is conjugate to either an element of the form $\theta_1=x_{\varphi_1}(a)x_{\varphi_2}(b)x_{\varphi_3}(c)$ with $a,b,c\neq0$, or $\theta_2=h_{\omega_7}(t)$ with $t\neq 0,1$. The element $\theta_1$ is in turn conjugate (by an element of $H$) to $x_{\varphi_1}(1)x_{\varphi_2}(1)x_{\varphi_3}(1)$. 

Conversely, by \cite[Theorem~3.1 and Lemma~4.5]{PVMexc} the elements of the form $\theta_1$ and $\theta_2$ are domestic with opposition diagram~$\sE_{7;3}$. 
\end{proof}

\subsubsection{The $\sE_{7;4}$ diagram}

We now turn our attention to the $\sE_{7;4}$ diagram. This case is considerably more involved than the $\sE_{7;3}$ case. Let $\Phi_{\sD_4}$ be the $\sD_4$ subsystem of the $\sE_7$ root system~$\Phi$. To fix conventions, we identify the simple roots $\alpha_2,\alpha_3,\alpha_4,\alpha_5$ of the $\sE_7$ root system with simple roots $\alpha_1',\alpha_3',\alpha_2',\alpha_4'$, respectively, of the $\sD_4$ system. For example, the root $\alpha=(0101000)$ of $\sE_7$ is identified with the root $(1100)$ of $\sD_4$. Thus, for example, we shall write $x_{1100}(a)$ and $x_{0101000}(a)$ interchangeably. Moreover, we adopt the standard realisation of the simple roots of $\sD_4$ in $\mathbb{R}^4$, with $\alpha_1'=e_1-e_2$, $\alpha_2'=e_2-e_3$, $\alpha_3'=e_3-e_4$, and $\alpha_4'=e_3+e_4$.

Let $G_{\sD_4}$ be the subgroup of the $\sE_7$ Chevalley group $G$ generated by the root subgroups $U_{\alpha}$ with $\alpha$ in the $\sD_4$ subsystem. Let $w_{\sD_4}$ be the longest element of the $\sD_4$ Coxeter group. In the theorem below, we shall see that all automorphisms of the $\sE_7$ building with opposition diagram $\sE_{7;4}$ can be conjugated into $G_{\sD_4}$, allowing us to make explicit matrix calculations in the $G_{\sD_4}$ group by realising it as a subgroup of $\mathsf{GL}_8(\K)$. We shall use the sign conventions in $\sD_4$ that are inherited from the sign choices in $\mathsf{MAGMA}$ for the $\sE_7$ system (see the Groups of Lie Type package~\cite{CMT:04}). With these conventions, explicit $8\times 8$ matrices for the standard representation of the $G_{\sD_4}$ are as follows:
\begin{align*}
x_{e_1-e_2}(a)&=I+aE_{21}-aE_{87}&x_{e_1+e_2}(a)&=I+aE_{71}-aE_{82}\\
x_{e_2-e_3}(a)&=I+aE_{32}-aE_{76}&x_{e_2+e_3}(a)&=I+aE_{62}-aE_{73}\\
x_{e_3-e_4}(a)&=I+aE_{43}-aE_{65}&x_{e_3+e_4}(a)&=I+aE_{53}-aE_{64}\\
x_{e_1-e_3}(a)&=I-aE_{31}+aE_{86}&x_{e_1+e_3}(a)&=I-aE_{61}+aE_{83}\\
x_{e_1-e_4}(a)&=I-aE_{41}+aE_{85}&x_{e_1+e_4}(a)&=I+aE_{51}-aE_{84}\\
x_{e_2-e_4}(a)&=I+aE_{42}-aE_{75}&x_{e_2+e_4}(a)&=I-aE_{52}+aE_{74}.
\end{align*}
The group $G_{\sD_4}$ acts on $V=\mathbb{R}^8$ on the right, and this action preserves the bilinear form
$$
(X,Y)=X_1Y_8+X_2Y_7+X_3Y_6+X_4Y_5+X_5Y_4+X_6Y_3+X_7Y_2+X_8Y_1.
$$
Writing $f(X)=X_1X_8+X_2X_7+X_3X_6+X_4X_5$, a vector $X$ is isotropic if $f(X)=0$. A subspace $V'$ is singular if $(X,Y)=0$ for all $X,Y\in V'$. Then the associated $\sD_4$ building is realised as the oriflamme complex in the usual way.

\begin{thm}\label{thm:E74chamberfixing}
Let $\theta$ be be an automorphism of the $\sE_7(\K)$ building with $\Diag(\theta)=\sE_{7;4}$, where $|\K|>2$. Then $\theta$ is conjugate to an element of $G_{\sD_4}$ of the form
$
\theta=uhw_{\sD_4},
$
where $u,h\in G_{\sD_4}$ are of the form
\begin{align*}
u&=x_{1000}(t_2t_3t_4a)x_{1100}(-t_1t_2t_3t_4b)x_{1101}(t_1t_2t_3t_4c)x_{1111}(-t_1t_2^2t_3^2t_4b)x_{1211}(t_1t_2^2t_3^2t_4a)x_{0010}(t_2t_3a)\\
&\qquad x_{0110}(-t_1t_2t_3b)x_{0111}(t_1t_2t_3c)x_{0001}(t_2a)x_{1110}(-t_1t_2t_3^2t_4c)x_{0101}(t_1t_2b)x_{0100}(t_1c)\\
h&=h_{1000}(t_1t_2^2t_3^3t_4^2)h_{0100}(t_1^2t_2^3t_3^4t_4^2)h_{0010}(t_1t_2^2t_3^3t_4)h_{0001}(t_1t_2^2t_3^2t_4)
\end{align*}
with $a,b,c\in\K$ and $t_1,t_2,t_3,t_4\in\K^{\times}$. Moreover, if the automorphism $\theta$ fixes a chamber of $\sE_7(\K)$ then the polynomial 
$$
p(Y)=Y^2-(t_2a^2+t_1t_2b^2+t_1c^2-t_1t_2abc-2)Y+1
$$
splits over $\K$.
\end{thm}

\begin{proof}
By Theorem~\ref{thm:partgeneral}, and following the initial paragraph of the proof of Theorem~\ref{thm:E73chamberfixing}, if $\theta$ has opposition diagram $\sE_{7;4}$ then $\theta$ is conjugate to an element of the form 
\begin{align*}
\theta'&=u's_{\varphi_1}^{-1}s_{\varphi_2}^{-1}s_{\varphi_3}^{-1}s_{\varphi_4}^{-1}h_{\omega_1}(t_1)h_{\omega_3}(t_2)h_{\omega_4}(t_3)h_{\omega_6}(t_4),
\end{align*}
where $u'\in\Psi_J$ (with $J=\{1,3,4,6\}$, and $\varphi_1=\varphi_{\sE_7}$, $\varphi_2=\varphi_{\sD_6}$, $\varphi_3=\varphi_{\sD_4}$, and $\varphi_4=\alpha_3$).

We will now make explicit calculations to further restrict~$u'$. Write
\begin{align*}
u'&=x_{0112221}(a_1) x_{01112221}(a_2) x_{1 1 2 2 2 2 1}(a_3) x_{1 2 3 4 3 2 1}(a_4) x_{2 2 3 4 3 2 1}(a_5) x_{0 1 1 2 1 0 0}(a_6) x_{1 1 1 2 1 0 0}(a_7)\\
&\times x_{1 1 2 2 1 0 0}(a_8) x_{0 0 1 0 0 0 0}(a_9) x_{1 2 2 4 3 2 1}(a_{10})x_{1 0 1 0 0 0 0}(a_{11}) x_{1 0 0 0 0 0 0}(a_{12})
\end{align*}
(the roots appearing here are the twelve roots of $\Psi_J$). We consider the element $v\in W$ given by 
$$
Bx_{-\beta}(-1)x_{-\alpha}(-1)\theta' x_{-\alpha}(1)x_{-\beta}(1)B=BvB
$$
for various choices of $\alpha$ and $\beta$. Since $
v=\delta(x_{-\alpha}(1)x_{-\beta}(1)B,\theta' x_{-\alpha}(1)x_{-\beta}(1)B)$, if $\ell(v)>\ell(w_{S\backslash J}w_0)=60$ then we have a contradiction with the fact that $\theta$ has opposition diagram $\sE_{7;4}$. By computation (using $\mathsf{MAGMA}$ \cite{CMT:04}), if $\alpha=(0101000)$ and $\beta=(0111100)$ and $a_9\neq t_3a_6$ then $v=s_5s_7w_0$, a contradiction. Thus $a_9=t_3a_6$. Similarly, taking $\alpha=(0101000)$ and $\beta=(1111100)$ forces $a_{11}=t_3a_7$ (otherwise again $v=s_5s_7w_0$). Taking $\alpha=(0101000)$ and $\beta=(1223321)$ gives $a_{10}=t_3^{-1}a_3$ (otherwise again $v=s_5s_7w_0$). Taking $\alpha=(0000110)$ and $\beta=(0112211)$ forces $a_6=t_4a_1$ (otherwise $v=s_2s_7w_0$). Taking $\alpha=(0000110)$ and $\beta=(1112211)$ forces $a_7=t_4a_2$ (otherwise $v=s_2s_7w_0$). Taking $\alpha=(0000110)$ and $\beta=(1122211)$ forces $a_8=t_4a_3$ (otherwise $v=s_2s_7w_0$). Taking $\alpha=(0111000)$ and $\beta=(1111100)$ forces $a_{12}=t_2t_3t_4a_3$ (otherwise $v=s_5s_7w_0$). Taking $\alpha=(1111000)$ and $\beta=(1223321)$ forces $a_5= t_1^{-1}t_2^{-1}t_3^{-1}a_1$ (otherwise $v=s_5s_7w_0$). Finally, taking $\alpha=(0111000)$ and $\beta=(1223321)$ forces $a_4=-t_2^{-1}t_3^{-1}a_2$ (otherwise $v=s_5s_7w_0$).

Conjugating by the element $\su$ from Lemma~\ref{lem:magicelement} it follows that $\theta$ is conjugate to an element of the form $\theta''=u''w_{\sD_4}h$, where $h\in H_{\sD_4}$ is as in the statement of the theorem, and $u''\in U_{\sD_4}^+$ is of the form
\begin{align*}
u''&=x_{1000}(a_1) x_{1100}(-a_2) x_{1101}(a_3) x_{1111}(-t_2^{-1}t_3^{-1}a_2) x_{1211}(t_1^{-1}t_2^{-1}t_3^{-1}a_1)x_{0010}(t_4a_1)\\
&\times x_{0110}(-t_4a_2) x_{0111}(t_4a_3) x_{0001}(t_3t_4a_1)x_{1110}(-t_3^{-1}a_3)x_{0101}(t_3t_4a_2)x_{0100}(t_2t_3t_4a_3).
\end{align*} 
Finally, replacing $\theta''$ by $h^{-1}\theta'' h$, and setting $a=t_3t_4a_1$, $b=t_3t_4a_2$, and $c=t_2t_3t_4a_3$ we see that $\theta$ is conjugate to $uhw_{\sD_4}$, with $u$ and $h$ as in the statement of the theorem.

Now, if $\theta=uhw_{\sD_4}$ fixes a chamber of $\sE_7$ then by Lemma~\ref{lem:fixing} it also fixes a chamber of the $\sD_4$ residue. Thus $\theta$ can be conjugated (in $G_{\sD_4}$) into the standard Borel subgroup of $G_{\sD_4}$ consisting of lower triangular matrices (lower triangular due to the right action). Thus all eigenvalues of the $8\times 8$ matrix representing $\theta$ lie in~$\K$. Direct calculation, using the matrices listed above, gives
\begin{align*}
\det(\theta-\lambda I)&=(\lambda-1)^4p(\lambda)^2,
\end{align*}
where 
$
p(\lambda)=\lambda^2-(t_2a^2+t_1t_2b^2+t_1c^2-t_1t_2abc-2)\lambda+1
$, hence the result.
\end{proof}

\begin{remark}
In fact, $\theta$ fixes a chamber of $\sE_7(\K)$ if and only if the polynomial $p(Y)$ splits over~$\K$. The converse will be proved in the theorem below, where we will show that if $z\in \K$ satisfies $p(z)=0$ then $\theta$ can be conjugated into the standard Borel of $G_{\sD_4}$, and hence fixes a chamber of the $\sD_4$ residue (and hence the $\sE_7$ building).  
\end{remark}

The following lemmas will be used.

\begin{lemma}\label{lem:conjugate}
Let $A=\{\alpha\in\Phi_{\sD_4}^+\mid \alpha\geq \alpha_2\}=\Phi_{\sD_4}^+\backslash\{\alpha_1,\alpha_3,\alpha_4\}$ (a closed sets of roots). Suppose that $\theta\in U_A$ with $\theta\notin U_{A\backslash\{\alpha_2\}}$. Then $\theta$ is conjugate to an element of the form
$$
\theta'=x_{1111}(b_0)x_{0100}(b_1)x_{1110}(b_2)x_{1101}(b_3)x_{0111}(b_4)x_{1211}(b_5).
$$
If $b_0,b_1,b_2,b_3,b_4\neq 0$ then $\theta'$ is conjugate to
$$
\theta''=x_{1111}(1)x_{0100}(-b_0^2b_1b_2^{-1}b_3^{-1}b_4^{-1})x_{1110}(1)x_{1101}(1)x_{0111}(1).
$$
\end{lemma}

\begin{proof}
Write $\theta=\prod_{\alpha\in A}x_{\alpha}(a_{\alpha})$ (with the product taken in any fixed order). Since $\theta\notin U_{A\backslash \{\alpha_2\}}$ we have $a_{0100}\neq 0$. Conjugating $\theta$ by $g=x_1(\pm a_{1100}/a_{0100})x_3(\pm a_{0110}/a_{0100})x_4(\pm a_{0101}/a_{0100})$ for an appropriate choice of signs shows that $\theta\sim\theta'$. If $b_0,b_1,b_2,b_3,b_4\neq 0$ then conjugating $\theta'$ by 
$$h_{\alpha_2^{\vee}}(b_0)h_{\omega_1}(b_4)h_{\omega_2}(b_2^{-1}b_3^{-1}b_4^{-1})h_{\omega_3}(b_3)h_{\omega_4}(b_2)x_{0100}(\pm b_5/b_0)$$ for an appropriate sign shows that $\theta'\sim \theta''$.
\end{proof}

\begin{lemma}\label{lem:4perpD4}
Let $\beta_1,\beta_2,\beta_3,\beta_4\in\Phi_{\sD_4}^+$ be mutually perpendicular roots of the $\sD_4$ root system. If $\theta=x_{\beta_1}(a_1)x_{\beta_2}(a_2)x_{\beta_3}(a_3)x_{\beta_4}(a_4)$ with $b_1,b_2,b_3,b_4\neq 0$ then $\theta$ is conjugate to an element of the form $x_{0100}(a)x_{1110}(1)x_{1101}(1)x_{0111}(1)$ with $a\neq 0$. 
\end{lemma}

\begin{proof}
The sets of $4$ mutually perpendicular roots of $\Phi_{\sD_4}^+$ are $X_1=\{(1000),(0010),(0001),(1211)\}$, $X_2=\{(1100),(0110),(0101),(1111)\}$, and $X_3=\{(0100),(1110),(0111),(1101)\}$. Then $s_2$ conjugates $X_1$ to $X_2$, and $s_1$ conjugates $X_2$ to $X_3$, and it follows that $\theta$ is conjugate to an element of the form $x_{0100}(\pm a_1)x_{1110}(\pm a_2)x_{1101}(\pm a_3)x_{0111}(\pm a_4)$. One can now conjugate by an appropriate element of~$H$. 
\end{proof}

We now prove Theorem~\ref{thm:E74chamberfixingmain1}, which we restate below. 

\begin{thm}\label{thm:E74chamberfixingmain}
An automorphism $\theta$ of the building $\sE_7(\K)$ with $|\K|>2$ has opposition diagram $\sE_{7;4}$ and fixes a chamber if and only if $\theta$ is conjugate to one of the following elements
\begin{compactenum}[$(1)$]
\item $x_{0100}(a)x_{1110}(1)x_{1101}(1)x_{0111}(1)$ with $a\neq 0$; 
\item $x_{1111}(1)x_{0100}(a)x_{1110}(1)x_{1101}(1)x_{0111}(1)$ with $a\neq 0$; 
\item $h_{\varphi^{\vee}}(c)$ with $c\neq 0,1$;
\item $x_{\varphi}(1)h_{\varphi^{\vee}}(-1)$ (with $\kar(\K)\neq 2$).
\end{compactenum}
\end{thm}

\begin{proof}
Suppose that $\theta$ fixes a chamber and has opposition diagram~$\sE_{7;4}$. Then by Theorem~\ref{thm:E74chamberfixing} we may assume (up to conjugation) that $
\theta=uhw_{\sD_4},
$
where $u,h\in G_{\sD_4}$ are as in Theorem~\ref{thm:E74chamberfixing}, and moreover the polynomial $
p(Y)=Y^2-(t_2a^2+t_1t_2b^2+t_1c^2-t_1t_2abc-2)Y+1
$
has a root $z\in\K$. Then $z^{-1}\in\K$ is also a root of $p(Y)$. 

We shall now explicitly conjugate $\theta$ into the standard Borel (the subgroup of lower triangular matrices in $G_{\sD_4}$). The working depends on various cases for $z,z^{-1}$. 
\smallskip

\noindent \textbf{Case: $z=1$}. The equation $p(1)=0$ gives the equation
\begin{align}\label{eq:mainrelation}
t_2a^2+t_1t_2b^2+t_1c^2-t_1t_2abc-4=0,
\end{align}
and we make frequent use of this relation in the calculations below for this case. Let $(\rho,V)$ be the standard $8$-dimensional representation of $\sD_4$. By a \textit{standard basis} of $V$ we shall mean a basis $\{v_1,v_2,\ldots,v_8\}$ of $V$ with $f(v_i)=0$ for $1\leq i\leq 8$, and $(v_i,v_{8-i+1})=1$ for $1\leq i\leq 4$, and $(v_i,v_j)=0$ for all other pairs $i,j$. 
\smallskip

\noindent\textit{Subcase: $q=t_2b^2-c^2\neq 0$ and $r=t_1c^2-4\neq 0$.} Let 
\begin{align*}
v_1&=(-t_1t_2^2t_3^3t_4^2c,\,  t_1t_2^2t_3^2t_4b,\,   -t_2^2t_3^2t_4a,\, 2t_2t_3t_4,\,   2t_2t_3^2t_4, \,t_2t_3t_4a,\,  t_2t_3t_4(ac-b),\,   c)\\
v_2&=(-2t_1t_2^2t_3^3t_4^2,\,   t_1t_2^2t_3^2t_4a,\, t_1t_2^2t_3^2t_4(b-ac),\, 
    t_1t_2t_3t_4c,\,   t_1t_2t_3^2t_4c,\,   t_1t_2t_3t_4b,\, t_2t_3t_4a,\,   2)\\
    v_3&=(0,\,   0,\,   0,\,   0,\,  t_1t_2t_3^2t_4c,\,t_1t_2t_3t_4b,\,t_2t_3t_4a,\, 2)\\
v_4&=(t_1t_2t_3^2t_4c,\,   -t_1t_2t_3b,\,   t_2t_3a,\, -2,\,   0,\,   0,\,   0,\,   0).
\end{align*}
Then $f(v_i)=0$ for $1\leq i\leq 4$ and $(v_i,v_j)=0$ for $1\leq i,j\leq 4$. Moreover $v_1,v_2,v_3,v_4$ are linearly independent (using the assumption $q\neq 0$). Recalling that $\rho(\theta)$ acts on (row) vectors on the right, we note that $v_1$ and $v_2$ are $1$-eigenvectors of~$\theta$. Thus the points $P_i=\langle v_i\rangle$ for $1\leq i\leq 4$ generate a $3$-space stabilised by $\theta$ (with $P_1$ and $P_2$ generating a line full of fixed points). 

We now extend $v_1,v_2,v_3,v_4$ to a standard basis $v_1,\ldots,v_8$. There is considerable choice in doing this, however fixing the shape
\begin{align*}
v_5&=(0,0,*,*,0,0,*,*)\\
v_6&=(0,0,*,*,0,0,*,*)\\
v_7&=(0,0,0,0,0,0,*,*)\\
v_8&=(0,0,*,*,0,0,0,0)
\end{align*}
makes the choice unique (here the condition $r\neq 0$ is required to ensure the normalisation condition $(v_i,v_{8-i+1})=1$). Thus, assuming $q,r\neq 0$, the chamber 
\begin{align*}
\langle v_1\rangle\leq \langle v_1,v_2\rangle\leq \begin{aligned}&\langle v_1,v_2,v_3,v_4\rangle\\
&\langle v_1,v_2,v_3,v_5\rangle\end{aligned}
\end{align*}
of the $\sD_4$ oriflamme complex is fixed by $\theta$. Writing $g$ for the matrix with rows $v_1,\ldots,v_8$, the element $g\theta g^{-1}$ is necessarily in $B$. Calculation shows that $g\theta g^{-1}\in U_{A}$ with $A$ as in Lemma~\ref{lem:conjugate}. Carrying through the conjugations outlined in Lemma~\ref{lem:conjugate}, it turns out that $\theta$ is conjugate to
$$
x_{1111}(1)x_{0100}(-t_1c^2)x_{1110}(1)x_{1101}(1)x_{0111}(1),
$$
and element of the form (2).

\noindent\textit{Subcase: $q=0$.} Suppose that $q=0$. Then $t_2=t_0^2$ for some $t_0\in\K$, and $b=t_0^{-1}c$. Equation~(\ref{eq:mainrelation}) implies that either $a=2t_0^{-1}$ of $a=t_0^{-1}(t_1c^2-2)$, and we consider these cases in turn. If $a=2t_0^{-1}$ we define a standard basis by
\begin{align*}
v_1&=(-t_0^4t_1t_3^3t_4^2,\,t_0^3t_1t_3^2t_4,\,0,0,0,0,t_0t_3t_4,\,1)&
v_2&=(0,0,-t_0^2t_3,\,t_0,\,t_0t_3,\,1,0,0)\\
v_3&=(0,0,0,0,t_0t_3,\,1,0,0)&
v_4&=(0,0,0,0,0,0,t_0t_3t_4,\,1)\\
v_5&=(1,0,0,0,0,0,-t_0^{-3}t_1^{-1}t_3^{-2}t_4^{-1},\,0)&
v_6&=(0,0,1,0,-t_0^{-1},0,0,0)\\
v_7&=(0,0,0,0,t_0^{-1},0,0,0)&
v_8&=(0,0,0,0,0,0,t_0^{-3}t_1^{-1}t_3^{-2}t_4^{-1},0).
\end{align*}
Note that $v_1$ and $v_2$ are $1$-eigenvectors, and again we have a fixed chamber. Writing $g$ for the matrix with rows $v_1,\ldots,v_8$ we have (by calculation)
$$
g\theta g^{-1}=x_{0101}(-ct_0^3t_3^2t_4)x_{0111}(t_0^{-2}t_3^{-1})x_{1101}(t_0^{-4}t_1^{-1}t_3^{-3}t_4^{-2})x_{1211}(ct_0^{-3}t_3^{-2}t_4^{-1})x_{1110}(-1)x_{0100}(1).
$$
Carrying through the conjugations outlined in Lemma~\ref{lem:conjugate} we see that $\theta$ is conjugate to 
$$
x_{1111}(1)x_{0100}(-t_1c^2)x_{1110}(1)x_{1101}(1)x_{0111}(1),
$$
again an element of the form (2) (this time note that $t_1c^2=4$ is permitted).

Now suppose that $a=t_0^{-1}(t_1c^2-2)$. We take $v_1$ and $v_2$ as in the $q,r\neq 0$ case (these two fixed vectors remain linearly independent). If $r=0$ then $a=2t_0^{-1}$, and this case was dealt with in the above paragraph. Similarly if $c=0$ then $a=-2t_0^{-1}$, which is also dealt with above (replacing $t_0$ by $-t_0$). Thus we assume $r\neq 0$ and $c\neq 0$. Let
\begin{align*}
v_3&=(-t_0t_1t_3t_4c,0,1,t_0^{-1}t_3^{-1},0,0,0,0)&
v_4&=(-t_0t_3t_4(t_1c^2-1),1,0,t_0^{-1}t_3^{-1}c,0,0,0,0)\\
v_5&=(-t_0^3t_1t_3^2t_4(t_1c^2-1),0,t_0^2t_1t_3c,0,0,0,1,0)&
v_6&=(t_0^3t_1t_3^2t_4c(t_1c^2-2),-t_0^2t_1t_3c,0,-t_0,0,1,0,0)\\
v_7&=(-2r^{-1},0,0,t_0^{-2}t_3^{-2}t_4^{-1}cr^{-1},0,0,0,0)&
v_8&=(t_1cr^{-1},0,0,-2t_0^{-2}t_3^{-2}t_4^{-1}r^{-1},0,0,0,0).
\end{align*}
Then $v_1,\ldots,v_8$ is a standard basis, and hence we again have a fixed chamber. Writing $g$ for the matrix with rows $v_1,\ldots,v_8$ calculation shows that $g\theta g^{-1}\in U_A$, and carrying out the conjugations outlined in Lemma~\ref{lem:conjugate} gives
\begin{align*}
\theta\sim x_{1111}(1)x_{0100}(t_1c^2-4)x_{1110}(1)x_{1101}(1)x_{0111}(1),
\end{align*}
which is again of the form (2). 

\noindent\textit{Subcase: $r=0$.} If $r=0$ then $t_1=t_0'^2$ for some $t_0'\in\K$, and $c=2t_0'^{-1}$. Equation~(\ref{eq:mainrelation}) then gives $b=t_0'^{-1}a$. The following gives a standard basis
\begin{align*}
v_1&=(-t_0't_2t_3^2t_4,0,0,1,t_3,a,t_0'^{-1}a,t_0'^{-1}t_2^{-1}t_3^{-1}t_4^{-1})&
v_2&=(0,t_0',-1,0,0,-t_2^{-1}t_3^{-1},-t_0'^{-1}t_2^{-1}t_3^{-1},0)\\
v_3&=(0,0,0,0,t_0't_2t_3^2t_4,0,0,1)&
v_4&=(0,0,0,0,0,t_0',1,0)\\
v_5&=(0,1,0,0,-t_0'^{-1}a,-t_0'^{-1}t_2^{-1}t_3^{-1},0,0)&
v_6&=(1,0,0,0,-t_0'^{-1}t_2^{-1}t_3^{-1}t_4^{-1},0,0,0)\\
v_7&=(0,0,0,0,0,-1,0,0)&
v_8&=(0,0,0,0,1,0,0,0),
\end{align*}
with $v_1$ and $v_2$ fixed vectors. Writing $g$ for the matrix with rows $v_1,\ldots,v_8$ we again have $g\theta g^{-1}\in U_A$. If $a=0$ then $\theta\in U_{A\backslash\{\alpha_2\}}$, specifically
$$
g\theta g^{-1}=x_{0101}(t_0'^{-1})x_{1111}(-t_0'^{-1}t_2^{-1}t_3^{-2}t_4^{-1})x_{0110}(-t_0't_2t_3)x_{1100}(-t_0't_2t_3t_4).
$$
By Lemma~\ref{lem:4perpD4} this is conjugate to an element of the form (1).

If $a\neq 0$ then $g\theta g^{-1}\notin U_{A\backslash\{\alpha_2\}}$, and carrying through the first conjugation described in Lemma~\ref{lem:conjugate} we see that $\theta$ is conjugate to
\begin{align*}
\theta'&=x_{1111}(-(t_2a^2-2)a^{-2}t_0'^{-1}t_2^{-2}t_3^{-2}t_4^{-1})x_{0100}(at_0't_2^2t_3^2t_4)x_{1110}(-t_0'a^{-1})\\
&\qquad\times x_{1101}(t_0'^{-1}t_2^{-1}t_3^{-1}a^{-1})x_{0111}(t_0'^{-1}t_2^{-1}t_3^{-1}t_4^{-1}a^{-1})x_{1211}(-(t_2a^2-1)a^{-1}).
\end{align*}
If $t_2a^2=2$ then
\begin{align*}
\theta'&=x_{0100}(at_0't_2^2t_3^2t_4)x_{1110}(-t_0'a^{-1})x_{1101}(t_0'^{-1}t_2^{-1}t_3^{-1}a^{-1})x_{0111}(t_0'^{-1}t_2^{-1}t_3^{-1}t_4^{-1}a^{-1})x_{1211}(-a^{-1}).
\end{align*}
One can now conjugate by an element $x_{1100}(*)$ to remove the $x_{1211}(*)$ term, giving a product of $4$ mutually perpendicular root elations, which is hence of form (1) by Lemma~\ref{lem:4perpD4}. 

If $t_2a^2\neq 2$ then carrying through the second conjugation described in Lemma~\ref{lem:conjugate} we see that $\theta$ is conjugate to 
\begin{align*}
x_{1111}(1)x_{0100}(-(t_2a^2-2)^2)x_{1110}(1)x_{1101}(1)x_{0111}(1).
\end{align*}

\smallskip

\noindent\textbf{Case: $z= -1$.} We will assume $\kar(\K)\neq 2$ (otherwise we are in the previous case). If $z=-1$ (that is, $z=z^{-1}$, as $z\neq 1$) then the equation $p(-1)=0$ gives 
$$
t_2a^2+t_1t_2b^2+t_1c^2-t_1t_2abc=0.
$$
Let $s=t_2b^2+c^2$. Define isotropic vectors
\begin{align*}
v_1&=(-t_1t_2t_3^2t_4c,t_1t_2t_3b,-t_2t_3a,0,0,a,ac-b,t_2^{-1}t_3^{-1}t_4^{-1}c)\\
v_2&=(0,t_3^{-1}t_4^{-1}a,t_3^{-1}t_4^{-1}(b-ac),t_2^{-1}t_3^{-2}t_4^{-1}c,-t_2^{-1}t_3^{-1}t_4^{-1}c,-t_2^{-1}t_3^{-2}t_4^{-1}b,-t_1^{-1}t_2^{-1}t_3^{-2}t_4^{-1}a,0)\\
v_3&=(t_3t_4sc,-sb,t_1^{-1}sa,0,0,t_1^{-1}t_2^{-1}t_3^{-1}sa,t_1^{-1}t_2^{-1}t_3^{-1}s(ac-b),t_1^{-1}t_2^{-2}t_3^{-2}t_4^{-1}sc)\\
v_4&=(0,s,-t_1^{-1}t_2ab-c^3,t_3^{-1}bc^2-t_1^{-1}t_3^{-1}ac,bc^2-t_1^{-1}ac,t_3^{-1}b^2c-t_1^{-1}t_3^{-1}ab,t_1^{-1}t_2^{-1}t_3^{-1}s,0).
\end{align*}
Then $v_1$ and $v_2$ are $1$-eigenvectors of $\theta$, and $v_3,v_4$ are $(-1)$-eigenvectors of $\theta$. Assuming $s\neq 0$ and $c\neq 0$, there is a unique extension to a standard basis $v_1,\ldots,v_8$ of the form
\begin{align*}
v_5&=(0,0,0,*,0,0,*,*)&v_6&=(0,0,*,*,0,0,0,*)\\
v_7&=(0,0,0,*,0,0,*,*)&v_8&=(0,0,*,*,0,0,0,*).
\end{align*}
Writing $g$ for the matrix with rows $v_1,\ldots,v_8$ we see that $g\theta g^{-1}$ is of the form
\begin{align*}
g\theta g^{-1}=x_{0001}(*)x_{0101}(*)x_{0111}(*)x_{1101}(*)x_{1111}(*)h_{0001}(-1). 
\end{align*}
Conjugating by an appropriate element $x_{1000}(*)x_{0010}(*)$ we see that $\theta$ is conjugate to
\begin{align*}
x_{0001}(t_1t_2^2t_3ac^{-1}s^{-2}/2)x_{0101}(-t_2^2t_3^2t_4ac^{-2}s^{-1}/2)x_{1111}(t_1^{-1}t_2^2(t_1bc-a)^2s^{-3}/2)h_{0001}(-1)
\end{align*}
Conjugating by $x_{0101}(t_2^2t_3^2t_4ac^{-2}s^{-1}/4)x_{1111}(-t_1^{-1}t_2^2(t_1bc-a)^2s^{-3}/4)$ we see that $\theta$ is conjugate to 
\begin{align*}
\theta'=x_{0001}(t_1t_2^2t_3ac^{-1}s^{-2}/2)h_{0001}(-1)
\end{align*}
If $a\neq 0$ then conjugating by a suitable diagonal matrix we see that $\theta$ is conjugate to the element
$
x_{0001}(1)h_{0001}(-1)
$ (which is conjugate to the element of form (4)), and if $a=0$ then $\theta$ is conjugate to $h_{0001}(-1)$ (which is of form (3)). The cases where either $c=0$ or $s=0$ require minor modifications (analogous to the $z=1$ case), however no new conjugacy classes for $\theta$ are obtained.
\smallskip

\noindent\textbf{Case: $z\neq \pm1$.} Then $z\neq z^{-1}$. We have $\dim(\ker(\theta-I))=4$, and the vectors 
\begin{align*}
v_1&=(-t_1t_2^2t_3^3t_4^2c,t_1t_2^2t_3^2t_4b,-t_2^2t_3^2t_4a,t_2t_3t_4y,t_2t_3^2t_4y',t_2t_3t_4a,t_2t_3t_4(ac-b),c)\\
v_2&=(-t_1t_2^2t_3^3t_4^2c,t_1t_2^2t_3^2t_4b,-t_2^2t_3^2t_4a,t_2t_3t_4y',t_2t_3^2t_4y,t_2t_3t_4a,t_2t_3t_4(ac-b),c)\\
v_3&=(-t_1t_2^2t_3^3t_4^2y,t_1t_2^2t_3^2t_4a,t_1t_2^2t_3^2t_4(b-ac),t_1t_2t_3t_4c,ct_1t_2t_3^{2}t_4z,t_1t_2t_3t_4bz,t_2t_3t_4az,y)\\
v_4&=(-t_1t_2^2t_3^3t_4^2y',t_1t_2^2t_3^2t_4a,t_1t_2^2t_3^2t_4(b-ac),t_1t_2t_3t_4c,ct_1t_2t_3^{2}t_4z^{-1},t_1t_2t_3t_4bz^{-1},t_2t_3t_4az^{-1},y')
\end{align*}
(with $y=z+1$ and $y'=z^{-1}+1$) are linearly independent isotropic elements of $\ker(\theta-I)$, with $(v_1,v_4)=(v_2,v_3)=0$ (that is, $\ker(\theta-I)$ contains two lines full of fixed points). Moreover, $\ker(\theta-zI)$ and $\ker(\theta-z^{-1}I)$ are both $2$-dimensional, each consisting precisely of a line of the polar space, and if $v\in\ker(\theta-I)$ and $v'\in \ker(\theta-zI)\cup\ker(\theta-z^{-1}I)$ then $(v,v')=0$. 

Thus one can then choose (necessarily isotropic) vectors $v_1'\in\langle v_1,v_4\rangle$, $v_2'\in\ker(\theta-z^{-1}I)$, $v_3'\in\langle v_2,v_3\rangle$, $v_4'\in\ker(\theta-zI)$, $v_5'\in\ker(\theta-z^{-1}I)$, $v_6'\in\langle v_1,v_4\rangle$, $v_7'\in\ker(\theta-zI)$, $v_8'\in\langle v_2,v_3\rangle$ such that $v_1',\ldots,v_8'$ forms a standard basis with respect to the bilinear form. It follows that $\theta$ is conjugate in $\sD_4$ to the matrix
$$
\mathrm{diag}(1,z^{-1},1,z,z^{-1},1,z,1)
$$
and this is in turn conjugate to $h_{\alpha^{\vee}}(z)$ for any positive root $\alpha$, and hence is of form (3). 
\smallskip

\noindent\textbf{Converse:} 
We must now show that all elements of the forms (1), (2), (3), and (4) have opposition diagram~$\sE_{7;4}$. Reversing the analysis above, it follows that each such element is conjugate to an element of the form $\theta=uhw_{\sD_4}$ that we started with, and by Theorem~\ref{thm:E74chamberfixing} such $\theta$ is conjugate to an element of $Bs_{\varphi_1}s_{\varphi_2}s_{\varphi_3}s_{\varphi_4}B=Bs_2s_5s_7w_0B$, and hence this conjugate maps the base chamber $B$ to Weyl distance $s_2s_5s_7w_0$. This shows that $\theta$ maps some type $\{1,3,4,6\}$ simplex to an opposite, and hence $\theta$ either has diagram $\sE_{7;4}$, or $\theta$ is not domestic. Thus it is sufficient to prove that the elements of the form (1), (2), (3), and (4) are domestic.

Consider the element $\theta=x_{0100}(a)x_{1110}(1)x_{1101}(1)x_{0111}(1)$ of form (1). By Lemma~\ref{lem:4perpD4} this element is conjugate to an element of the form $\theta'=x_{1211}(a')x_{1000}(1)x_{0010}(1)x_{0001}(1)$ with $a'\neq 0$. Then, with $\su$ as in Lemma~\ref{lem:magicelement} we have
$
\su \theta'\su^{-1}=x_{\varphi_1}(a')x_{\varphi_2}(1)x_{\varphi_3}(1)x_{\varphi_4}(1),
$
with $\varphi_1=\varphi_{\sE_7}$, $\varphi_2=\varphi_{\sD_6}$, $\varphi_3=\varphi_{\sD_4}$, and $\varphi_4=\alpha_3$, which has opposition diagram $\sE_{7;4}$ by \cite[Theorem~3.1]{PVMexc}.

Consider the element $\theta=x_{1111}(1)x_{0100}(a)x_{1110}(1)x_{1101}(1)x_{0111}(1)$ with $a\neq 0$. Then
$$
\su \theta\su^{-1}=x_{1234321}(\pm 1)x_{1000000}(\pm a)x_{1224321}(\pm 1)x_{1122221}(\pm 1)x_{1122100}(\pm 1), 
$$
and thus by \cite[Lemma~3.4(2)(c)]{PVMexc} we have $\disp(\theta)\leq 60$. 

The element $h_{\varphi^{\vee}}(c)$, with $c\neq 0,1$, is proved to be domestic (with diagram $\sE_{7;4}$) in \cite[Lemma~4.5]{PVMexc} (note that $\omega_1=\varphi^{\vee}$). 

Finally, it remains to prove that elements of the form $x_{\varphi}(1)h_{\varphi^{\vee}}(-1)$, with $\kar(\K)\neq 2$, are domestic. Such elements have not arisen explicitly in our previous work, and so we will give a general analysis below, which completes the proof of the theorem.
\end{proof}

As noted above, to complete the proof of Theorem~\ref{thm:E74chamberfixingmain} we must show that the automorphism of $\sE_7(\K)$ given by $\theta=x_{\varphi}(1)h_{\varphi^{\vee}}(-1)$ has opposition diagram $\sE_{7;4}$ (for $\kar(\K)\neq 2$). We give an analysis of automorphisms of this form in general type, for later reference.

\begin{prop}\label{prop:extraauto}
Let $\Delta$ be the split building of a Chevalley group of irreducible type~$\Phi$ over a field $\K$ of characteristic not~$2$. Let $\theta=x_{\varphi}(1)h_{\varphi^{\vee}}(-1)$ and $\theta'=x_{\varphi'}(1)h_{\varphi'^{\vee}}(-1)$, with $\varphi$ and $\varphi'$ the highest root and highest short root of $\Phi$ (with $\theta'=\theta$ in the simply laced case).
\begin{compactenum}[$(1)$]
\item If $\Phi=\sA_n$ with $n\geq 5$ then $\theta$ is domestic with opposition digram ${^2}\sA_{n;2}$.
\item If $\Phi=\sB_n$ then $\theta$ (respectively $\theta'$) is domestic for $n\geq 5$ (respectively $n\geq 3$) with opposition diagram $\sB_{n;4}^1$ (respectively $\sB_{n;2}^1$). 
\item If $\Phi=\sC_n$ then $\theta$ (respectively $\theta'$) is domestic for $n\geq 3$ (respectively $n\geq 5$) with opposition diagram $\sC_{n;2}^1$ (respectively $\sC_{n;4}^1$). 
\item If $\Phi=\sD_n$ with $n\geq 6$ then $\theta$ is domestic with opposition diagram $\sD_{n;4}^1$. 
\item If $\Phi=\sE_n$ with $n=7,8$ then $\theta$ is domestic with opposition diagram $\sE_{n;4}$.
\end{compactenum}
In all other cases, neither $\theta$ nor $\theta'$ are domestic. 
\end{prop}

\begin{proof}
The $\sA_n$ case follows from \cite[Theorem~1.5(1)]{PVMclass} and the statement for $\theta$ for type $\sC_n$ follows from \cite[Theorem~2(2)(b)]{PVMclass}. The remaining statements for classical types $\sB_n$, $\sC_n$ and $\sD_n$ are easily obtained using the explicit matrix realisations of these groups and making calculations analogous to those in \cite[Proposition~4.4]{PVMclass}, and we omit the details. 

Domestic collineations of split buildings of types $\sF_4$, $\sE_6$, and $\sG_2$ are classified in~\cite[Theorems~7, 8, and~9]{PVMexc} and it follows that neither $\theta$ nor $\theta'$ are  domestic in these cases. Collineations of $\sE_7$ and $\sE_8$ buildings with opposition diagrams $\sE_{7;1}$, $\sE_{7;2}$, $\sE_{8;1}$, or $\sE_{8;2}$ are classified in \cite{PVMexc}, and combining this with the classification of collineations of $\sE_7$ buildings with diagram $\sE_{7;3}$ fixing a chamber from Theorem~\ref{thm:E73Classification} and the classification of admissible diagrams we see that if $\theta$ is domestic in an $\sE_n$ building $(n=7,8)$ then it necessarily has opposition diagram $\sE_{n;4}$.

It remains to prove that $\theta$ is domestic in types $\sE_7$ and $\sE_8$. Thus suppose that $\Phi=\sE_7$ or $\Phi=\sE_8$. Let $\wp=\{p\}$ be the polar type (so $p=1$ for $\sE_7$, and $p=8$ for $\sE_8$). By \cite[VI, \S1.8]{Bou:02} we have
$\langle\varphi^{\vee},\alpha\rangle=0,1,2$ for $\alpha\in\Phi_{S\backslash \wp}^+$, $\alpha\in\Phi^+\backslash (\Phi_{S\backslash \wp}\cup\{\varphi\})$, and $\alpha=\varphi$ (respectively),
and it follows that the coefficient of $\alpha_p$ in any positive root $\alpha$ is either $0,1$, or $2$ (respectively). Thus if $\alpha,\beta\in\Phi^+\backslash \Phi_{S\backslash \wp}$ and $\alpha+\beta$ is a root then $\alpha+\beta=\varphi$, and so $x_{\alpha}(a)x_{\beta}(b)=x_{\beta}(b)x_{\alpha}(a)$, or $x_{\alpha}(a)x_{\beta}(b)=x_{\beta}(b)x_{\alpha}(a)x_{\varphi}(\pm ab)$. 

Consider $gB=uw_0B$ with $u\in U^+$. Write $u=u_1u_2$ with $u_1\in U_{\Phi_{S\backslash \wp}}^+$ and $u_2\in U_{\Phi\backslash \Phi_{S\backslash \wp}}^+$. Since $x_{\varphi}(1)$ is central in $U^+$, and since $h_{\varphi^{\vee}}(-1)x_{\alpha}(a)h_{\varphi^{\vee}}(-1)=x_{\alpha}((-1)^{\langle\varphi^{\vee},\alpha\rangle}a)$ we see that $\theta$ and $u_1$ commute. It follows from the above observations that
\begin{align*}
g^{-1}\theta g&=w_0^{-1}u_2^{-1}x_{\varphi}(1)h_{\varphi^{\vee}}(-1)u_2w_0\in BU_{w_0\beta_1}\cdots U_{w_0\beta_k}B,
\end{align*}
where $\beta_1,\ldots,\beta_k$ are the roots of $\Phi^+\backslash \Phi_{S\backslash\wp}$. By  \cite[Proposition~1.8]{PVMexc}, if there exists $w_1\in W$ with $w_1w_0\beta_j>0$ for all $j=1,\ldots,k$, then $\disp(\theta)\leq 2\ell(w_1)-1$. In particular, taking $w_1=w_{S\backslash\wp}w_0$ (note that $\Phi(w_{S\backslash \wp})=\Phi_{S\backslash \wp}^+$) we have 
$$
\disp(\theta)\leq 2\ell(w_0)-2\ell(w_{S\backslash\wp})-1.
$$
If $\Phi=\sE_8$ then $\ell(w_0)=120$ and $\ell(w_{S\backslash\wp})=63$, and so $\disp(\theta)\leq 113<120$, and so $\theta$ is domestic as required. If $\Phi=\sE_7$ then the inequality only yields $\disp(\theta)\leq 65$, which is not sufficient to prove domesticity. However in this case we apply Case (2)(c) of \cite[Lemma~3.4]{PVMexc}, showing that $\theta$ is domestic. 
\end{proof}

\begin{remark}
The element $\theta=x_{\alpha}(1)h_{\alpha^{\vee}}(-1)$ is not conjugate to an element of $H$ or $U^+$. To see this, consider the $\mathsf{SL}_2(\K)$ generated by $U_{\alpha}$ and $U_{-\alpha}$. The matrix of $\theta$ is $-x_{\varphi}(1)$ which shows that it is neither unipotent (eigenvalues distinct from $1$) nor a homology (as it is not diagonalisable). Thus in sufficiently high rank this automorphism gives a simple example of a chamber fixing domestic automorphism that is neither unipotent nor a homology. In contrast, note that if $c\neq 0,-1$ then $x_{\alpha}(1)h_{\alpha^{\vee}}(c)$ is conjugate to a homology as it is diagaonlisable in~$\mathsf{SL}_2(\K)$. 
\end{remark}

\begin{remark}\label{rem:perp}
For $1\leq k\leq 8$ let $\mathcal{P}_k$ denote the set of all sets of three mutually perpendicular positive roots of $\sE_7$. The Weyl group $W$ acts on $\mathcal{P}_k$ in the natural way, and it is not hard to see that for $k=1,2,3,4$ there are $1,1,2,4$ orbits for this action (respectively). Indeed, for any simply laced irreducible diagram, representatives for the action of $W$ on $\mathcal{P}_k$ can be found by running the following algorithm up to step~$k$: At each step, choose a connected component of the diagram, and modify the diagram by removing the nodes of the polar type of this component. The set of highest roots associated to the sequence of connected components chosen is a set of $k$ mutually perpendicular positive roots, and running the algorithm in all possible ways generates a set of representatives for the action of $W$ on $\mathcal{P}_k$. For example, for $\sE_{7;4}$ the $4$ representatives are obtained by:
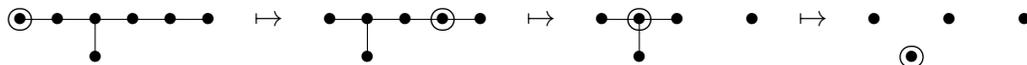
\begin{align*}
&\begin{tikzpicture}[scale=0.5,baseline=-0.5ex]
\node at (0,0.8) {};
\node at (0,-0.8) {};
\node [inner sep=0.8pt,outer sep=0.8pt] at (-2,0) (1) {$\bullet$};
\node [inner sep=0.8pt,outer sep=0.8pt] at (-1,0) (3) {$\bullet$};
\node [inner sep=0.8pt,outer sep=0.8pt] at (0,0) (4) {$\bullet$};
\node [inner sep=0.8pt,outer sep=0.8pt] at (1,0) (5) {$\bullet$};
\node [inner sep=0.8pt,outer sep=0.8pt] at (2,0) (6) {$\bullet$};
\node [inner sep=0.8pt,outer sep=0.8pt] at (3,0) (7) {$\bullet$};
\node [inner sep=0.8pt,outer sep=0.8pt] at (0,-1) (2) {$\bullet$};
\draw (-2,0)--(3,0);
\draw (0,0)--(0,-1);
\draw [line width=0.5pt,line cap=round,rounded corners] (1.north west)  rectangle (1.south east);
\end{tikzpicture}\quad\mapsto\quad
\begin{tikzpicture}[scale=0.5,baseline=-0.5ex]
\node [inner sep=0.8pt,outer sep=0.8pt] at (-1,0) (3) {$\bullet$};
\node [inner sep=0.8pt,outer sep=0.8pt] at (0,0) (4) {$\bullet$};
\node [inner sep=0.8pt,outer sep=0.8pt] at (1,0) (5) {$\bullet$};
\node [inner sep=0.8pt,outer sep=0.8pt] at (2,0) (6) {$\bullet$};
\node [inner sep=0.8pt,outer sep=0.8pt] at (3,0) (7) {$\bullet$};
\node [inner sep=0.8pt,outer sep=0.8pt] at (0,-1) (2) {$\bullet$};
\draw (-1,0)--(3,0);
\draw (0,0)--(0,-1);
\draw [line width=0.5pt,line cap=round,rounded corners] (6.north west)  rectangle (6.south east);
\end{tikzpicture}\quad\mapsto\quad
\begin{tikzpicture}[scale=0.5,baseline=-0.5ex]
\node [inner sep=0.8pt,outer sep=0.8pt] at (-1,0) (3) {$\bullet$};
\node [inner sep=0.8pt,outer sep=0.8pt] at (0,0) (4) {$\bullet$};
\node [inner sep=0.8pt,outer sep=0.8pt] at (1,0) (5) {$\bullet$};
\node [inner sep=0.8pt,outer sep=0.8pt] at (3,0) (7) {$\bullet$};
\node [inner sep=0.8pt,outer sep=0.8pt] at (0,-1) (2) {$\bullet$};
\draw (-1,0)--(1,0);
\draw (0,0)--(0,-1);
\draw [line width=0.5pt,line cap=round,rounded corners] (4.north west)  rectangle (4.south east);
\end{tikzpicture}\quad\mapsto\quad
\begin{tikzpicture}[scale=0.5,baseline=-0.5ex]
\node [inner sep=0.8pt,outer sep=0.8pt] at (-1,0) (3) {$\bullet$};
\node [inner sep=0.8pt,outer sep=0.8pt] at (1,0) (5) {$\bullet$};
\node [inner sep=0.8pt,outer sep=0.8pt] at (3,0) (7) {$\bullet$};
\node [inner sep=0.8pt,outer sep=0.8pt] at (0,-1) (2) {$\bullet$};
\draw [line width=0.5pt,line cap=round,rounded corners] (2.north west)  rectangle (2.south east);
\end{tikzpicture}
\end{align*}
along with the variations with any one of the other nodes encircled at the fourth stage. Thus the representatives for the action of $W$ on $\mathcal{P}_4$ are $\{\varphi_{\sE_7},\varphi_{\sD_6},\varphi_{\sD_4},\alpha_2\}$, $\{\varphi_{\sE_7},\varphi_{\sD_6},\varphi_{\sD_4},\alpha_3\}$, $\{\varphi_{\sE_7},\varphi_{\sD_6},\varphi_{\sD_4},\alpha_5\}$, and $\{\varphi_{\sE_7},\varphi_{\sD_6},\varphi_{\sD_4},\alpha_7\}$ (note that one could also run the algorithm by picking the $\sA_1$ component at the third stage, and then the $\sD_4$ component is forced at the fourth stage, but this results in the set $\{\varphi_{\sE_7},\varphi_{\sD_6},\alpha_7,\varphi_{\sD_4}\}$ which was already found). 

Similarly, representatives for the action of $W$ on $\mathcal{P}_3$ are $\{\varphi_{\sE_7},\varphi_{\sD_6},\alpha_7\}$ and $\{\varphi_{\sE_7},\varphi_{\sD_6},\varphi_{\sD_4}\}$. By Theorem~\ref{thm:E73Classification}, only the first orbit gives rise to automorphisms with diagram $\sE_{7;3}$. For the second orbit, elements of the form $x_{\varphi_{\sE_7}}(a)x_{\varphi_{\sD_6}}(b)x_{\varphi_{\sD_4}}(c)$ with $a,b,c\neq 0$ are conjugate to $x_{1110}(1)x_{1101}(1)x_{0111}(1)$, and it turns out that this element is conjugate to $$x_{1111}(1)x_{0100}(-4)x_{1110}(1)x_{1101}(1)x_{0111}(1).$$ Thus, by Theorem~\ref{thm:E74chamberfixing}, the automorphisms  $\theta=x_{\varphi_{\sE_7}}(a)x_{\varphi_{\sD_6}}(b)x_{\varphi_{\sD_4}}(c)$ with $a,b,c\neq 0$ have opposition diagram $\sE_{7;4}$. Furthermore, by Theorem~\ref{thm:E74chamberfixing} we see that only one orbit of the action of $W$ on $\mathcal{P}_4$ gives rise to automorphisms with diagram $\sE_{7;4}$, and it turns out that for all other orbits the associated automorphisms are not domestic. 
\end{remark}

\end{document}